\pdfoutput=1
\documentclass[11pt]{article}

\usepackage[margin=1in]{geometry}
\usepackage{amsmath,amssymb,amsthm,mathtools}
\usepackage{graphicx}
\usepackage{enumitem}
\usepackage[colorlinks=true,linkcolor=blue,citecolor=blue,urlcolor=blue]{hyperref}
\usepackage{aliascnt}
\usepackage{cleveref}

\numberwithin{equation}{section}

\theoremstyle{plain}
\newtheorem{proposition}{Proposition}[section]
\newaliascnt{lemma}{proposition}
\newtheorem{lemma}[lemma]{Lemma}
\aliascntresetthe{lemma}
\theoremstyle{definition}
\newaliascnt{definition}{proposition}
\newtheorem{definition}[definition]{Definition}
\aliascntresetthe{definition}
\newaliascnt{remark}{proposition}
\newtheorem{remark}[remark]{Remark}
\aliascntresetthe{remark}
\theoremstyle{remark}
\crefname{remark}{Remark}{Remarks}
\crefname{definition}{Definition}{Definitions}
\crefname{proposition}{Proposition}{Propositions}
\crefname{lemma}{Lemma}{Lemmas}

\usepackage{tikz}
\usetikzlibrary{arrows.meta,positioning,calc,decorations.pathreplacing,decorations.markings}
\definecolor{upcol}{rgb}{0.10,0.28,0.80}
\definecolor{downcol}{rgb}{0.84,0.18,0.18}
\tikzset{
  uppath/.style={upcol,line width=1.1pt},
  downpath/.style={downcol,line width=1.1pt},
  gridln/.style={gray!45,line width=0.4pt},
  dgridln/.style={gray!55,line width=0.35pt,densely dotted},
  vtx/.style={circle,fill=black,inner sep=0pt,minimum size=2.0pt},
  ybeflow/.style={postaction={decorate},
    decoration={markings,mark=at position 0.55 with {\arrow{Stealth[length=2mm]}}}},
  ybeflowb/.style={postaction={decorate},
    decoration={markings,mark=at position 0.74 with {\arrow{Stealth[length=2mm]}}}},
}
\newcommand{\fvv}[1]{\begin{tikzpicture}[scale=0.7,baseline=0pt,>={Stealth[length=1.8mm]},rounded corners=1pt]
  \draw[gridln] (0,-0.6)--(0,0.6) (-0.6,0)--(0.6,0);
  #1
\end{tikzpicture}}
\newcommand{\hcross}[4]{\begin{tikzpicture}[scale=0.55,baseline=-2.5pt]
  \ifnum#1=1 \draw[uppath,-{Stealth[length=1.5mm]}] (-0.8,-0.8)--(-0.06,-0.06);
  \else \draw[dgridln] (-0.8,-0.8)--(0,0); \fi
  \ifnum#3=1 \draw[uppath,-{Stealth[length=1.5mm]}] (0,0)--(0.8,0.8);
  \else \draw[dgridln] (0,0)--(0.8,0.8); \fi
  \ifnum#2=1 \draw[downpath,-{Stealth[length=1.5mm]}] (-0.8,0.8)--(-0.06,0.06);
  \else \draw[dgridln] (-0.8,0.8)--(0,0); \fi
  \ifnum#4=1 \draw[downpath,-{Stealth[length=1.5mm]}] (0,0)--(0.8,-0.8);
  \else \draw[dgridln] (0,0)--(0.8,-0.8); \fi
  \node[vtx] at (0,0) {};
\end{tikzpicture}}
\newcommand{\yd}[3][black]{\begin{tikzpicture}[scale=#2]
  \foreach \w [count=\r] in {#3}{\foreach \c in {1,...,\w}{%
    \draw[#1,line width=0.28pt] (\c-1,-\r+1) rectangle (\c,-\r);}}
\end{tikzpicture}}

\title{From Yang--Baxter to Robinson--Schensted--Knuth}
\date{}
\author{Leonid Petrov}

\begin{document}
\maketitle

\begin{abstract}
	We explain how to derive the Robinson--Schensted--Knuth (RSK) correspondence,
	a fundamental bijection in algebraic combinatorics, from the Yang--Baxter equation.
	The Yang--Baxter equation arose in the study of quantum many-body systems and later became
	a cornerstone of the theory of solvable lattice models, particularly vertex models.
	In a vertex model, arrows occupy the edges of a grid, and each vertex
	carries a Boltzmann weight determined by the arrows on the four edges
	meeting at it. The weight of a configuration is the product of these local
	weights, and a partition function is the sum of the weights of all configurations
	with prescribed boundary conditions.
	For a vertex model whose partition functions are the Schur polynomials,
	the two sides of each instance of the Yang--Baxter equation admit exactly one
	weight-preserving matching of their summands. Carried across a grid,
	this forced matching is the classical RSK
	correspondence in the form of Fomin's growth diagrams. 
	In natural coordinates the local matching becomes the
	combinatorial three-dimensional $R$, a set-theoretic
	solution of the Zamolodchikov tetrahedron equation. Its
	periodic closure returns the combinatorial $R$-matrices
	of one-row crystals.
	
	The forced matching is special to the Schur weights. For the
	Hall--Littlewood and $q$-Whittaker deformations and their spin versions,
	at generic parameter values no deterministic matching works for all
	boundary data.
	Reading each instance of the Yang--Baxter equation probabilistically, we
	replace the matching by a coupling of the two sides --- a bijectivization,
	or probabilistic bijection --- and obtain
	Markov operators that
	transport probability measures attached to vertex models.
	Iterated over the grid, these
	operators produce randomized RSK-type dynamics and interacting particle
	systems, including $q$-PushTASEP and the stochastic six-vertex model.
	
	This survey grew out of lectures given at Institut Mittag-Leffler in
	the summer of 2026.
\end{abstract}



\newpage
\setcounter{tocdepth}{2}
\tableofcontents
\setcounter{tocdepth}{4}

\newpage

\section{Introduction}
\label{sec:intro}

\subsection{Overview}
\label{subsec:intro_overview}

The Robinson--Schensted--Knuth (RSK) correspondence and
the Yang--Baxter equation (YBE) come from different
subjects: the first is a bijection of algebraic
combinatorics due to Robinson \cite{robinson1938representations},
Schensted \cite{Schensted1961}, and Knuth \cite{Knuth1970}, the second an
identity of exactly solvable lattice models, see
Baxter \cite{baxter2007exactly} and Reshetikhin
\cite{reshetikhin2010lectures}. These notes explain how the first follows
from the second (this derivation first appeared in
Mucciconi--Petrov \cite[\S5]{MucciconiPetrov2020}).
A \emph{vertex model} assigns a weight
to every vertex of a grid according to the states of the
four edges meeting there; the weight of a configuration
is the product of these local weights, and a
\emph{partition function} --- a generating function over
configurations --- is the sum of the weights of all
configurations with prescribed boundary conditions.
\Cref{fig:intro_model_ybe} (top) shows a configuration of
one such vertex model. Each row of the grid carries its
own variable, called its \emph{spectral parameter}. The
\emph{Yang--Baxter equation} of Yang \cite{YangSystem1967}
and Baxter \cite{Baxter1972} (\Cref{fig:intro_model_ybe},
bottom) equates the partition functions of two
three-vertex configurations that differ by the position
of a cross vertex. Between the two sides of the equation,
the two rows keep their spectral parameters but meet the
vertical line in the opposite order.
Originally, the YBE was discovered in the study of quantum
integrability, which
includes the six-vertex model (exactly solved by Lieb
\cite{Lieb1967SixVertex}), higher-spin models, the
quantum inverse scattering method and algebraic Bethe
ansatz of the Leningrad school (see Faddeev
\cite{Faddeev_Lectures}), and the quantum groups of
Drinfeld \cite{Drinfeld1987QuantumGroups} and Jimbo
\cite{Jimbo1985}, whose $R$-matrices produce the vertex
weights.

\medskip

We consider the YBE for the vertex model whose partition
functions are the \emph{Schur polynomials} (the ordinary
or, more generally, the skew ones), as in Macdonald
\cite{Macdonald1995} and Stanley \cite{Stanley1999}; see
\Cref{sec:schur_models}. The presentation of Schur
polynomials as free fermionic lattice partition functions
is classical, see Zinn-Justin
\cite{ZinnJustin20096Vertex}.\footnote{Its deformation (a
Schur polynomial times a deformed Weyl denominator) first
appeared in Tokuyama \cite{tokuyama1988generating} as a
generating function over strict Gelfand--Tsetlin
patterns, with classical-group analogues in Hamel--King
\cite{hamel2005uturn}. These identities were recast by
Brubaker--Bump--Friedberg \cite{brubaker2011schur} as
solvable six-vertex partition functions obeying the YBE,
connecting the subject to $p$-adic representation
theory.} For fixed boundary conditions,
each of the two sides of the YBE is a sum of monomials in
the two spectral parameters, and within each side these
monomials are pairwise distinct (see
\Cref{fig:intro_ybe_instance}). Therefore the two sides
admit exactly one weight-preserving matching of their
summands. Performed at every cell of a square grid, this
unique matching is the classical RSK correspondence in
Fomin's growth-diagram form.

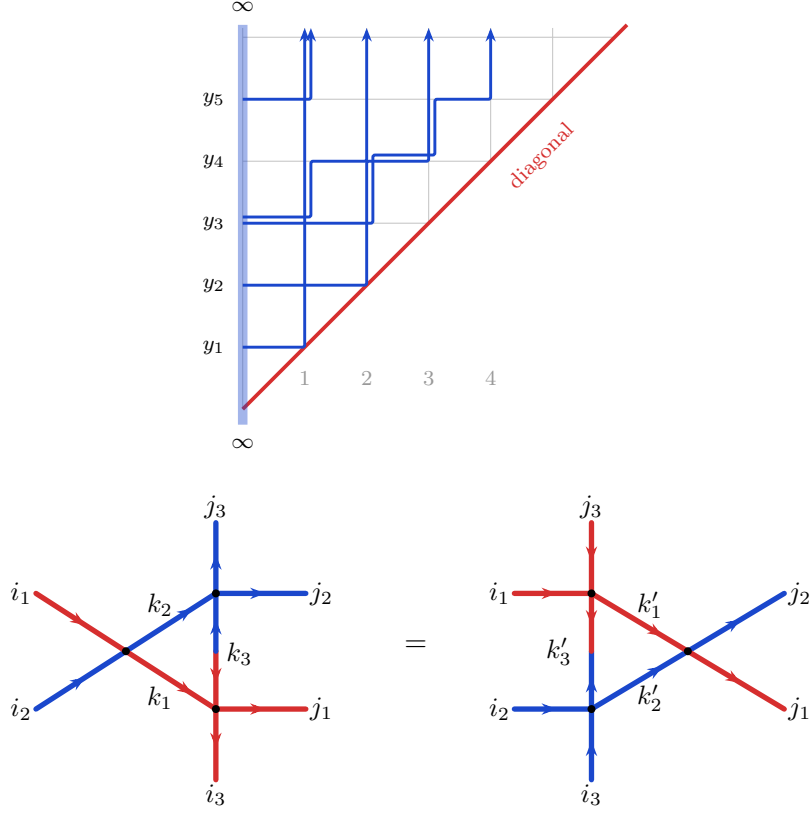
\begin{figure}[ht]
\centering
\begin{tikzpicture}[scale=0.82,>={Stealth[length=1.8mm]},rounded corners=0.8pt,baseline={(current bounding box.center)}]
  \foreach \j in {1,...,6}{\draw[gridln] (0,\j)--(\j,\j);}
  \foreach \i in {0,...,5}{\draw[gridln] (\i,\i)--(\i,6.15);}
  \draw[downpath,line width=1.4pt] (0,0)--(6.2,6.2);
  \node[downcol,font=\scriptsize,rotate=45,anchor=north] at (4.6,4.3) {diagonal};
  \foreach \j in {1,...,5}{\node[left,font=\scriptsize] at (-0.12,\j) {$y_{\j}$};}
  \foreach \i in {1,...,4}{\node[font=\scriptsize,gray!80] at (\i,0.5) {$\i$};}
  \draw[uppath,->] (0,1)--(1,1)--(1,6.15);                                   
  \draw[uppath,->] (0,2)--(2,2)--(2,4)--(3,4)--(3,6.15);                     
  \draw[uppath,->] (0,3.1)--(1.1,3.1)--(1.1,4)--(2,4)--(2,6.15);
  \draw[uppath,->] (0,3)--(2.1,3)--(2.1,4.1)--(3.1,4.1)--(3.1,5)--(4,5)--(4,6.15);
  \draw[uppath,->] (0,5)--(1.1,5)--(1.1,6.15);
  \draw[uppath,line width=3.6pt,opacity=0.40] (0,-0.25)--(0,6.2);
  \node[font=\scriptsize] at (0,-0.55) {$\infty$};
  \node[font=\scriptsize] at (0,6.5) {$\infty$};
\end{tikzpicture}

\bigskip

\begin{tikzpicture}[scale=0.85,line cap=round,
    baseline={(current bounding box.center)},font=\small,
    every node/.style={inner sep=1.4pt}]
  \def\lw{2pt}
  \draw[uppath,line width=\lw,ybeflow]  (-1.2,-0.9)--(0.2,0);
  \draw[uppath,line width=\lw,ybeflowb] (0.2,0)--(1.6,0.9);
  \draw[uppath,line width=\lw,ybeflow]  (1.6,0.9)--(3.0,0.9);
  \draw[downpath,line width=\lw,ybeflow]  (-1.2,0.9)--(0.2,0);
  \draw[downpath,line width=\lw,ybeflowb] (0.2,0)--(1.6,-0.9);
  \draw[downpath,line width=\lw,ybeflow]  (1.6,-0.9)--(3.0,-0.9);
  \draw[downpath,line width=\lw,ybeflow] (1.6,-0.9)--(1.6,-2.0);
  \draw[downpath,line width=\lw,ybeflow] (1.6,0)--(1.6,-0.9);
  \draw[uppath,line width=\lw,ybeflow]   (1.6,0)--(1.6,0.9);
  \draw[uppath,line width=\lw,ybeflow]   (1.6,0.9)--(1.6,2.0);
  \node[circle,fill=black,inner sep=0pt,minimum size=3pt] at (0.2,0){};
  \node[circle,fill=black,inner sep=0pt,minimum size=3pt] at (1.6,0.9){};
  \node[circle,fill=black,inner sep=0pt,minimum size=3pt] at (1.6,-0.9){};
  \node[left]  at (-1.2,0.9){$i_1$};
  \node[left]  at (-1.2,-0.9){$i_2$};
  \node[above] at (1.6,2.0){$j_3$};
  \node[below] at (1.6,-2.0){$i_3$};
  \node[right] at (3.0,0.9){$j_2$};
  \node[right] at (3.0,-0.9){$j_1$};
  \node at (0.73,0.72){$k_2$};
  \node at (0.73,-0.72){$k_1$};
  \node[right] at (1.70,-0.04){$k_3$};
\end{tikzpicture}
\qquad$=$\qquad
\begin{tikzpicture}[scale=0.85,line cap=round,
    baseline={(current bounding box.center)},font=\small,
    every node/.style={inner sep=1.4pt}]
  \def\lw{2pt}
  \draw[uppath,line width=\lw,ybeflow]  (-1.2,-0.9)--(0,-0.9);
  \draw[uppath,line width=\lw,ybeflowb] (0,-0.9)--(1.5,0);
  \draw[uppath,line width=\lw,ybeflow]  (1.5,0)--(3.0,0.9);
  \draw[downpath,line width=\lw,ybeflow]  (-1.2,0.9)--(0,0.9);
  \draw[downpath,line width=\lw,ybeflowb] (0,0.9)--(1.5,0);
  \draw[downpath,line width=\lw,ybeflow]  (1.5,0)--(3.0,-0.9);
  \draw[uppath,line width=\lw,ybeflow]   (0,-2.0)--(0,-0.9);
  \draw[uppath,line width=\lw,ybeflow]   (0,-0.9)--(0,0);
  \draw[downpath,line width=\lw,ybeflow] (0,0.9)--(0,0);
  \draw[downpath,line width=\lw,ybeflow] (0,2.0)--(0,0.9);
  \node[circle,fill=black,inner sep=0pt,minimum size=3pt] at (0,0.9){};
  \node[circle,fill=black,inner sep=0pt,minimum size=3pt] at (0,-0.9){};
  \node[circle,fill=black,inner sep=0pt,minimum size=3pt] at (1.5,0){};
  \node[left]  at (-1.2,0.9){$i_1$};
  \node[left]  at (-1.2,-0.9){$i_2$};
  \node[above] at (0,2.0){$j_3$};
  \node[below] at (0,-2.0){$i_3$};
  \node[right] at (3.0,0.9){$j_2$};
  \node[right] at (3.0,-0.9){$j_1$};
  \node at (0.91,0.72){$k_1'$};
  \node at (0.91,-0.72){$k_2'$};
  \node[left] at (-0.24,0.0){$k_3'$};
\end{tikzpicture}
\caption{Top: a configuration of one of the vertex models of these notes
(\Cref{subsec:model_II}): paths enter on the left and travel up and to the
right, each row of the grid carries its spectral parameter $y_j$, and the
weight of the configuration is the product of the local vertex weights. Bottom:
the Yang--Baxter equation (\Cref{subsec:ybe_cauchy}): the number of arrows
along each of the six boundary edges $i_1,i_2,i_3,j_1,j_2,j_3$ is fixed, the
arrow counts $k_1,k_2,k_3$ and $k_1',k_2',k_3'$ on the internal edges are
summed over, and the two partition functions are equal.}
\label{fig:intro_model_ybe}
\end{figure}

\medskip

Uniqueness of the matching is special to the Schur weights. For a general
YBE the summands of the two sides are not matched term by term; one
couples the two sides instead by a pair of Markov transitions, a construction
called \emph{bijectivization}, introduced by
Bufetov--Petrov \cite{BufetovPetrovYB2017}, which we take up in
\Cref{sec:bijectivization}. The deterministic matching just described is its
degenerate case, the one in which all transition probabilities are $0$ or $1$.

\medskip

\textbf{Why rederive RSK from the Yang--Baxter equation?}
The YBE is a local identity that makes a vertex model solvable.
It may seem surprising that the global RSK bijection
is its direct consequence.
Moreover, unlike the bijection itself, the
derivation from the YBE
admits interesting deformations. 
We pay special attention to the
$q$- and $t$-deformations (the $q$-Whittaker and
Hall--Littlewood weights of \Cref{sec:qdef}):
on the YBE side the deformed weights look very similar to the Schur ones,
but on the RSK side, where the deterministic matching is lost, they lead to
\emph{randomized} RSK-type insertions and interacting particle systems.

\medskip

We assume basic familiarity with RSK and Schur symmetric polynomials. No familiarity with
vertex models or quantum integrability is needed: we define in the text every
notion from that side that our arguments use, while a few standard terms
appear only as references to the literature. This survey grew out of overview
lectures given at the conference ``Solvable lattice models, representation theory of
quantum groups and algebraic combinatorics'' at Institut Mittag-Leffler, Djursholm,
Sweden, July 27--31, 2026.

\subsection{A worked example}
\label{subsec:intro_worked_example}

RSK is a bijection between matrices with nonnegative
integer entries and pairs $(P,Q)$ of semistandard Young
tableaux of the same shape, see Knuth \cite{Knuth1970},
Fulton \cite{fulton1997young}, and Stanley \cite{Stanley1999}. It reads, for example,
\begin{equation*}
  \begin{pmatrix} 1&1&0\\ 0&2&0\\ 1&0&1 \end{pmatrix}
  \;\longleftrightarrow\;
  P=\begin{tikzpicture}[scale=0.42,baseline={(current bounding box.center)}]
    \foreach \rr/\ww in {1/4,2/1,3/1}{\foreach \cc in {1,...,\ww}{\draw (\cc-1,-\rr+1) rectangle (\cc,-\rr);}}
    \foreach \cc/\val in {1/1,2/1,3/2,4/2}{\node at (\cc-0.5,-0.5) {\small$\val$};}
    \node at (0.5,-1.5) {\small$2$};
    \node at (0.5,-2.5) {\small$3$};
  \end{tikzpicture}\,,
  \qquad
  Q=\begin{tikzpicture}[scale=0.42,baseline={(current bounding box.center)}]
    \foreach \rr/\ww in {1/4,2/1,3/1}{\foreach \cc in {1,...,\ww}{\draw (\cc-1,-\rr+1) rectangle (\cc,-\rr);}}
    \foreach \cc/\val in {1/1,2/1,3/2,4/3}{\node at (\cc-0.5,-0.5) {\small$\val$};}
    \node at (0.5,-1.5) {\small$2$};
    \node at (0.5,-2.5) {\small$3$};
  \end{tikzpicture}\,.
\end{equation*}
Recall Fomin's growth-diagram form of RSK \cite{Fomin1986},
\cite{fomin1995schensted}: the
vertices of an $N\times N$ grid are labeled by
partitions, the labels at adjacent vertices interlace
(that is, differ by a \emph{horizontal strip}, the part
of a semistandard tableau occupied by a given entry; written
$\varkappa\prec\lambda$), the left and bottom boundaries
carry the empty partitions, and each unit cell carries a
nonnegative integer, the corresponding entry of the input
matrix. 
The two output tableaux are read off the top and
the right boundary of the grid.
See \Cref{fig:intro_growth_example} for an example
involving the same $3\times3$ matrix as above.
Expository accounts of growth diagrams include
Roby \cite{Roby91ThesisRSK}, Krattenthaler \cite{krattenthaler2006growth},
and Appendix~1 (by Fomin) to Chapter~7 of Stanley \cite{Stanley1999}.

Let us examine a single cell. It looks as follows:
\begin{equation*}
\begin{tikzpicture}[scale=1.5,baseline={(current bounding box.center)},font=\small,
    pp/.style={fill=white,inner sep=1pt,font=\scriptsize}]
  \draw[gridln] (0,0) rectangle (1,1);
  \node[gray!75,font=\scriptsize] at (0.5,0.5) {$k$};
  \node[pp] at (0,1) {$\lambda$};
  \node[pp] at (1,1) {$\nu$};
  \node[pp] at (0,0) {$\varkappa$};
  \node[pp] at (1,0) {$\mu$};
\end{tikzpicture}
\end{equation*}
The bottom-left, top-left, and bottom-right corners carry
partitions which are already known from the previous cells. 
The entry $k$ inside the cell is the corresponding entry of the input matrix.
The top-right partition $\nu$ is uniquely determined from this data
via
\begin{equation}
  \label{eq:intro_toggle}
  \nu_1 = k+\max(\lambda_1,\mu_1),
  \qquad
  \nu_{c+1} = \max(\lambda_{c+1},\mu_{c+1})+\min(\lambda_c,\mu_c)-\varkappa_c
  \quad(c\ge1).
\end{equation}
With $\lambda$ and $\mu$ held fixed,
\eqref{eq:intro_toggle} is a bijection between the pairs
$(\varkappa,k)$ with $\varkappa\prec\lambda$,
$\varkappa\prec\mu$ and $k\in\mathbb{Z}_{\ge0}$, and the
partitions $\nu$ with $\nu\succ\lambda$, $\nu\succ\mu$.
Identity \eqref{eq:intro_toggle} is the local rule of the RSK
which is sometimes referred to as the \emph{toggle rule}, 
cf.~Hopkins \cite{hopkins2014rsk}.
We return to it in \Cref{subsec:rsk_toggles}.

\begin{figure}[ht]
\centering
\begin{tikzpicture}[pp/.style={fill=white,inner sep=1pt,font=\scriptsize},
    yy/.style={fill=white,inner sep=1pt},
    corner/.style={font=\tiny,gray!50!black}]
  \fill[gray!12] (1.37,1.37) rectangle (2.38,2.38);
  \foreach \a in {0,1,2,3}{\draw[gridln] ({1.25*\a},0)--({1.25*\a},3.75);}
  \foreach \b in {0,1,2,3}{\draw[gridln] (0,{1.25*\b})--(3.75,{1.25*\b});}
  \node[gray!75,font=\scriptsize] at (0.625,0.625) {$1$};
  \node[gray!75,font=\scriptsize] at (1.875,0.625) {$0$};
  \node[gray!75,font=\scriptsize] at (3.125,0.625) {$1$};
  \node[gray!75,font=\scriptsize] at (0.625,1.875) {$0$};
  \node[gray!75,font=\scriptsize] at (1.875,1.875) {$2$};
  \node[gray!75,font=\scriptsize] at (3.125,1.875) {$0$};
  \node[gray!75,font=\scriptsize] at (0.625,3.125) {$1$};
  \node[gray!75,font=\scriptsize] at (1.875,3.125) {$1$};
  \node[gray!75,font=\scriptsize] at (3.125,3.125) {$0$};
  \node[corner] at (1.48,1.44) {$\varkappa$};
  \node[corner] at (1.48,2.31) {$\lambda$};
  \node[corner] at (2.27,1.44) {$\mu$};
  \node[corner] at (2.27,2.31) {$\nu$};
  \foreach \p in {(0,0),(1.25,0),(2.5,0),(3.75,0),(0,1.25),(0,2.5),(0,3.75)}
    {\node[pp] at \p {$\varnothing$};}
  \node[yy] at (1.25,1.25) {\yd{0.15}{1}};
  \node[yy] at (2.5,1.25) {\yd{0.15}{1}};
  \node[yy] at (3.75,1.25) {\yd[downcol]{0.15}{2}};
  \node[yy] at (1.25,2.5) {\yd{0.15}{1}};
  \node[yy] at (2.5,2.5) {\yd{0.15}{3}};
  \node[yy] at (3.75,2.5) {\yd[downcol]{0.15}{3,1}};
  \node[yy] at (1.25,3.75) {\yd[upcol]{0.15}{2}};
  \node[yy] at (2.5,3.75) {\yd[upcol]{0.15}{4,1}};
  \node[yy] at (3.75,3.75) {\yd{0.15}{4,1,1}};
  \node[upcol,font=\small] at (1.875,4.2) {$P$};
  \node[downcol,font=\small] at (4.45,1.875) {$Q$};
\end{tikzpicture}
\caption{The growth diagram of a $3\times3$ input matrix. Each vertex carries a
partition, drawn as its Young diagram; the left and bottom boundaries are
empty; the top and right boundary chains encode the output tableaux $P$ and
$Q$. In the shaded cell the corners are $\varkappa=(1)$,
$\lambda=(1)$, $\mu=(1)$ in the local rule \eqref{eq:intro_toggle} with $k=2$,
which yields $\nu=(3)$.}
\label{fig:intro_growth_example}
\end{figure}
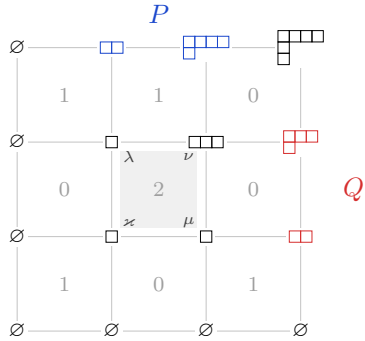

Let us take an example. To better illustrate the matching, we choose
$\lambda$ and $\mu$ bigger than the partitions in
\Cref{fig:intro_growth_example}: say $\lambda=(7,3)$ and $\mu=(6,4)$, with
the entry $k=2$ inside the cell. To keep the lists short, we restrict the
bottom-left corner to a single row, $\varkappa=(\varkappa_1)$. The
interlacing $\varkappa\prec\lambda$ means $7\ge\varkappa_1\ge3$, and
$\varkappa\prec\mu$ means $6\ge\varkappa_1\ge4$; together they leave
exactly three admissible corners, $\varkappa_1\in\{4,5,6\}$. For each of
them the rule \eqref{eq:intro_toggle} yields $\nu_1=2+\max(7,6)=9$,
$\nu_2=\max(3,4)+\min(7,6)-\varkappa_1=10-\varkappa_1$, and
$\nu_3=0+\min(3,4)-0=3$; that is, $\nu=(9,\,10-\varkappa_1,\,3)$. 
In the
opposite direction, a top-right corner $\nu$ with $\nu_1=9$ and $\nu_3=3$
interlaces with both $\lambda$ and $\mu$ exactly when $4\le\nu_2\le6$.
This yields three admissible top-right corners with $\nu_2\in\{4,5,6\}$, 
as it should be by the toggle bijection 
$(\varkappa,k)\mapsto\nu$.

\medskip

In the simplified example, one degree of freedom remains on each
side of the toggle bijection: the input corner varies
through $\varkappa_1\in\{4,5,6\}$, and the output
partition through $\nu_2\in\{4,5,6\}$. As we explain in
these notes, one such degree of freedom corresponds to a
single instance of the YBE. Generically, the full map
$(\varkappa,k)\mapsto\nu$ would correspond to a whole train of
Yang--Baxter moves, which we explain in \Cref{subsec:schur_yb}.

To write down
the YBE for our example, we need spectral
parameters $x,y$ (which are exchanged by the YBE).
This instance of the YBE reads
\begin{equation}
  \label{eq:intro_two_sides}
  \underbrace{x^{4}y^{3}+x^{5}y^{4}+x^{6}y^{5}}_{\varkappa_1=6,\;5,\;4}
  \;=\;
  \underbrace{x^{4}y^{3}+x^{5}y^{4}+x^{6}y^{5}}_{\nu_2=4,\;5,\;6}\,,
\end{equation}
with one summand for each admissible $\varkappa_1$ on the
left, and one for each admissible $\nu_2$ on the right.
\Cref{fig:intro_ybe_instance} draws all six
configurations of this instance. 
The boundary occupations are the
same in all six pictures and are fixed by the data of the
example: $\nu_1-\mu_1=3$ and $\nu_1-\lambda_1=2$ arrows
enter on the left, $\lambda_2-\varkappa_2=3$ and
$\mu_2-\varkappa_2=4$ leave on the right, and the
vertical line carries $\lambda_1-\lambda_2=4$ arrows at
the bottom and $\mu_1-\mu_2=2$ at the top.
Only the three
internal occupations vary. The vertex weights behind
these monomials are written out in
\Cref{subsec:model_II,subsec:ybe_cauchy}, see
\eqref{eq:upright_cross}, \eqref{eq:downright_cross}, and
\eqref{eq:bbR}; the caption of
\Cref{fig:intro_ybe_instance} briefly describes them in words.

\begin{figure}[ht]
\centering
\newcommand{\bnd}[5][uppath]{%
  \ifnum#2>0
  \foreach \s in {1,...,#2}{%
    \pgfmathsetmacro{\off}{(\s-(#2+1)/2)*0.09}
    \draw[#1,line width=0.8pt,->] ($#3+\off*#5$)--($#4+\off*#5$);}%
  \fi}
\newcommand{\bndto}[5][uppath]{%
  \ifnum#2>0
  \foreach \s in {1,...,#2}{%
    \pgfmathsetmacro{\off}{(\s-(#2+1)/2)*0.09}
    \draw[#1,line width=0.8pt,->,shorten >=5pt] ($#3+\off*#5$)--($#4+\off*#5$);}%
  \fi}
\newcommand{\iYBEL}[3]{%
\begin{tikzpicture}[scale=0.72,line cap=round,>={Stealth[length=1.2mm]},
    baseline={(current bounding box.center)},
    skel/.style={gray!55,densely dotted,line width=0.5pt},
    vx/.style={circle,draw=black,fill=white,line width=0.7pt,inner sep=0pt,minimum size=9pt}]
  \draw[skel] (-1.2,-0.9)--(0.2,0)--(1.6,0.9)--(3.0,0.9);
  \draw[skel] (-1.2,0.9)--(0.2,0)--(1.6,-0.9)--(3.0,-0.9);
  \draw[skel] (1.6,-2.0)--(1.6,2.0);
  \bndto[uppath]{2}{(-1.2,-0.9)}{(0.2,0)}{(0,1)}
  \bndto[downpath]{3}{(-1.2,0.9)}{(0.2,0)}{(0,1)}
  \bndto[uppath]{#2}{(0.2,0)}{(1.6,0.9)}{(0,1)}
  \bndto[downpath]{#1}{(0.2,0)}{(1.6,-0.9)}{(0,1)}
  \bnd[uppath]{4}{(1.6,0.9)}{(3.0,0.9)}{(0,1)}
  \bnd[downpath]{3}{(1.6,-0.9)}{(3.0,-0.9)}{(0,1)}
  \bnd[uppath]{2}{(1.6,0.9)}{(1.6,2.0)}{(1,0)}
  \bnd[downpath]{4}{(1.6,-0.9)}{(1.6,-2.0)}{(1,0)}
  \bndto[uppath]{#3}{(1.6,0)}{(1.6,0.9)}{(1,0)}
  \bndto[downpath]{#3}{(1.6,0)}{(1.6,-0.9)}{(1,0)}
  \node[vx] at (0.2,0){};
  \node[vx] at (1.6,0.9){};
  \node[vx] at (1.6,-0.9){};
  \node[upcol,font=\scriptsize,left]   at (-1.25,-0.9){$x$};
  \node[downcol,font=\scriptsize,left] at (-1.25,0.9){$y$};
\end{tikzpicture}}
\newcommand{\iYBER}[3]{%
\begin{tikzpicture}[scale=0.72,line cap=round,>={Stealth[length=1.2mm]},
    baseline={(current bounding box.center)},
    skel/.style={gray!55,densely dotted,line width=0.5pt},
    vx/.style={circle,draw=black,fill=white,line width=0.7pt,inner sep=0pt,minimum size=9pt}]
  \draw[skel] (-1.2,-0.9)--(0,-0.9)--(1.5,0)--(3.0,0.9);
  \draw[skel] (-1.2,0.9)--(0,0.9)--(1.5,0)--(3.0,-0.9);
  \draw[skel] (0,-2.0)--(0,2.0);
  \bndto[uppath]{2}{(-1.2,-0.9)}{(0,-0.9)}{(0,1)}
  \bndto[downpath]{3}{(-1.2,0.9)}{(0,0.9)}{(0,1)}
  \bndto[uppath]{#2}{(0,-0.9)}{(1.5,0)}{(0,1)}
  \bndto[downpath]{#1}{(0,0.9)}{(1.5,0)}{(0,1)}
  \bnd[uppath]{4}{(1.5,0)}{(3.0,0.9)}{(0,1)}
  \bnd[downpath]{3}{(1.5,0)}{(3.0,-0.9)}{(0,1)}
  \bndto[downpath]{2}{(0,2.0)}{(0,0.9)}{(1,0)}
  \bndto[uppath]{4}{(0,-2.0)}{(0,-0.9)}{(1,0)}
  \bnd[uppath]{#3}{(0,-0.9)}{(0,0)}{(1,0)}
  \bnd[downpath]{#3}{(0,0.9)}{(0,0)}{(1,0)}
  \node[vx] at (0,0.9){};
  \node[vx] at (0,-0.9){};
  \node[vx] at (1.5,0){};
  \node[upcol,font=\scriptsize,left]   at (-1.25,-0.9){$x$};
  \node[downcol,font=\scriptsize,left] at (-1.25,0.9){$y$};
\end{tikzpicture}}
\setlength{\tabcolsep}{9pt}
\makebox[\textwidth][l]{%
\begin{tabular}{@{}ccccc@{}}
\iYBEL{1}{0}{6} & ${}+{}$ & \iYBEL{2}{1}{5} & ${}+{}$ & \iYBEL{3}{2}{4}\\[4pt]
{\footnotesize$x^{4}y^{3}$} & & {\footnotesize$x^{5}y^{4}$} & &
{\footnotesize$x^{6}y^{5}$}
\end{tabular}}

\vspace{3.5mm}

\makebox[\textwidth][r]{%
\begin{tabular}{@{}c@{\hspace{6pt}}ccccc@{}}
${}={}$ & \iYBER{0}{1}{5} & ${}+{}$ & \iYBER{1}{2}{4} & ${}+{}$ &
\iYBER{2}{3}{3}\\[4pt]
& {\footnotesize$x^{4}y^{3}$} & & {\footnotesize$x^{5}y^{4}$} & &
{\footnotesize$x^{6}y^{5}$}
\end{tabular}}
\caption{All six configurations of the YBE instance
\eqref{eq:intro_two_sides}; below each configuration is
its monomial weight. The blue line carries the spectral
parameter $x$ and the red line $y$. 
The weight of a configuration is the product of
its three vertex weights: each of the two vertices on the
vertical line contributes the parameter of its line ($x$
or $y$) raised to the number of arrows going out to the
right, and the cross vertex contributes $xy$ raised to
the minimum of its two outgoing arrow counts.}
\label{fig:intro_ybe_instance}
\end{figure}
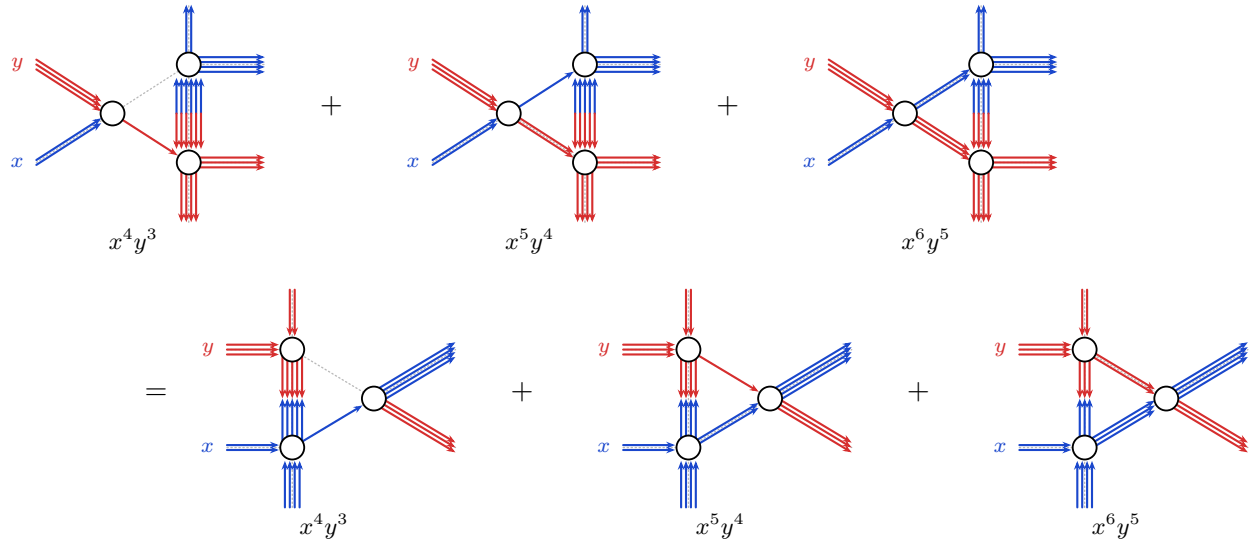

As an identity between polynomials, \eqref{eq:intro_two_sides} is a triviality.
However, as each summand corresponds 
to a choice of $\varkappa_1$ or $\nu_2$ 
on the left or right, matching these summands term by term
produces our toggle bijection:
\begin{equation*}
  \varkappa_1=6\ \longleftrightarrow\ \nu_2=4,
  \qquad
  \varkappa_1=5\ \longleftrightarrow\ \nu_2=5,
  \qquad
	\varkappa_1=4\ \longleftrightarrow\ \nu_2=6.
\end{equation*}
This is the same matching as in \eqref{eq:intro_toggle},
but now it is seen as a consequence of the YBE from 
\Cref{fig:intro_model_ybe}.
The general statement of this bijeciton is
\Cref{prop:train_is_rsk}.

\subsection{The combinatorial 3D $R$ as a set-theoretic tetrahedron map}
\label{subsec:intro_3dR}

The matching in the example
in \Cref{subsec:intro_worked_example}
moved a single part ($\varkappa_1$ or $\nu_2$)
because the fixed corners $\lambda$ and $\mu$ froze
everything else. Let us extract the map that the matching
performs in general, when the boundary data may vary as
well. Given the four occupations $i_1,i_2,i_3,j_3$,
which are the same on the two sides
(\Cref{fig:intro_model_ybe}, bottom), a summand on the
left-hand side is determined by the two outgoing
horizontal occupations $j_1$ and $j_2$ together with the
smaller of the two internal ones, $h=\min(k_1,k_2)$. We
record these as the triple $(a,h,c)=(j_1,h,j_2)$.
On the right,
a summand is determined by the two internal
occupations $k_1'$ and $k_2'$ together with 
$m=\min(j_1,j_2)$.
We record these as the triple $(a',m,c')=(k_1',m,k_2')$.

In these coordinates the unique weight-preserving
matching 
between the two sides of the YBE
reads
\begin{equation}
  \label{eq:intro_3d_R}
  (a,h,c)\longmapsto
	\bigl(h+\max(a-c,0),\;\min(a,c),\;h+\max(c-a,0)\bigr).
\end{equation}
For the middle terms in the example in \Cref{fig:intro_ybe_instance},
this gives
$(3,1,4)\mapsto(1+0,\;3,\;1+1)=(1,3,2)$.
This map is the \emph{combinatorial three-dimensional
$R$} of Kuniba--Okado \cite{KunibaOkado2012Tetrahedron3D}
and Kuniba \cite{Kuniba2016CombinatorialYBmaps},
which solves the set-theoretic Zamolodchikov tetrahedron
equation \cite{Zamolodchikov1981Tetrahedron}.
We discuss
this in more detail in
\Cref{sec:3dR}.

Note that the map \eqref{eq:intro_3d_R} is not a
set-theoretic Yang--Baxter map in the usual sense: it
acts on triples and not on pairs, and the tetrahedron
equation it satisfies lives one dimension above the
Yang--Baxter equation. Genuine set-theoretic Yang--Baxter maps do
come out of \eqref{eq:intro_3d_R}, for instance after its
third coordinate is closed into a periodic loop, which
turns it into the combinatorial $R$-matrices of
one-row crystals which also power the 
box-ball system of Takahashi--Satsuma
\cite{TakahashiSatsuma1990}; see also the survey by
Inoue--Kuniba--Takagi \cite{InoueKunibaTakagi2012}.
Keeping the middle coordinate open instead, as in these
notes, turns \eqref{eq:intro_3d_R} into Fomin's growth
rule and RSK. Thus, the box-ball system and RSK are two
global readings of one local map
(\Cref{subsec:two_geometries}). For a direct
combinatorial link between 
the box-ball system and RSK,
see Fukuda \cite{fukuda2004box}.

\subsection{From Yang--Baxter to interacting particle systems}
\label{subsec:intro_literature}

The second half of these notes
deforms the constructions we outlined
in \Cref{subsec:intro_worked_example}.
At the level of symmetric functions, the Schur polynomials 
are replaced either by the Hall--Littlewood polynomials
or by the $q$-Whittaker polynomials, which
are two different one-parameter specializations of the Macdonald
polynomials \cite{Macdonald1995}.
In the Yang--Baxter equations for the corresponding deformed
weights, the two sides no longer admit a term-by-term
bijective matching.
Instead, one can
produce a \emph{randomized} matching
(a \emph{bijectivization} of the YBE),
which is a conditional distribution (or a Markov 
transition kernel) assigning a random term
from the right-hand side to each term of the left-hand side
(or vice versa, a random term from the left-hand side to
each term of the right-hand side).

For the instance of the YBE from
\Cref{fig:intro_ybe_instance} with the $q$-Whittaker
deformed weights, at $q=\frac12$, the six
summands become proportional to
\begin{multline*}
  \underbrace{651}_{\varkappa_1=6}
  +\underbrace{\bigl(217+868\,xy\bigr)}_{\varkappa_1=5}
  +\underbrace{\bigl(4+168\,xy+448\,(xy)^2\bigr)}_{\varkappa_1=4}
  \\=\underbrace{620}_{\nu_2=4}
  +\underbrace{\bigl(245+840\,xy\bigr)}_{\nu_2=5}
	+\underbrace{\bigl(7+196\,xy+448\,(xy)^2\bigr)}_{\nu_2=6}.
\end{multline*}
The deformed summands for general $q$ are written out in
\Cref{sec:def_rsk}. The two sides are equal as
polynomials in $xy$, but no summand on the left equals a
summand on the right. The
independent coupling, which chooses a term on the right
with probability proportional to its weight, sends
$\varkappa_1=5$ to $\nu_2=4,\,5,\,6$ with probabilities
proportional to the three summands on the right. At
$x=y=\frac12$, these probabilities are
$\frac{620}{1159},\,\frac{455}{1159},\,\frac{84}{1159}$,
respectively.

\medskip

The resulting
randomized RSKs
obtained from bijectivizations of the YBE
are related to
known integrable
interacting particle systems. 
Let us
survey the literature on these constructions.
The totally asymmetric simple exclusion process (TASEP),
the simplest example of an integrable interacting particle system,
was introduced in biology by MacDonald--Gibbs--Pipkin
\cite{macdonald1968bioASEP} and in probability by
Spitzer \cite{Spitzer1970} (see also Liggett
\cite{Liggett1975} for the general theory of interacting
particle systems, and Rost \cite{Rost1981} for the limit
shape of TASEP started from the step initial configuration). 
Schur functions and RSK found applications
to particle systems around 2000
starting from the work of Johansson \cite{johansson2000shape},
though they entered probability much earlier:
Vershik--Kerov \cite{VershikKerov_LimShape1077} and
Logan--Shepp \cite{logan_shepp1977variational} found the
limit shape of the Plancherel random partition already
in 1977.

Schur measures and processes, introduced by
Okounkov \cite{okounkov2001infinite} and
Okounkov--Reshetikhin \cite{okounkov2003correlation}, are
probability distributions on partitions (or sequences
of partitions) whose probability
weights are expressed through Schur
polynomials. 
Through RSK, 
Schur measures are connected to TASEP, last-passage
percolation, and longest increasing subsequences of
random permutations, 
see Baik--Deift--Johansson
\cite{baik1999distribution},
Borodin--Okounkov--Olshanski
\cite{Borodin2000b},
and Johansson
\cite{johansson2000shape}.
Borodin and Ferrari
\cite{BorFerr2008DF} constructed two-dimensional growth
dynamics on Schur processes. These dynamics contain TASEP
as a one-dimensional marginal.

Following Borodin--Ferrari, other symmetric functions
from the Macdonald hierarchy \cite{Macdonald1995} give
rise to deformed particle systems such as the $q$-TASEP
pioneered by Borodin--Corwin
\cite{BorodinCorwin2011Macdonald}. Randomized RSK-type
dynamics for these deformations predate Yang--Baxter
bijectivizations: such dynamics were constructed by
O'Connell--Pei \cite{OConnellPei2012} and Borodin--Petrov
\cite{BorodinPetrov2013NN}, and then by Bufetov--Petrov
\cite{BufetovPetrov2014} in the Hall--Littlewood case and
Matveev--Petrov \cite{MatveevPetrov2014} in the full
$q$-Whittaker case (including discrete variants of
$q$-TASEP and $q$-PushTASEP), all without using the YBE.

Stochastic vertex models entered probability in Borodin
\cite{Borodin2014vertex}, Corwin--Petrov
\cite{CorwinPetrov2015}, and Borodin--Petrov
\cite{BorodinPetrov2016inhom}. Bijectivization, first
used in Bufetov--Petrov \cite{BufetovPetrovYB2017}, turns
the YBE itself into local stochastic moves. It was
developed further in \cite{BufetovMucciconiPetrov2018},
\cite{MucciconiPetrov2020},
\cite{PetrovSaenz2019backTASEP},
\cite{petrov2022rewriting},
\cite{nicoletti2022irreversible}. The Yang--Baxter
bijectivization serves as a bridge between vertex models
and interacting particle systems: directly, through
scalar marginals, or through degenerations, it connects
to $q$-TASEP and $q$-PushTASEP, to the stochastic
six-vertex model and ASEP \cite{BCG6V},
\cite{BorodinBufetovWheeler2016},
\cite{BufetovMatveev2017},
and to directed
random polymers 
\cite{Seppalainen2012},
\cite{OConnellYor2001}, \cite{CorwinBarraquand2015Beta}.
The symmetric functions
play a secondary role, and the YBE side can even produce
new families, such as the spin $q$-Whittaker polynomials
\cite{BorodinWheelerSpinq}, \cite{MucciconiPetrov2020},
\cite{korotkikh2024representation}. See
\cite{Corwin2014ICM} and \cite{borodin2019shift} for
charts of hierarchies of integrable models. In these
notes, \Cref{sec:bijectivization} presents the general
bijectivization, and \Cref{sec:def_rsk} sketches 
some of the
resulting particle systems.

\medskip

Let us also mention that 
in rigorous
two-dimensional statistical mechanics, the random star--triangle transformation 
(itself a bijectivization of the YBE) underlies the proof of
rotational invariance of the critical Ising, FK, and six-vertex models in
Duminil-Copin--Kozlowski--Krachun--Manolescu--Oulamara
\cite{DuminilCopinKozlowskiKrachunManolescuOulamara2020}.

\subsection{Outline}
\label{subsec:intro_outline}

The notes are organized as follows.
\Cref{sec:schur_models} presents the three vertex models
whose partition functions are the Schur polynomials,
together with the YBE and the Cauchy identity. In
\Cref{sec:rsk}, we match the two sides of the YBE term by
term and identify this matching with RSK in Fomin's
growth-diagram form. \Cref{sec:3dR} recasts the same
matching in the coordinates of \eqref{eq:intro_3d_R} as
the combinatorial three-dimensional $R$ satisfying the
tetrahedron equation, and discusses the periodic closure
and the one-row crystals powering the box-ball system. In
\Cref{sec:bijectivization}, we describe the general
bijectivization procedure, of which the matching of
\Cref{sec:rsk} is the deterministic case. \Cref{sec:qdef}
treats the $t$- and $q$-deformations of the Schur vertex
models (Hall--Littlewood and $q$-Whittaker), related by
Macdonald duality and fusion. \Cref{sec:def_rsk} sketches
how the deformed weights, combined with bijectivization,
yield integrable stochastic interacting particle systems.
\Cref{sec:outlook} discusses further directions.

A reader who wants only the derivation of RSK from the
Yang--Baxter equation may read the Introduction and then
\Cref{sec:schur_models,sec:rsk}, and stop.
\Cref{sec:3dR} adds the three-dimensional point of view
connected to the set-theoretic tetrahedron equation.
The sections after it form the probabilistic half of
the notes powered by $q$- and $t$- deformations,
where the matching of the two sides
of the YBE
is no
longer deterministic.

\subsection{Acknowledgements}

I am grateful to Alexei Borodin, Ben Brubaker, Persi Diaconis, Christian Korff,
Matteo Mucciconi, and Greta Panova for helpful suggestions.
This material is based upon work supported by the Swedish Research Council under grant
no.~2021-06594 while the author was in residence at Institut Mittag-Leffler in Djursholm,
Sweden during the conference
``Solvable lattice models, representation theory of quantum groups and algebraic
combinatorics'' at Institut Mittag-Leffler, July 2026.
I am grateful to the organizers of that conference,
and to its participants
for their enthusiastic questions.

\smallskip\noindent

I was partially supported by the NSF grant DMS-2153869 and the Simons
Collaboration Grant for Mathematicians 709055.

\medskip\noindent
\textbf{AI Statement.}
Claude Code (Opus~4.8-5, Fable~5--5.1) helped draft portions
of the exposition, generate several of the \texttt{TikZ} figures, and carry out
coding and numerical consistency checks. All definitions, statements, and proofs
were verified by the author and edited for presentation.
The author bears sole responsibility for their correctness.

\section{Vertex models for Schur polynomials}
\label{sec:schur_models}

\subsection{Conventions}
\label{subsec:conventions}

A \emph{signature} of length $N$ is a weakly decreasing integer vector
$\lambda=(\lambda_1\ge\cdots\ge\lambda_N)$, $\lambda_i\in\mathbb{Z}$; the set of these is
$\mathrm{Sign}_N$. A \emph{partition} is a signature with $\lambda_N\ge0$, identified with
its Young diagram (the left-justified array of $\lambda_i$ boxes in row $i$, drawn
top to bottom in the English convention); $|\lambda|=\sum_i\lambda_i$ is its size,
$\ell(\lambda)$ its number of nonzero parts, and $s_\lambda(x_1,\dots,x_n)$ is the
Schur polynomial. We write $\mathbb{Z}_{\ge0}=\{0,1,2,\dots\}$. The \emph{conjugate} (or
transpose) partition $\lambda'$ is the reflection of the diagram across the main
diagonal,
\begin{equation*}
  \lambda'_i\;=\;\#\{j:\lambda_j\ge i\}\qquad(\text{the height of the $i$th column of }\lambda),
\end{equation*}
so that $(\lambda')'=\lambda$ and $|\lambda'|=|\lambda|$. Our running example is
\begin{equation*}
  \lambda=(4,4,1),\qquad \lambda'=(3,2,2,2),
\end{equation*}
displayed in \Cref{fig:young}.

\medskip

\emph{Strips and interlacing.} Two partitions $\mu\subseteq\lambda$ (diagram
inclusion) differ by a \emph{horizontal strip} if their skew shape $\lambda/\mu$
has at most one box in each \emph{column}, and by a \emph{vertical strip} if
$\lambda/\mu$ has at most one box in each \emph{row}. These two conditions are
exchanged by conjugation, and each is encoded by an interlacing relation on
signatures. We write $\mu\prec\lambda$ (equivalently $\lambda\succ\mu$) for the
\emph{horizontal} interlacing
\begin{equation*}
  \lambda_1\ \ge\ \mu_1\ \ge\ \lambda_2\ \ge\ \mu_2\ \ge\ \cdots,
\end{equation*}
which for partitions says exactly that $\lambda/\mu$ is a horizontal strip. The
\emph{transposed} (vertical) interlacing is its conjugate,
\begin{equation*}
  \mu\prec'\lambda \quad\stackrel{\text{def}}{\Longleftrightarrow}\quad \mu'\prec\lambda',
  \qquad\text{i.e.}\qquad
  \lambda'_1\ \ge\ \mu'_1\ \ge\ \lambda'_2\ \ge\ \mu'_2\ \ge\ \cdots.
\end{equation*}
See
\Cref{fig:strips} for an illustration.

\begin{figure}[ht]
\centering
\begin{tikzpicture}[scale=0.42,baseline=-1.6cm]
  \foreach \rr/\ww in {1/4,2/4,3/1}{%
    \foreach \cc in {1,...,\ww}{\draw (\cc-1,-\rr+1) rectangle (\cc,-\rr);}}
  \node at (2,-3.6) {$\lambda=(4,4,1)$};
\end{tikzpicture}
\hspace{1.1cm}
\begin{tikzpicture}[scale=0.42,baseline=-1.6cm]
  \foreach \rr/\ww in {1/3,2/2,3/2,4/2}{%
    \foreach \cc in {1,...,\ww}{\draw (\cc-1,-\rr+1) rectangle (\cc,-\rr);}}
  \node at (1.6,-4.6) {$\lambda'=(3,2,2,2)$};
\end{tikzpicture}
\hspace{1.1cm}
\begin{tikzpicture}[scale=0.42,baseline=-1.6cm]
  \foreach \rr/\ww in {1/4,2/4,3/1}{%
    \foreach \cc in {1,...,\ww}{\draw (\cc-1,-\rr+1) rectangle (\cc,-\rr);}}
  \foreach \cc/\val in {1/1,2/1,3/2,4/2}{\node at (\cc-0.5,-0.5) {\small$\val$};}
  \foreach \cc/\val in {1/2,2/3,3/3,4/3}{\node at (\cc-0.5,-1.5) {\small$\val$};}
  \node at (0.5,-2.5) {\small$3$};
  \node at (2,-3.6) {SSYT $T$};
\end{tikzpicture}
\caption{The partition
$\lambda=(4,4,1)$ (left), its conjugate $\lambda'=(3,2,2,2)$ obtained by reflecting
across the main diagonal (middle), and a semistandard Young tableau $T$ of shape
$\lambda$ with entries in $\{1,\dots,n\}$, $n=3$ (right): rows weakly increase
left to right, columns strictly increase top to bottom. The cells filled with a
given entry $k$ form a horizontal strip (at most one per column), so $T$ records
the horizontal-strip interlacing sequence
$\varnothing\prec(2)\prec(4,1)\prec(4,4,1)$. The weight of $T$ is
$x^{\mathrm{wt}(T)}=x_1^{2}x_2^{3}x_3^{4}$, and $s_\lambda=\sum_T x^{\mathrm{wt}(T)}$
is the sum of weights
over all semistandard tableaux.}
\label{fig:young}
\end{figure}
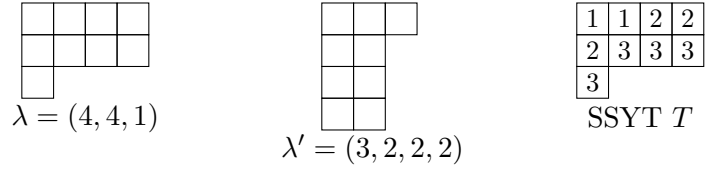

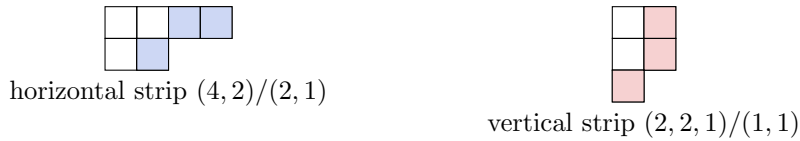
\begin{figure}[ht]
\centering
\begin{tikzpicture}[scale=0.42,baseline=-1.0cm]
  \foreach \rr/\ww in {1/4,2/2}{%
    \foreach \cc in {1,...,\ww}{\draw (\cc-1,-\rr+1) rectangle (\cc,-\rr);}}
  \foreach \rr/\cc in {1/3,1/4,2/2}{\fill[upcol!22] (\cc-1,-\rr+1) rectangle (\cc,-\rr);}
  \draw (2,0) rectangle (3,-1); \draw (3,0) rectangle (4,-1); \draw (1,-1) rectangle (2,-2);
  \node at (2,-2.7) {\small horizontal strip $(4,2)/(2,1)$};
\end{tikzpicture}
\hspace{1.6cm}
\begin{tikzpicture}[scale=0.42,baseline=-1.0cm]
  \foreach \rr/\ww in {1/2,2/2,3/1}{%
    \foreach \cc in {1,...,\ww}{\draw (\cc-1,-\rr+1) rectangle (\cc,-\rr);}}
  \foreach \rr/\cc in {1/2,2/2,3/1}{\fill[downcol!22] (\cc-1,-\rr+1) rectangle (\cc,-\rr);}
  \draw (1,0) rectangle (2,-1); \draw (1,-1) rectangle (2,-2); \draw (0,-2) rectangle (1,-3);
  \node at (1,-3.7) {\small vertical strip $(2,2,1)/(1,1)$};
\end{tikzpicture}
\caption{Left: a horizontal strip $\mu\prec\lambda$ (shaded), at most one box per
column. Right: a vertical strip $\mu\prec'\lambda$ (shaded), at most one box per
row. Conjugation exchanges the two.}
\label{fig:strips}
\end{figure}

\medskip

\emph{Weights.} Vertex weights are taken \emph{nonnegative} --- and that is all
the Yang--Baxter bijectivization of \Cref{sec:bijectivization} will require of
them. A model is moreover called \emph{stochastic} when its weights, read as
conditional probabilities of the outgoing edges given the incoming ones, sum to
one, so that the row-to-row transfer matrix is a Markov operator; this is a convenient extra structure used in connection
with interacting particle systems.
Both are claims about real parameters: whenever a statement below is probabilistic, we work
in the regime $0\le q<1$, $0\le t<1$, and $x_i,y_j\ge0$ with $x_iy_j<1$. The identities
themselves need no such restriction, holding as identities of rational functions or of
formal series.

\subsection{Schur vertex model I: Horizontal strips}
\label{subsec:model_I}

A solvable lattice model assigns Boltzmann weights to the local configurations on the
square grid, with paths (arrows) conserved at each vertex, and the weight of a whole
configuration is the product of its vertex weights.
The most classical instance is the \emph{free fermion five-vertex model}, whose
partition function is a Schur polynomial. This identity is classical
\cite{ZinnJustin20096Vertex}: in the equivalent phase model language
(cf.\ \Cref{rem:higher_spin} below) it appears in \cite{tsilevich2006quantum},
and as a six-vertex partition function governed by the YBE
it was developed in \cite{brubaker2011schur}. We describe the model concretely.

\medskip

Fix $n$ and work on the grid $\mathbb{Z}_{\ge1}\times\{1,\dots,n\}$: $n$ horizontal rows
indexed from the bottom by $i=1,\dots,n$, and columns indexed by a half-line
coordinate. Each edge carries an occupation in $\{0,1\}$, and we encode a local vertex by
its four edge occupations
\begin{equation}
  \label{eq:vertex_args}
  (i_1,j_1;\,i_2,j_2)
  \;=\;
  (\text{bottom in},\ \text{left in};\ \text{top out},\ \text{right out}),
  \qquad i_1,j_1,i_2,j_2\in\{0,1\}.
\end{equation}
Path conservation forces the weight to vanish unless $i_1+j_1=i_2+j_2$, leaving
six admissible vertices; the \emph{five-vertex} condition discards the
doubly-occupied vertex $(1,1;1,1)$. The remaining five vertices, and the
Boltzmann weights attached to a vertex sitting in row $i$, are shown in
\Cref{fig:fivevertex_weights}.

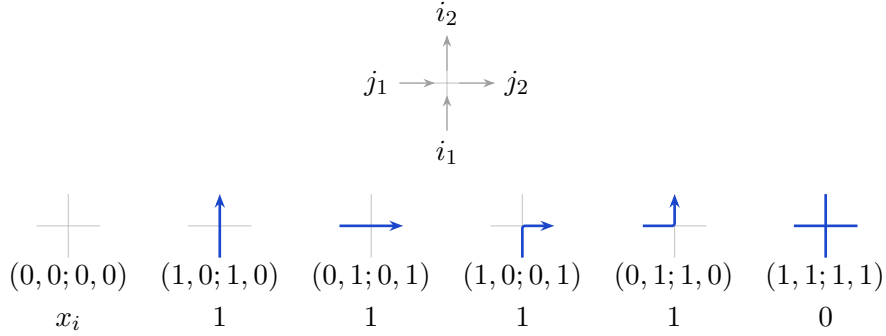
\begin{figure}[ht]
\centering
\begin{tikzpicture}[scale=0.9,>={Stealth[length=1.8mm]}]
  \draw[gray!45,line width=0.4pt] (0,-0.7)--(0,0.7) (-0.7,0)--(0.7,0);
  \draw[-{Stealth[length=1.5mm]},gray!75] (0,-0.7)--(0,-0.18);
  \draw[-{Stealth[length=1.5mm]},gray!75] (-0.7,0)--(-0.18,0);
  \draw[-{Stealth[length=1.5mm]},gray!75] (0,0.18)--(0,0.7);
  \draw[-{Stealth[length=1.5mm]},gray!75] (0.18,0)--(0.7,0);
  \node[below] at (0,-0.7) {$i_1$};
  \node[left]  at (-0.7,0) {$j_1$};
  \node[above] at (0,0.7) {$i_2$};
  \node[right] at (0.7,0) {$j_2$};
\end{tikzpicture}
\par\medskip
\begin{tabular}{cccccc}
\fvv{} &
\fvv{\draw[upcol,line width=1pt,->] (0,-0.6)--(0,0.6);} &
\fvv{\draw[upcol,line width=1pt,->] (-0.6,0)--(0.6,0);} &
\fvv{\draw[upcol,line width=1pt,->] (0,-0.6)--(0,0)--(0.6,0);} &
\fvv{\draw[upcol,line width=1pt,->] (-0.6,0)--(0,0)--(0,0.6);} &
\fvv{\draw[upcol,line width=1pt] (0,-0.6)--(0,0.6); \draw[upcol,line width=1pt] (-0.6,0)--(0.6,0);} \\[3pt]
$(0,0;0,0)$ & $(1,0;1,0)$ & $(0,1;0,1)$ & $(1,0;0,1)$ & $(0,1;1,0)$ & $(1,1;1,1)$ \\[2pt]
$x_i$ & $1$ & $1$ & $1$ & $1$ & $0$ \\
\end{tabular}
\caption{Top: the edge-occupation convention of \eqref{eq:vertex_args} --- a vertex
in row $i$ with vertical-in $i_1$ (bottom), horizontal-in $j_1$ (left),
vertical-out $i_2$ (top), horizontal-out $j_2$ (right); paths run up and to the
right. Bottom: the five admissible vertices with their
Boltzmann weights in row $i$.}
\label{fig:fivevertex_weights}
\end{figure}

\begin{remark}[Free fermion condition]
\label{rmk:free_fermion}
The six Boltzmann weights (the five admissible ones together with the discarded doubly-occupied vertex of weight~$0$) satisfy the \emph{free fermion condition}
\begin{equation}
\label{eq:free_fermion}
w\!\left(\vcenter{\hbox{\fvv{}}}\right)
\cdot
w\!\left(\vcenter{\hbox{\fvv{\draw[upcol,line width=1pt] (0,-0.6)--(0,0.6);\draw[upcol,line width=1pt] (-0.6,0)--(0.6,0);}}}\right)
+
w\!\left(\vcenter{\hbox{\fvv{\draw[upcol,line width=1pt,->] (0,-0.6)--(0,0)--(0.6,0);}}}\right)
\cdot
w\!\left(\vcenter{\hbox{\fvv{\draw[upcol,line width=1pt,->] (-0.6,0)--(0,0)--(0,0.6);}}}\right)
=
w\!\left(\vcenter{\hbox{\fvv{\draw[upcol,line width=1pt,->] (0,-0.6)--(0,0.6);}}}\right)
\cdot
w\!\left(\vcenter{\hbox{\fvv{\draw[upcol,line width=1pt,->] (-0.6,0)--(0.6,0);}}}\right)\;.
\end{equation}
\end{remark}

\medskip

A partition $\lambda=(\lambda_1\ge\cdots\ge\lambda_N\ge0)$, padded by zeros to a fixed
length $N$, enters as boundary data
through the ``parts $+$ staircase'' point set
\begin{equation}
  \label{eq:S_lambda}
  \mathcal S(\lambda)=\{\lambda_j+N-j+1\}_{1\le j\le N}\subset\mathbb{Z}_{\ge1},
\end{equation}
a set of $N$ distinct occupied sites, the listed values strictly decreasing in $j$.
A signature with a negative part would leave the half-line $\mathbb{Z}_{\ge1}$ on which the
columns are indexed, and is excluded here. For the partition function $Z_\lambda$ we
take a single block of $n$ rows on which $\lambda$ has $N=n$ parts, feed in $n$
paths from the bottom at the positions $\mathcal S(\lambda)$, require them to exit
through the right boundary at infinity (each path eventually travels horizontally,
so all but finitely many vertices are weight-$1$ pass-throughs), and keep the top
and left boundaries empty. Summing
the product of weights from \Cref{fig:fivevertex_weights} over all admissible
five-vertex configurations with this boundary, by definition, is the
partition function $Z_\lambda(x_1,\dots,x_n)$.

\begin{proposition}[Schur polynomial as a five-vertex partition function]
\label{prop:schur_partition_function}
We have
\begin{equation*}
  s_\lambda(x_1,\dots,x_n) \;=\; Z_\lambda(x_1,\dots,x_n).
\end{equation*}
\end{proposition}
\begin{proof}[Sketch of proof]
	One can use a straightforward bijection between five-vertex configurations and semistandard Young tableaux
	which preserves the weights $x^{\mathrm{wt}(T)}$ up to reversing the variables
	$x_i\mapsto x_{n+1-i}$ (the reversal does not affect the partition function,
	which is symmetric in the $x_i$'s).
	See \Cref{fig:5v_model} for an illustration
	of a vertex model configuration corresponding to the
	semistandard Young tableau from \Cref{fig:young}.
	Note that each row of the vertex model corresponds to a horizontal strip in the tableau.
\end{proof}

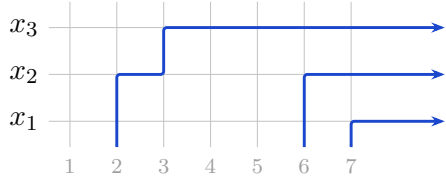
\begin{figure}[ht]
\centering
\begin{tikzpicture}[scale=0.62,>={Stealth[length=1.8mm]},rounded corners=1.2pt]
  \foreach \c in {1,...,7}{\draw[gridln] (\c,0.45)--(\c,3.55);}
  \foreach \r in {1,2,3}{\draw[gridln] (0.55,\r)--(7.45,\r);}
  \foreach \c in {1,...,7}{\node[gray!75,font=\scriptsize] at (\c,0.05) {$\c$};}
  \node[left] at (0.55,1) {$x_1$};
  \node[left] at (0.55,2) {$x_2$};
  \node[left] at (0.55,3) {$x_3$};
  \draw[uppath,->] (7,0.45)--(7,1)--(9,1);
  \draw[uppath,->] (6,0.45)--(6,2)--(9,2);
  \draw[uppath,->] (2,0.45)--(2,2)--(3,2)--(3,3)--(9,3);
\end{tikzpicture}
\caption{The five-vertex configuration corresponding to the semistandard tableau
$T$ of \Cref{fig:young} of shape $\lambda=(4,4,1)$ with $n=3$.}
\label{fig:5v_model}
\end{figure}

\begin{remark}
\label{rmk:skew}
	Fix an additional partition $\mu$ with $N$ parts and take $\lambda$ with $N+n$ parts.
	Reading the paths in a block of $n$ rows
	between the bottom boundary $\mathcal S(\lambda)$ and the top boundary
	$\mathcal S(\mu)$, with the left boundary empty and one path exiting through the right
	of each row, yields the skew Schur polynomial $s_{\lambda/\mu}$.
	The length relation is forced by the boundary flux: each of the $n$ rows releases
	exactly one path to the right, so the bottom signature carries $n$ coordinates more
	than the top one.
\end{remark}

\begin{remark}[Higher spin version]
\label{rem:higher_spin}
The free fermion five-vertex model described above
can be bijectively mapped to a free fermion instance of the higher spin six-vertex model
studied in \cite{Borodin2014vertex}, \cite{BorodinPetrov2016inhom}.
The allowed configurations in this model have $i_1,i_2\in \mathbb{Z}_{\ge0}$
and $j_1,j_2\in \left\{ 0,1 \right\}$. The vertex weights (in the Schur case)
in the $i$-th row are given by
\begin{equation*}
	w(i_1,j_1;i_2,j_2)=
	\mathbf 1_{i_1+j_1=i_2+j_2}\,x_i^{j_2}.
\end{equation*}
The configuration which bijectively corresponds to the one in \Cref{fig:5v_model} is
shown in \Cref{fig:5v_model_higher_spin}. Now the top boundary
is given by the multiset $\left\{ \lambda_1,\lambda_2,\ldots,\lambda_n  \right\}$,
and the paths enter everywhere from the left boundary, while the bottom and the right boundaries
are empty.
In contrast with the five-vertex model, the columns are now indexed by $\mathbb{Z}_{\ge0}$
(zero parts of $\lambda$ occupy column~$0$); this indexing is needed for the weights
$x_i^{j_2}$ to produce the correct total power of $x_i$ in each row.
\end{remark}

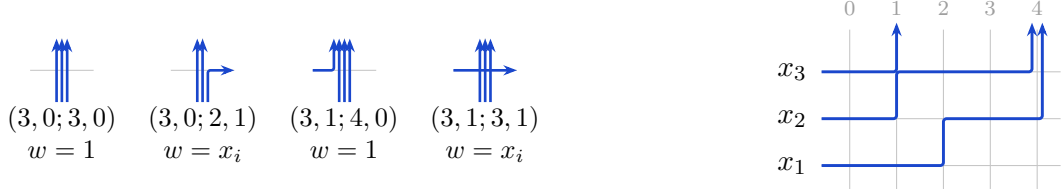
\begin{figure}[ht]
\centering
\begin{minipage}[c]{0.54\textwidth}
\centering
{\small
\begin{tabular}{cccc}
\fvv{%
\draw[upcol,line width=1pt,->] (-0.1,-0.6)--(-0.1,0.6);
\draw[upcol,line width=1pt,->] (0,-0.6)--(0,0.6);
\draw[upcol,line width=1pt,->] (0.1,-0.6)--(0.1,0.6);} &
\fvv{%
\draw[upcol,line width=1pt,->] (-0.1,-0.6)--(-0.1,0.6);
\draw[upcol,line width=1pt,->] (0,-0.6)--(0,0.6);
\draw[upcol,line width=1pt,->] (0.1,-0.6)--(0.1,0)--(0.6,0);} &
\fvv{%
\draw[upcol,line width=1pt,->] (-0.1,-0.6)--(-0.1,0.6);
\draw[upcol,line width=1pt,->] (0,-0.6)--(0,0.6);
\draw[upcol,line width=1pt,->] (0.1,-0.6)--(0.1,0.6);
\draw[upcol,line width=1pt,->] (-0.6,0)--(-0.2,0)--(-0.2,0.6);} &
\fvv{%
\draw[upcol,line width=1pt,->] (-0.1,-0.6)--(-0.1,0.6);
\draw[upcol,line width=1pt,->] (0,-0.6)--(0,0.6);
\draw[upcol,line width=1pt,->] (0.1,-0.6)--(0.1,0.6);
\draw[upcol,line width=1pt,->] (-0.6,0)--(0.6,0);} \\[3pt]
$(3,0;3,0)$ & $(3,0;2,1)$ & $(3,1;4,0)$ & $(3,1;3,1)$ \\
$w=1$ & $w=x_i$ & $w=1$ & $w=x_i$ \\
\end{tabular}
}
\end{minipage}%
\hfill
\begin{minipage}[c]{0.44\textwidth}
\centering
\begin{tikzpicture}[scale=0.62,>={Stealth[length=1.8mm]},rounded corners=1.2pt,baseline=0pt]
  \def\e{0.11}
  \foreach \c in {0,...,4}{\draw[gridln] (\c,0.5)--(\c,3.9);}
  \foreach \r in {1,2,3}{\draw[gridln] (-0.6,\r)--(4.5,\r);}
  \foreach \c in {0,...,4}{\node[gray!75,font=\scriptsize] at (\c,4.35) {$\c$};}
  \node[left] at (-0.7,1) {$x_1$};
  \node[left] at (-0.7,2) {$x_2$};
  \node[left] at (-0.7,3) {$x_3$};
  \draw[uppath,->] (-0.6,1)--(2,1)--(2,2)--(4+\e,2)--(4+\e,4.05);
  \draw[uppath,->] (-0.6,2)--(1,2)--(1,3)--(4-\e,3)--(4-\e,4.05);
  \draw[uppath,->] (-0.6,3)--(1,3)--(1,4.05);
\end{tikzpicture}
\end{minipage}
\caption{Left: the four vertex types $(j_1,j_2)\in\{0,1\}^2$, labelled
$(i_1,j_1;i_2,j_2)$, with weights $w=x_i^{j_2}$. The vertical occupation
$i_1,i_2\ge0$ is arbitrary, subject to $i_1+j_1=i_2+j_2$; here we draw the
representatives with $i_1=3$.
Right: the higher-spin six-vertex configuration corresponding to the tableau $T$
(equivalently, to the five-vertex configuration of \Cref{fig:5v_model});
the columns are indexed from~$0$.}
\label{fig:5v_model_higher_spin}
\end{figure}

\begin{remark}[The free fermion six-vertex model]
\label{rem:ff6v}
	Allowing the vertex $(1,1;1,1)$ turns the five-vertex model into the
	\emph{six-vertex model}. Keeping the free fermion condition \eqref{eq:free_fermion},
	one can define free fermion Schur functions
	\cite{ABPW2021free},
	\cite{Naprienko2023}, which generalize
	both horizontal- and vertical-strip Young tableaux
	(the vertical strip setting is detailed in \Cref{subsec:schur_vertical} below).
	This family of functions includes as
	particular cases the factorial and shifted Schur polynomials, as well as the Frobenius--Schur functions.
\end{remark}

\subsection{Skew Young diagrams and vertex model rows}
\label{subsec:skew_rows}

Drop the condition $j_1,j_2\le1$ on the horizontal edges of the higher-spin model
(\Cref{rem:higher_spin}), keeping $i_1,i_2\in\mathbb{Z}_{\ge0}$
and the arrow conservation $i_1+j_1=i_2+j_2$ at each vertex.
Consider a single row of vertices on this more general vertex model.
We read configurations of vertical arrows as multisets, and
interpret them as partitions $\lambda$ (at the top) and $\mu$ (at the bottom).
In particular, the vertical occupation $i_2$ at the top of column $x$ is the multiplicity
$m_x(\lambda)=\#\{i:\lambda_i=x\}$.
Writing $h_x$ for the
occupation of the horizontal edge entering column $x$ from the left, we have
\begin{equation*}
  h_x \;=\; \sum_{y\ge x}\bigl(m_y(\lambda)-m_y(\mu)\bigr) \;=\; \lambda'_x-\mu'_x.
\end{equation*}
The horizontal edges are thus
nonnegative if and only if
$\lambda'_x\ge\mu'_x$ for all $x$, which is equivalent to the Young
diagram inclusion $\mu\subseteq\lambda$.
\Cref{fig:skew_row} illustrates this dictionary.

\begin{figure}[ht]
\centering
\begin{tikzpicture}[scale=0.8,>={Stealth[length=1.6mm]},rounded corners=1.0pt,baseline={(current bounding box.center)}]
  \foreach \c in {1,...,7}{\draw[gridln] (\c,-0.72)--(\c,0.72);}
  \draw[gridln] (0.35,0)--(7.5,0);
  \foreach \c in {1,...,7}{\node[gray!70,font=\scriptsize] at (\c,-1.02){$\c$};}
  \node[font=\scriptsize] at (7.95,0.42){$\lambda$};
  \node[font=\scriptsize] at (7.95,-0.42){$\mu$};
  \draw[uppath,->] (0.35,0.1)--(2,0.1)--(2,0.72);                 
  \draw[uppath,->] (0.35,0)--(2.9,0)--(2.9,0.72);              
  \draw[uppath,->] (1,-0.72)--(1,-0.1)--(3.1,-0.1)--(3.1,0.72);  
  \draw[uppath,->] (4,-0.72)--(4,0.72);                        
  \draw[uppath,->] (4.2,-0.72)--(4.2,0)--(7,0)--(7,0.72);      
\end{tikzpicture}
\hspace{1.0cm}
\begin{tikzpicture}[scale=0.42,baseline={(current bounding box.center)}]
  \fill[gray!18] (0,0) rectangle (4,-2);
  \fill[gray!18] (0,-2) rectangle (1,-3);
  \fill[upcol!22] (4,0) rectangle (7,-1);
  \fill[upcol!22] (1,-2) rectangle (3,-3);
  \fill[upcol!22] (0,-3) rectangle (3,-4);
  \fill[upcol!22] (0,-4) rectangle (2,-5);
  \foreach \rr/\ww in {1/7,2/4,3/3,4/3,5/2}{\foreach \cc in {1,...,\ww}{\draw (\cc-1,-\rr+1) rectangle (\cc,-\rr);}}
\end{tikzpicture}
\caption{The dictionary between general vertex models and skew Young diagrams.}
\label{fig:skew_row}
\end{figure}
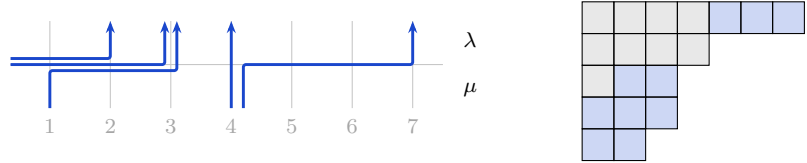

Restricting $\lambda/\mu$ to horizontal strips recovers
the higher-spin model of \Cref{rem:higher_spin}. Restricting to
vertical strips, we obtain another combinatorial condition on the allowed vertex types:
\begin{lemma}
\label{lemma:vertical_strip_condition}
For a vertex row configuration
corresponding to a skew Young diagram $\lambda/\mu$,
the inequality $j_2\le i_1$ holds at every
vertex if and only if $\lambda/\mu$ is a vertical strip.
\end{lemma}
\begin{proof}
At column $x$ the bottom vertical edge carries $i_1=m_x(\mu)=\mu'_x-\mu'_{x+1}$ and the outgoing
horizontal edge carries $j_2=h_{x+1}=\lambda'_{x+1}-\mu'_{x+1}$, so
\begin{equation*}
  j_2\le i_1
  \quad\iff\quad
  \lambda'_{x+1}-\mu'_{x+1}\le\mu'_x-\mu'_{x+1}
  \quad\iff\quad
  \lambda'_{x+1}\le\mu'_x.
\end{equation*}
This completes the proof.
\end{proof}

The vertical-strip model
is
discussed in detail in \Cref{subsec:schur_vertical} below.

\begin{remark}
\label{rmk:not_5v_model}
	The five-vertex model is not
	a direct specialization of
	the model in \Cref{rem:higher_spin},
	but it is a bijective change of coordinates from the horizontal-strip
	case, which corresponds to the shifted coordinates
	$\mathcal S(\lambda)$
	\eqref{eq:S_lambda}.
\end{remark}

\medskip

The rest of this subsection is a digression, not used later in the survey.
In the shifted coordinates \eqref{eq:S_lambda}, the six-vertex model of
\Cref{rem:ff6v} (that is, with the doubly occupied vertex $(1,1;1,1)$ allowed)
captures one more distinguished class of skew Young diagrams --- \emph{ribbons}
(border strips): connected skew diagrams containing no $2\times2$ block of boxes.

\begin{lemma}[Ribbon condition]
\label{lemma:ribbon_condition}
Let $\lambda,\mu$ be partitions with at most $N$ parts, both encoded as in
\eqref{eq:S_lambda} with this common $N$. Consider a single row of the
six-vertex model (all edges carry occupations in $\{0,1\}$, and all six
vertices of \Cref{fig:fivevertex_weights}, including $(1,1;1,1)$, are
admissible) with bottom boundary $\mathcal S(\mu)$, top boundary
$\mathcal S(\lambda)$, and empty left and right horizontal edges. An
admissible configuration exists if and only if $\mu\subseteq\lambda$ and every
connected component of $\lambda/\mu$ is a ribbon (equivalently, $\lambda/\mu$
contains no $2\times2$ block of boxes). The configuration is then unique;
see \Cref{fig:ribbon_row}.
\end{lemma}
\begin{proof}
The configuration is forced by the boundary data: the occupation $h_x$ of the
horizontal edge between columns $x$ and $x+1$ must equal
\begin{equation}
\label{eq:ribbon_h}
	h_x=\#\bigl(\mathcal S(\mu)\cap\{1,\dots,x\}\bigr)
	-\#\bigl(\mathcal S(\lambda)\cap\{1,\dots,x\}\bigr)
	=\#\{i\colon \lambda_i-i\ge x-N\}-\#\{i\colon \mu_i-i\ge x-N\},
\end{equation}
and any assignment with all $h_x\in\{0,1\}$ is admissible. Row $i$ of
$\lambda/\mu$ contains a box of content $x-N$ (content $=$ column $-$ row) if
and only if $\mu_i-i<x-N\le\lambda_i-i$. Hence $h_x\ge0$ for all $x$ is
equivalent to $\mu\subseteq\lambda$, and then $h_x$ counts the boxes of
$\lambda/\mu$ of content $x-N$. The condition $h_x\le1$ states that every
diagonal of $\lambda/\mu$ carries at most one box, which is equivalent to the
absence of a $2\times2$ block: two boxes of equal content force boxes of that
content in all intermediate rows, and two of them in adjacent rows produce a
$2\times2$ block.
\end{proof}

\begin{figure}[ht]
\centering
\begin{tikzpicture}[scale=0.72,>={Stealth[length=1.6mm]},rounded corners=1.0pt,baseline={(current bounding box.center)}]
  \foreach \c in {1,...,12}{\draw[gridln] (\c,-0.72)--(\c,0.72);}
  \draw[gridln] (0.35,0)--(12.5,0);
  \foreach \c in {1,...,12}{\node[gray!70,font=\scriptsize] at (\c,-1.02){$\c$};}
  \node[font=\scriptsize] at (13.15,0.42){$\mathcal S(\lambda)$};
  \node[font=\scriptsize] at (13.15,-0.42){$\mathcal S(\mu)$};
  \draw[uppath,->] (1,-0.72)--(1,0)--(2,0)--(2,0.72);      
  \draw[uppath,->] (6,-0.72)--(6,0.72);                    
  \draw[uppath,->] (7,-0.72)--(7,0.72);                    
  \draw[uppath,->] (4,-0.72)--(4,0)--(8,0)--(8,0.72);      
  \draw[uppath,->] (10,-0.72)--(10,0)--(12,0)--(12,0.72);  
\end{tikzpicture}
\hspace{0.6cm}
\begin{tikzpicture}[scale=0.42,baseline={(current bounding box.center)}]
  \fill[gray!18] (0,0) rectangle (5,-1);
  \fill[gray!18] (0,-1) rectangle (3,-2);
  \fill[gray!18] (0,-2) rectangle (3,-3);
  \fill[gray!18] (0,-3) rectangle (2,-4);
  \fill[upcol!22] (5,0) rectangle (7,-1);
  \fill[upcol!22] (3,-1) rectangle (4,-2);
  \fill[upcol!22] (3,-2) rectangle (4,-3);
  \fill[upcol!22] (2,-3) rectangle (4,-4);
  \fill[upcol!22] (0,-4) rectangle (1,-5);
  \node[gray!70,font=\scriptsize] at (5.5,-0.5) {$5$};
  \node[gray!70,font=\scriptsize] at (6.5,-0.5) {$6$};
  \node[gray!70,font=\scriptsize] at (3.5,-1.5) {$2$};
  \node[gray!70,font=\scriptsize] at (3.5,-2.5) {$1$};
  \node[gray!70,font=\scriptsize] at (2.5,-3.5) {$-1$};
  \node[gray!70,font=\scriptsize] at (3.5,-3.5) {$0$};
  \node[gray!70,font=\scriptsize] at (0.5,-4.5) {$-4$};
  \foreach \rr/\ww in {1/7,2/4,3/4,4/4,5/1}{\foreach \cc in {1,...,\ww}{\draw (\cc-1,-\rr+1) rectangle (\cc,-\rr);}}
\end{tikzpicture}
\caption{Left: the unique single-row six-vertex configuration with bottom
$\mathcal S(\mu)=\{1,4,6,7,10\}$ and top $\mathcal S(\lambda)=\{2,6,7,8,12\}$,
where $\mu=(5,3,3,2)$, $\lambda=(7,4,4,4,1)$, and $N=5$. Right: the
disconnected skew shape $\lambda/\mu$, a disjoint union of three ribbons of
heights $1,3,1$, with the content of each box indicated; the components do
not even share corners, being separated by the empty diagonals of contents
$-3,-2$ and $3,4$. The horizontal edge between columns $x$ and $x+1$ is
occupied precisely when $\lambda/\mu$ has a box of content $x-N$. The path
emitted at column $4$ jumps over the two stationary particles at columns $6$
and $7$ --- the doubly occupied vertices $(1,1;1,1)$ --- which match the two
boxes with a bottom neighbor (contents $1$ and $2$) in the height-$3$ ribbon,
in agreement with \Cref{rmk:ribbon_MN}.}
\label{fig:ribbon_row}
\end{figure}
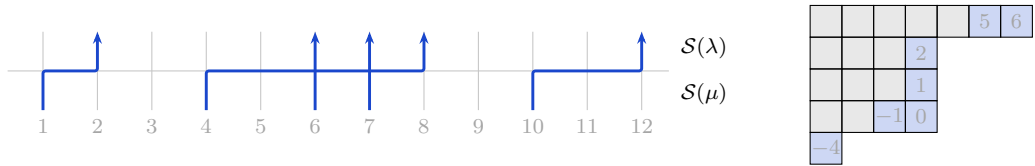

Compare \eqref{eq:ribbon_h} with the identity
$h_x=\lambda'_x-\mu'_x$ from the beginning of this subsection: in the
multiplicity coordinates the horizontal edges count the columns of
$\lambda/\mu$, while in the shifted coordinates they count the diagonals.

\begin{remark}[Murnaghan--Nakayama structure]
\label{rmk:ribbon_MN}
	In the configuration of \Cref{lemma:ribbon_condition}, the number of
	vertices of type $(1,0;0,1)$ equals the number of ribbon components of
	$\lambda/\mu$, and the number of doubly occupied vertices $(1,1;1,1)$
	equals $\sum_{r}(\mathrm{ht}(r)-1)$, where $\mathrm{ht}(r)$ is the number
	of rows of the ribbon $r$ (the doubly occupied columns correspond to the
	boxes of $\lambda/\mu$ having a bottom neighbor). Thus, assigning weight
	$-1$ to the vertex $(1,1;1,1)$ produces the signs of the
	Murnaghan--Nakayama rule \cite[Ch.~I, \S7]{Macdonald1995}. Matrix elements
	of the free fermion six-vertex transfer matrices, with weights factorized
	over the ribbons of $\lambda/\mu$, are computed in
	\cite[\S4.2]{Naprienko2023}; for the Yang--Baxter algebra of the general
	asymmetric six-vertex model in terms of ribbons, leading to cylindric
	Murnaghan--Nakayama rules for Hecke characters, see
	\cite[\S4.3]{korff2021cylindric}.
\end{remark}

\subsection{Schur vertex model II: Vertical strips}
\label{subsec:model_II}
\label{subsec:schur_vertical}

We now describe the third Schur vertex model (labeled II in the subsection
titles, since the higher-spin realization of \Cref{rem:higher_spin} counts as
the second),
which operates with the transposed Young diagram, and corresponds to
vertical-strip interlacing sequences
as described in \Cref{subsec:skew_rows} above.
This model is the $q=0$ degeneration of the $q$-Whittaker vertex model
introduced in \cite[\S6.1]{korff2013cylindric}, see \Cref{subsec:qw} below.
We
follow
\cite{MucciconiPetrov2020}.
This vertex model may be viewed as a ``dual'' to the
free fermion five-vertex model of \Cref{subsec:model_I}
(and its higher-spin equivalent version from \Cref{rem:higher_spin}).

We encode top and bottom boundary conditions of this
new vertex model by Young diagrams. Namely, to
a Young diagram $\lambda$ we associate the multiset of its column lengths via
\begin{equation}
	\label{eq:vertical_strip_coordinates_dualization}
	\lambda_i-\lambda_{i+1} = \#\{\text{paths at the $i$-th vertical edge}\},\qquad i=1,2,\ldots.
\end{equation}
For example, the Young diagram $\lambda=(4,4,1)$ corresponds to the configuration
with three vertical edges in column $2$, one vertical edge in column $3$, and no vertical edges
in the other columns.

\begin{remark}
\label{rmk:conjugate_coordinates}
Note that the convention \eqref{eq:vertical_strip_coordinates_dualization}
used for the vertical-strip model
differs from the one of \Cref{subsec:skew_rows} by conjugation.
Therefore, in the vertical-strip model, we still get \emph{horizontal-strip differences}
between consecutive rows of the vertex model.
In the Schur case this change of convention is immaterial: in either dictionary the
partition function is a skew Schur polynomial, and only the indexing shape is transposed.
However, we follow the existing literature \cite{BorodinWheelerSpinq}, \cite{MucciconiPetrov2020}
which deals with deformations of the models
where this conjugation is meaningful. We
discuss some of the deformations in \Cref{sec:qdef,sec:def_rsk} below.
\end{remark}

We now define two vertex models, exchanged by the reflection
$i_1\leftrightarrow i_2$ of the vertical axis. The \emph{up-right}
(\textcolor{upcol}{blue}) weight is
\begin{equation}
  \label{eq:upright_cross}
  W(i_1,j_1;i_2,j_2)=\mathbf 1_{i_1+j_1=i_2+j_2}\,\mathbf 1_{i_1\ge j_2}\,x^{j_2},
\end{equation}
with $x$ the spectral parameter of the row, and the \emph{down-right}
(\textcolor{downcol}{red}) weight is its reflection,
\begin{equation}
  \label{eq:downright_cross}
  W^{*}(i_1,j_1;i_2,j_2)=\mathbf 1_{i_2+j_1=i_1+j_2}\,\mathbf 1_{i_2\ge j_2}\,x^{j_2}.
\end{equation}
Here $i_1,i_2, j_1,j_2\in\mathbb{Z}_{\ge0}$.

\medskip

Encode the bottom and top boundaries of a single row by the column-length coordinates
\eqref{eq:vertical_strip_coordinates_dualization} of two partitions $\mu$ and
$\lambda$. The condition $i_1\ge j_2$ in \eqref{eq:upright_cross} (respectively $i_2\ge j_2$ in
\eqref{eq:downright_cross}) is equivalent to the requirement that $\lambda/\mu$ be a
\emph{horizontal} strip, see \Cref{lemma:vertical_strip_condition}
and \Cref{rmk:conjugate_coordinates}.
The skew functions $G_{\lambda/\mu}(x)$ and $G^{*}_{\lambda/\mu}(x)$ are defined
as the weights of the single rows with these boundary conditions
(for a single row, there is only one admissible configuration in both vertex models).

Gluing $k$ rows with spectral
parameters $y_1,\dots,y_k$ produces $G_{\lambda/\mu}(y_1,\dots,y_k)$ and
$G^{*}_{\lambda/\mu}(y_1,\dots,y_k)$
in the similar manner.
\begin{remark}
\label{rmk:infinite_boundary_at_column_0}
The horizontal arrows on the left boundary
in the models for $G_{\lambda/\mu}$ and $G^{*}_{\lambda/\mu}$ may be arbitrary.
This may be modeled by placing an infinite number of vertical edges at column $0$ with occupation $i_1=i_2=\infty$; formally, one sets $i_1=M$, lets the outgoing occupation be determined
by arrow conservation, and sends $M\to\infty$. The two conservation laws differ:
for $W$ one has $i_2=M+j_1-j_2$, while for $W^{*}$, whose law is $i_2+j_1=i_1+j_2$,
one has $i_2=M+j_2-j_1$.
Then, the weights
become $W(\infty,j_1;\infty,j_2)=x^{j_2}$ and $W^{*}(\infty,j_1;\infty,j_2)=x^{j_2}$,
which are independent of $j_1$; for \eqref{eq:upright_cross} this value is already
attained at every finite $M\ge j_2$, and for \eqref{eq:downright_cross} at every finite
$M\ge j_1$.
These boundary weights are the correct choice for our partition functions.
This convention is also equivalent to setting $\lambda_0=+\infty$, see
\eqref{eq:vertical_strip_coordinates_dualization}.
\end{remark}

Configurations contributing a nonzero weight to $G_{\lambda/\mu}$ or $G^{*}_{\lambda/\mu}$ are in bijection with
interlacing chains of partitions $\mu=\lambda^{(0)}\prec\lambda^{(1)}\prec\cdots\prec\lambda^{(k)}=\lambda$,
with the $i$-th row creating the horizontal strip $\lambda^{(i)}/\lambda^{(i-1)}$.
Unlike for the models in
\Cref{subsec:model_I}, both $G$
and $G^{*}$ are defined for any number of variables and are stable under $y_k\to0$.
On the other hand, the number of variables $k$ must be at least the length of the longest column of $\lambda/\mu$
for the weight to be nonzero.
\begin{proposition}
\label{prop:schur_vertical}
For partitions $\mu\subseteq\lambda$ and any number of variables, we have
\begin{equation*}
  G_{\lambda/\mu}(y_1,\dots,y_k)\;=\;G^{*}_{\lambda/\mu}(y_1,\dots,y_k)\;=\;s_{\lambda/\mu}(y_1,\dots,y_k),
  \qquad G_{\lambda}:=G_{\lambda/\varnothing}.
\end{equation*}
\end{proposition}
\begin{proof}
Follows by identifying vertex model configurations with chains of interlacing partitions
(differing by horizontal strips), equivalently, with semistandard Young tableaux of shape $\lambda/\mu$.
An illustration of this identification is given in \Cref{fig:upright_model}.
The vertex weights \eqref{eq:upright_cross} and \eqref{eq:downright_cross}
produce the correct weight $y^{\mathrm{wt}(T)}$ for each tableau $T$ entering
the (skew) Schur polynomial.
\end{proof}

\begin{figure}[ht]
\centering
\begin{tikzpicture}[scale=0.82,>={Stealth[length=1.8mm]},rounded corners=0.8pt,baseline={(current bounding box.center)}]
  \foreach \j in {1,...,6}{\draw[gridln] (0,\j)--(\j,\j);}
  \foreach \i in {0,...,5}{\draw[gridln] (\i,\i)--(\i,6.15);}
  \draw[downpath,line width=1.4pt] (0,0)--(6.2,6.2);
  \node[downcol,font=\scriptsize,rotate=45,anchor=north] at (4.6,4.3) {diagonal};
  \foreach \j in {1,...,5}{\node[left,font=\scriptsize] at (-0.12,\j) {$y_{\j}$};}
  \foreach \i in {1,...,4}{\node[font=\scriptsize,gray!80] at (\i,0.5) {$\i$};}
  \draw[uppath,->] (0,1)--(1,1)--(1,6.15);                                   
  \draw[uppath,->] (0,2)--(2,2)--(2,4)--(3,4)--(3,6.15);                     
  \draw[uppath,->] (0,3.1)--(1.1,3.1)--(1.1,4)--(2,4)--(2,6.15);
  \draw[uppath,->] (0,3)--(2.1,3)--(2.1,4.1)--(3.1,4.1)--(3.1,5)--(4,5)--(4,6.15);
  \draw[uppath,->] (0,5)--(1.1,5)--(1.1,6.15);
  \draw[uppath,line width=3.6pt,opacity=0.40] (0,-0.25)--(0,6.2);
  \node[font=\scriptsize] at (0,-0.55) {$\infty$};
  \node[font=\scriptsize] at (0,6.5) {$\infty$};
\end{tikzpicture}
\hspace{0.9cm}
\begin{tikzpicture}[scale=0.42,baseline={(current bounding box.center)}]
  \foreach \rr/\ww in {1/5,2/3,3/2,4/1}{\foreach \cc in {1,...,\ww}{\draw (\cc-1,-\rr+1) rectangle (\cc,-\rr);}}
  \foreach \cc/\val in {1/1,2/2,3/3,4/3,5/5}{\node at (\cc-0.5,-0.5){\small$\val$};}
  \foreach \cc/\val in {1/2,2/3,3/4}{\node at (\cc-0.5,-1.5){\small$\val$};}
  \foreach \cc/\val in {1/4,2/4}{\node at (\cc-0.5,-2.5){\small$\val$};}
  \node at (0.5,-3.5){\small$5$};
  \node at (2.5,-4.9){SSYT};
\end{tikzpicture}
\caption{A configuration of the vertex model $W$ \eqref{eq:upright_cross} and the corresponding
semistandard Young tableau. The shape is $\lambda=(5,3,2,1,0)$, and
the weight is $y_1y_2^{2}y_3^{3}y_4^{3}y_5^{2}=y^{\mathrm{wt}(T)}$.
No path in the vertex model can go below the diagonal thanks to the condition $i_1\ge j_2$ in \eqref{eq:upright_cross}. The thick edge at column $0$ carries the infinite vertical occupation (see
\Cref{rmk:infinite_boundary_at_column_0}), from which the paths peel off to the right.}
\label{fig:upright_model}
\end{figure}
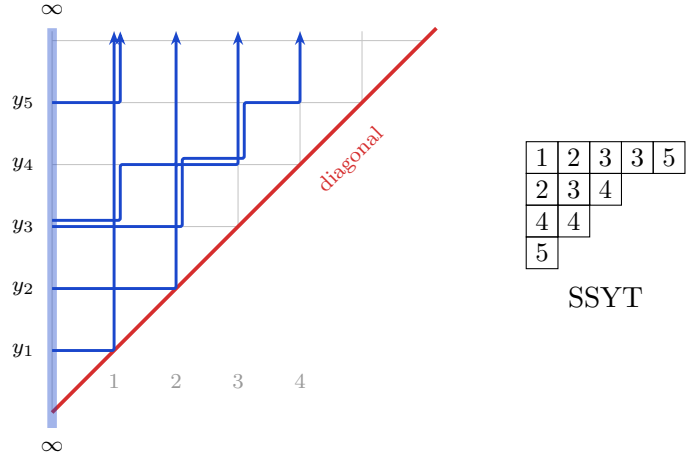

\subsection{Yang--Baxter equation and Cauchy identity}
\label{subsec:ybe_cauchy}

Let us now formulate the \emph{Yang--Baxter equation} connecting the models $W$ and $W^*$.
Define the cross-vertex weights by
\begin{equation}
\label{eq:bbR}
  \mathbb R_{x,y}(i_1,j_1;i_2,j_2)
  \;=\;
  \mathbb R_{x,y}\!\left(
  \vcenter{\hbox{\begin{tikzpicture}[scale=0.75,>={Stealth[length=2mm]},baseline=-0.4cm]
    \draw[uppath,line width=1.3pt,->] (-0.95,-0.95)--(0.95,0.95);
    \draw[downpath,line width=1.3pt,->] (-0.95,0.95)--(0.95,-0.95);
    \node[vtx] at (0,0) {};
    \node[below left]  at (-0.95,-0.95) {$i_1$};
    \node[above left]  at (-0.95,0.95)  {$j_1$};
    \node[above right] at (0.95,0.95)   {$i_2$};
    \node[below right] at (0.95,-0.95)  {$j_2$};
  \end{tikzpicture}}}
  \right)
  \;=\;
  \mathbf 1_{\,i_1+j_2=j_1+i_2\,}\,(xy)^{\min(i_2,j_2)},
\end{equation}
with $(i_1,j_1)$ the incoming and $(i_2,j_2)$ the outgoing arrow counts.

\begin{proposition}[Yang--Baxter equation]
\label{prop:ybe}
For all boundary occupations $i_1,i_2,i_3,j_1,j_2,j_3\in\mathbb Z_{\ge0}$ and
spectral parameters $x,y$, we have
\begin{multline}
  \label{eq:ybe}
  \sum_{k_1,k_2,k_3\ge0}
  \mathbb R_{x,y}(i_2,i_1;k_2,k_1)\,
  W(k_3,k_2;j_3,j_2)\,
  W^{*}(i_3,k_1;k_3,j_1)
  \\=
  \sum_{k_1',k_2',k_3'\ge0}
  W(i_3,i_2;k_3',k_2')\,
  W^{*}(k_3',i_1;j_3,k_1')\,
  \mathbb R_{x,y}(k_2',k_1';j_2,j_1),
\end{multline}
where $W=W_x$ \eqref{eq:upright_cross} and $W^{*}=W^{*}_y$
\eqref{eq:downright_cross} carry the row parameters $x$ and $y$. Graphically,
the two local vertex configurations of \Cref{fig:ybe} have equal partition
functions: the $\mathbb R$-cross may be dragged across the $W$/$W^{*}$ pair while
preserving the six boundary edges.
\end{proposition}
\begin{proof}
	Follows by a direct computation with explicit weights. A term-by-term
	verification, with the summands and their admissible ranges written out, is
	given in \Cref{lemma:schur_yb_bijection} below.
\end{proof}

\begin{figure}[ht]
\centering
\begin{tikzpicture}[scale=0.85,line cap=round,
    baseline={(current bounding box.center)},font=\small,
    every node/.style={inner sep=1.4pt}]
  \def\lw{2pt}
  \draw[uppath,line width=\lw,ybeflow]  (-1.2,-0.9)--(0.2,0);
  \draw[uppath,line width=\lw,ybeflowb] (0.2,0)--(1.6,0.9);
  \draw[uppath,line width=\lw,ybeflow]  (1.6,0.9)--(3.0,0.9);
  \draw[downpath,line width=\lw,ybeflow]  (-1.2,0.9)--(0.2,0);
  \draw[downpath,line width=\lw,ybeflowb] (0.2,0)--(1.6,-0.9);
  \draw[downpath,line width=\lw,ybeflow]  (1.6,-0.9)--(3.0,-0.9);
  \draw[downpath,line width=\lw,ybeflow] (1.6,-0.9)--(1.6,-2.0);
  \draw[downpath,line width=\lw,ybeflow] (1.6,0)--(1.6,-0.9);
  \draw[uppath,line width=\lw,ybeflow]   (1.6,0)--(1.6,0.9);
  \draw[uppath,line width=\lw,ybeflow]   (1.6,0.9)--(1.6,2.0);
  \node[circle,fill=black,inner sep=0pt,minimum size=3pt] at (0.2,0){};
  \node[circle,fill=black,inner sep=0pt,minimum size=3pt] at (1.6,0.9){};
  \node[circle,fill=black,inner sep=0pt,minimum size=3pt] at (1.6,-0.9){};
  \node[left]  at (-1.2,0.9){$i_1$};
  \node[left]  at (-1.2,-0.9){$i_2$};
  \node[above] at (1.6,2.0){$j_3$};
  \node[below] at (1.6,-2.0){$i_3$};
  \node[right] at (3.0,0.9){$j_2$};
  \node[right] at (3.0,-0.9){$j_1$};
  \node at (0.73,0.72){$k_2$};
  \node at (0.73,-0.72){$k_1$};
  \node[right] at (1.70,-0.04){$k_3$};
\end{tikzpicture}
\qquad$=$\qquad
\begin{tikzpicture}[scale=0.85,line cap=round,
    baseline={(current bounding box.center)},font=\small,
    every node/.style={inner sep=1.4pt}]
  \def\lw{2pt}
  \draw[uppath,line width=\lw,ybeflow]  (-1.2,-0.9)--(0,-0.9);
  \draw[uppath,line width=\lw,ybeflowb] (0,-0.9)--(1.5,0);
  \draw[uppath,line width=\lw,ybeflow]  (1.5,0)--(3.0,0.9);
  \draw[downpath,line width=\lw,ybeflow]  (-1.2,0.9)--(0,0.9);
  \draw[downpath,line width=\lw,ybeflowb] (0,0.9)--(1.5,0);
  \draw[downpath,line width=\lw,ybeflow]  (1.5,0)--(3.0,-0.9);
  \draw[uppath,line width=\lw,ybeflow]   (0,-2.0)--(0,-0.9);
  \draw[uppath,line width=\lw,ybeflow]   (0,-0.9)--(0,0);
  \draw[downpath,line width=\lw,ybeflow] (0,0.9)--(0,0);
  \draw[downpath,line width=\lw,ybeflow] (0,2.0)--(0,0.9);
  \node[circle,fill=black,inner sep=0pt,minimum size=3pt] at (0,0.9){};
  \node[circle,fill=black,inner sep=0pt,minimum size=3pt] at (0,-0.9){};
  \node[circle,fill=black,inner sep=0pt,minimum size=3pt] at (1.5,0){};
  \node[left]  at (-1.2,0.9){$i_1$};
  \node[left]  at (-1.2,-0.9){$i_2$};
  \node[above] at (0,2.0){$j_3$};
  \node[below] at (0,-2.0){$i_3$};
  \node[right] at (3.0,0.9){$j_2$};
  \node[right] at (3.0,-0.9){$j_1$};
  \node at (0.91,0.72){$k_1'$};
  \node at (0.91,-0.72){$k_2'$};
  \node[left] at (-0.24,0.0){$k_3'$};
\end{tikzpicture}
\caption{The Yang--Baxter equation \eqref{eq:ybe} of \Cref{prop:ybe}. The
boundary arrow counts $i_1,i_2,i_3,j_1,j_2,j_3\in\mathbb Z_{\ge0}$ are fixed, and the
internal arrow counts $k_1,k_2,k_3$ and $k_1',k_2',k_3'$ are summed over.
The vertical edge on the left or on the right serves as a source or a sink, respectively.}
\label{fig:ybe}
\end{figure}
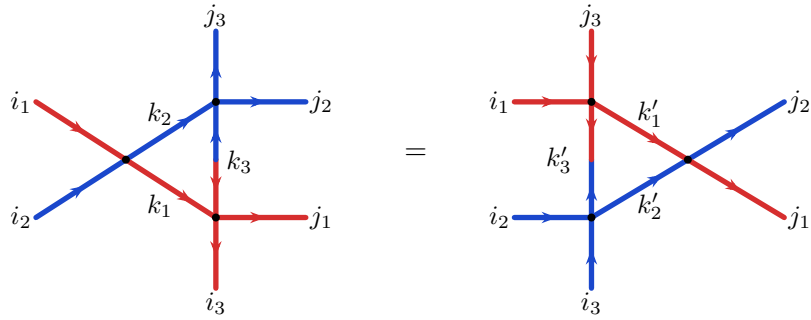

Let us now use the weights $W$ and $W^*$ to construct a vertex model for the sum
\begin{equation}
  \label{eq:cauchy_LHS}
  \sum_{\lambda} s_\lambda(x_1,\dots,x_N)\, s_\lambda(y_1,\dots,y_N),
\end{equation}
where $\lambda$ runs over all partitions with at most $N$ parts.
That is, we realize the Schur polynomial $s_\lambda(x_1,\dots,x_N)$ as the partition function of the up-right model
with $N$ rows, weights $W$, and top boundary $\lambda$; and we realize $s_\lambda(y_1,\dots,y_N)$ as the partition function of the down-right model
with $N$ rows, weights $W^*$, and bottom boundary $\lambda$. See \Cref{fig:cauchy_model} for an illustration.

\begin{remark}
\label{rmk:convergence}
The sums in the YBE \eqref{eq:ybe} are finite for any fixed finite boundary
$i_1,i_2,i_3,j_1,j_2,j_3\in\mathbb Z_{\ge0}$
(see \Cref{fig:ybe}), so
the YBE is a polynomial identity in $x,y$.
However, for the Cauchy sum
\eqref{eq:cauchy_LHS} we employ the infinite boundary condition in the zeroth column, so the sum
runs over all partitions (with at most $N$ parts), and is a power series in the $x_i,y_j$,
convergent for $|x_i|,|y_j|<1$. Alternatively, it may be treated algebraically as a formal series.
The two
viewpoints are essentially equivalent, but since we are dealing with probabilistic connections, we
adopt the analytic one.
\end{remark}

\begin{figure}[ht]
\centering
\begin{tikzpicture}[scale=0.62,>={Stealth[length=1.7mm]},rounded corners=1.0pt,
    baseline={(current bounding box.center)},font=\small]
  \foreach \y in {-4,-3,-2,-1}{\draw[gridln] (0,\y)--({\y+5},\y);}
  \foreach \c in {1,2,3,4}{\draw[gridln] (\c,{\c-5})--(\c,0);}
  \foreach \y in {1,2,3,4}{\draw[gridln] (0,\y)--({5-\y},\y);}
  \foreach \c in {1,2,3,4}{\draw[gridln] (\c,0)--(\c,{5-\c});}
  \draw[upcol,densely dashed,line width=0.9pt] (0,-5)--(5,0);
  \draw[downcol,densely dashed,line width=0.9pt] (0,5)--(5,0);
  \draw[gray!70,densely dotted,line width=0.7pt] (-0.5,0)--(5.6,0);
  \draw[uppath,line width=3pt,opacity=0.30] (0,-5)--(0,-0.05);
  \draw[downpath,line width=3pt,opacity=0.30] (0,5)--(0,0.05);
  \node[upcol] at (0,-5.62) {$\infty$};
  \node[downcol] at (0,5.62) {$\infty$};
  \draw[uppath,->] (0,-3)--(2,-3)--(2,-0.08);        
  \draw[uppath,->] (0,-2)--(2.1,-2)--(2.1,-0.08);    
  \draw[uppath,->] (0,-1)--(2.2,-1)--(2.2,-0.08);    
  \draw[uppath,->] (0,-2.9)--(1,-2.9)--(1,-1.9)--(3,-1.9)--(3,-0.08);
  \draw[downpath,->] (0,3)--(2,3)--(2,0.08);         
  \draw[downpath,->] (0,2)--(2.1,2)--(2.1,0.08);     
  \draw[downpath,->] (0,1)--(2.2,1)--(2.2,0.08);     
  \draw[downpath,->] (0,4)--(1,4)--(1,0.9)--(3,0.9)--(3,0.08);
  \node[vtx] at (2,0){};\node[vtx] at (3,0){};
  \foreach \y/\lab in {-4/{x_1},-3/{x_2},-2/{x_3},-1/{x_4}}
    {\node[anchor=east,font=\footnotesize] at (-0.18,\y) {$\lab$};}
  \foreach \y/\lab in {4/{y_1},3/{y_2},2/{y_3},1/{y_4}}
    {\node[anchor=east,font=\footnotesize] at (-0.18,\y) {$\lab$};}
  \foreach \x in {1,2,3,4}{\node[gray!75,font=\scriptsize] at (\x,-5.98) {$\x$};}
  \node[gray!55,font=\scriptsize] at (3.7,-3.4) {boundary $\varnothing$};
  \node[gray!55,font=\scriptsize] at (3.7,3.4) {boundary $\varnothing$};
  \node[downcol] at (5.95,0) {$\lambda$};
  \node[upcol,rotate=-90,anchor=center,font=\footnotesize] at (5.55,-2.7)
    {$s_\lambda(x)$, model $W$};
  \node[downcol,rotate=-90,anchor=center,font=\footnotesize] at (5.55,2.7)
    {$s_\lambda(y)$, model $W^{*}$};
\end{tikzpicture}
\caption{The vertex model for the Cauchy sum
$\sum_\lambda s_\lambda(x_1,\dots,x_N)\,s_\lambda(y_1,\dots,y_N)$
with $N=4$.}
\label{fig:cauchy_model}
\end{figure}
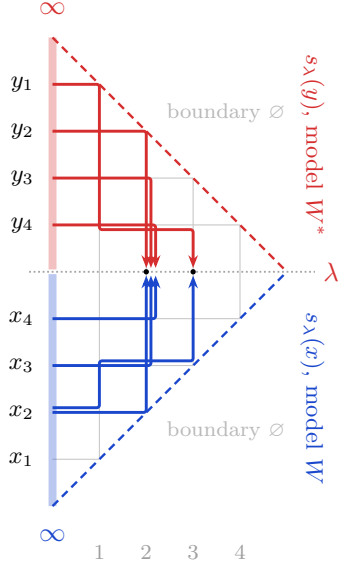

We now explain how to use the YBE to evaluate the sum \eqref{eq:cauchy_LHS}
as a finite product; the three steps are illustrated in \Cref{fig:Cauchy_proof}.
Insert an auxiliary
$N\times N$ square of the cross-vertices $\mathbb{R}_{x_i,y_j}$ to the right of the lattice in \Cref{fig:cauchy_model}.
The crosses are initially empty thanks to the boundary conditions $\varnothing$ on the right,
so this inserted square must be empty, which contributes a factor of $1$ to the partition function.

We then drag this square leftward through the lattice, one cross at a time, by the
sequence of
Yang--Baxter moves of \Cref{prop:ybe} (see \Cref{fig:ybe}). Each move preserves the
partition function while altering the lattice by exchanging one blue line with one red.
After $N^2$ such moves, one for each pair $x_i,y_j$, all the crosses have passed to the left
of the zeroth column, where they form a network fed by the empty left boundary
(See \Cref{fig:Cauchy_proof} for an illustration).
Thanks to the boundary conditions $\infty$ in the zeroth column,
the outgoing occupations of each cross can be arbitrary, and the sum over them does not
depend on the incoming ones:
\begin{equation}
  \label{eq:R_row_sum}
  \sum_{i_2,j_2\ge0}\mathbb R_{x_i,y_j}(i_1,j_1;i_2,j_2)
  =\sum_{a\ge0}(x_iy_j)^{a}
  =\frac{1}{1-x_iy_j}
  \qquad\text{for every }(i_1,j_1),
\end{equation}
because arrow conservation forces $i_2-j_2=i_1-j_1$, leaving $\min(i_2,j_2)$ to run over all
of $\mathbb Z_{\ge0}$. Summing the crosses in the reverse of the order in which they feed one
another, each cross in turn has all of its outgoing edges free when it is summed, so the
network contributes $\prod_{i,j}(1-x_iy_j)^{-1}$. Only the first cross of the network is
forced by the boundary to have incoming occupations $(0,0)$; the inner edges may carry
nonzero occupations, as displayed in \Cref{fig:crosses_left}.
Meanwhile, the lattice to the right of the crosses, with the blue and red blocks
interchanged, carries no paths at all (and thus has weight $1$): an up-right path there
would have to exit through the empty top boundary, and a down-right path through the
empty bottom one.
Thus, we obtain
\begin{proposition}
\label{prop:Cauchy_identity}
	For any $N$ and $x_1,\dots,x_N,y_1,\dots,y_N$ with $|x_i|,|y_j|<1$, we have
	\begin{equation}
	\label{eq:cauchy}
		\sum_{\lambda} s_\lambda(x_1,\dots,x_N)\, s_\lambda(y_1,\dots,y_N)=\prod_{i,j=1}^{N}\frac{1}{1-x_iy_j},
	\end{equation}
	where the sum is over all partitions $\lambda$ with at most $N$ parts.
\end{proposition}

\begin{figure}[ht]
\centering
\begin{tabular}{@{}c@{\hspace{1.5em}}c@{}}
\begin{tikzpicture}[scale=0.45,>={Stealth[length=1.7mm]},rounded corners=1.0pt,
    baseline={(current bounding box.center)},font=\small]
  \foreach \i in {1,...,5}{
    \pgfmathsetmacro{\ys}{-\i}
    \pgfmathsetmacro{\xend}{6+(\i+5)/2+0.5}
    \pgfmathsetmacro{\yend}{(5-\i)/2+0.5}
    \draw[upcol,densely dotted,thick] (0,\ys)--(6,\ys)--(\xend,\yend);
  }
  \foreach \j in {1,...,5}{
    \pgfmathsetmacro{\ys}{\j}
    \pgfmathsetmacro{\xend}{6+(5+\j)/2+0.5}
    \pgfmathsetmacro{\yend}{(\j-5)/2-0.5}
    \draw[downcol,densely dotted,thick] (0,\ys)--(6,\ys)--(\xend,\yend);
  }
  \foreach \i in {1,...,5}{
    \foreach \j in {1,...,5}{
      \pgfmathsetmacro{\xv}{6+(\i+\j)/2}
      \pgfmathsetmacro{\yv}{(\j-\i)/2}
      \node[vtx] at (\xv,\yv){};
    }
  }
  \draw[gray!70,densely dotted,line width=0.7pt] (-0.5,0)--(5.6,0);
  \draw[uppath,line width=3pt,opacity=0.30] (0,-5)--(0,-0.05);
  \draw[downpath,line width=3pt,opacity=0.30] (0,5)--(0,0.05);
  \node[upcol] at (0,-5.62) {$\infty$};
  \node[downcol] at (0,5.62) {$\infty$};
  \draw[uppath,->] (0,-3)--(2,-3)--(2,-0.08);        
  \draw[uppath,->] (0,-2)--(2.1,-2)--(2.1,-0.08);    
  \draw[uppath,->] (0,-1)--(2.2,-1)--(2.2,-0.08);    
  \draw[uppath,->] (0,-4)--(1,-4)--(1,-1.9)--(3,-1.9)--(3,-0.08);
  \draw[downpath,->] (0,3)--(2,3)--(2,0.08);         
  \draw[downpath,->] (0,2)--(2.1,2)--(2.1,0.08);     
  \draw[downpath,->] (0,1)--(2.2,1)--(2.2,0.08);     
  \draw[downpath,->] (0,4)--(1,4)--(1,0.9)--(3,0.9)--(3,0.08);
  \node[vtx] at (2,0){};\node[vtx] at (3,0){};
  \foreach \y/\lab in {-5/{x_1},-4/{x_2},-3/{x_3},-2/{x_4},-1/{x_5}}
    {\node[anchor=east,font=\footnotesize] at (-0.18,\y) {$\lab$};}
  \foreach \y/\lab in {5/{y_1},4/{y_2},3/{y_3},2/{y_4},1/{y_5}}
    {\node[anchor=east,font=\footnotesize] at (-0.18,\y) {$\lab$};}
  \foreach \x in {1,2,3,4,5}{\node[gray!75,font=\scriptsize] at (\x,-5.98) {$\x$};}
  \node[anchor=north,font=\small] at ([yshift=-2pt]current bounding box.south) {\textbf{(a)}};
\end{tikzpicture}
&\qquad
\begin{tikzpicture}[scale=0.45,>={Stealth[length=1.7mm]},rounded corners=1.0pt,
    baseline={(current bounding box.center)},font=\small]
  \foreach \i in {2,...,5}{
    \pgfmathsetmacro{\ys}{-\i}
    \pgfmathsetmacro{\xend}{6+(\i+5)/2+0.5}
    \pgfmathsetmacro{\yend}{(5-\i)/2+0.5}
    \draw[upcol,densely dotted,thick] (0,\ys)--(6,\ys)--(\xend,\yend);
  }
  \foreach \j in {2,...,5}{
    \pgfmathsetmacro{\ys}{\j}
    \pgfmathsetmacro{\xend}{6+(5+\j)/2+0.5}
    \pgfmathsetmacro{\yend}{(\j-5)/2-0.5}
    \draw[downcol,densely dotted,thick] (0,\ys)--(6,\ys)--(\xend,\yend);
  }
  \foreach \j in {2,...,5}{
    \pgfmathsetmacro{\xv}{6+(1+\j)/2}
    \pgfmathsetmacro{\yv}{(\j-1)/2}
    \node[vtx] at (\xv,\yv){};
  }
  \foreach \i in {2,...,5}{
    \foreach \j in {1,...,5}{
      \pgfmathsetmacro{\xv}{6+(\i+\j)/2}
      \pgfmathsetmacro{\yv}{(\j-\i)/2}
      \node[vtx] at (\xv,\yv){};
    }
  }
  \draw[upcol,densely dotted,thick] (-1.5,-.5)--(0,1)--(6,1)--++(.75,-.75)--++(.5,0)--++(2.25,2.25);
  \draw[downcol,densely dotted,thick] (-1.5,.5)--(0,-1)--(6,-1)--++(.75,.75)--++(.5,0)--++(2.25,-2.25);
  \node[vtx] at (-1,0) {};
  \draw[gray!70,densely dotted,line width=0.7pt] (-0.5,0)--(5.6,0);
  \draw[gray,line width=3pt,opacity=0.30] (0,-5)--(0,-1);
  \draw[gray,line width=3pt,opacity=0.30] (0,5)--(0,1);
  \draw[gray,line width=3pt,opacity=0.30] (0,-1)--(0,-.05);
  \draw[gray,line width=3pt,opacity=0.30] (0,1)--(0,.05);
  \node[upcol] at (0,-5.62) {$\infty$};
  \node[downcol] at (0,5.62) {$\infty$};
  \draw[uppath,->] (0,-2)--(2.5,-2)--(2.5,-0.08);      
  \draw[uppath,->] (0,-1.9)--(3.5,-1.9)--(3.5,-0.08);  
  \draw[uppath,->] (0,-3)--(2.6,-3)--(2.6,-0.08);      
  \draw[uppath,->] (0,-5)--(1.5,-5)--(1.5,-0.08);      
  \draw[downpath,->] (0,2)--(2.5,2)--(2.5,0.08);       
  \draw[downpath,->] (0,3)--(2.6,3)--(2.6,0.08);       
  \draw[downpath,->] (0,4)--(1.5,4)--(1.5,0.9)--(3.5,0.9)--(3.5,0.08); 
  \draw[downpath,->] (0,5)--(1.5,5)--(1.5,0.08);       
  \draw[uppath,->] (-1,0)--(0,1);          
  \draw[uppath,->] (-.9,-.1)--++(1,1);    
  \draw[downpath,->] (-.9,.1)--++(1,-1);     
  \draw[downpath,->] (-1,0)--(0.1,-1.1); 
  \node[vtx] at (0.4,0){};\node[vtx] at (1.5,0){};\node[vtx] at (2.5,0){};\node[vtx] at (3.5,0){};
  \foreach \y/\lab in {-5/{x_1},-4/{x_2},-3/{x_3},-2/{x_4}}
    {\node[anchor=east,font=\footnotesize] at (-0.18,\y) {$\lab$};}
	\node[anchor=east,font=\footnotesize] at (-1.2,-1.2) {$x_5$};
	\node[anchor=east,font=\footnotesize] at (-1.2,1.2) {$y_5$};
  \foreach \y/\lab in {5/{y_1},4/{y_2},3/{y_3},2/{y_4}}
    {\node[anchor=east,font=\footnotesize] at (-0.18,\y) {$\lab$};}
  \foreach \x in {1,2,3,4,5}{\node[gray!75,font=\scriptsize] at (\x,-5.98) {$\x$};}
  \node[anchor=north,font=\small] at ([yshift=-2pt]current bounding box.south) {\textbf{(b)}};
\end{tikzpicture}
\\[4mm]
\begin{tikzpicture}[scale=0.45,>={Stealth[length=1.7mm]},rounded corners=1.0pt,
    baseline={(current bounding box.center)},font=\small]
  \draw[upcol,densely dotted,thick] (-2.4,-1.4)--(-1,0);    
  \draw[downcol,densely dotted,thick] (-2.4,1.4)--(-1,0);   
	\def\shh{0.11}
  \draw[uppath,->] (-1,0)--(0,1.0);
  \draw[uppath,->] (-1+\shh,-\shh)--++(1,1);
  \draw[uppath,->] (-1-\shh,\shh)--++(1.0,1.0);
  \draw[downpath,->] (-1+\shh,\shh)--++(1,-1.0);
  \draw[downpath,->] (-1-\shh,-\shh)--++(1,-1.0);
  \draw[downpath,->] (-1,0)--++(1,-1.0);
  \node[vtx] at (-1,0){};
  \draw[gray,line width=3pt,opacity=0.30] (0,-2.4)--(0,2.4);
  \node[upcol] at (0,2.8) {$\infty$};
  \node[downcol] at (0,-2.8) {$\infty$};
  \node[anchor=west] at (1.2,0)
     {$\displaystyle\sum_{a\ge0}(x_iy_j)^{a}=\frac{1}{1-x_iy_j}$};
  \node[anchor=north,font=\small] at ([yshift=-2pt]current bounding box.south) {\textbf{(c)}};
\end{tikzpicture}
\\
\end{tabular}%
	\caption{Illustration of the proof of the Cauchy identity (\Cref{prop:Cauchy_identity}) by the
	YBE. \textbf{(a)} Attaching the empty $N\times N$ square of
	$\mathbb{R}_{x_i,y_j}$ vertices to the right of the lattice in
	\Cref{fig:cauchy_model}.
	\textbf{(b)} Dragging one cross through the lattice by the YBE.
	\textbf{(c)} 
	After $N^2$ moves, all crosses have reached the left boundary (see \Cref{fig:crosses_left}
	for a smaller but more detailed example).
	Each cross contributes a geometric series to the partition function.}
\label{fig:Cauchy_proof}
\end{figure}
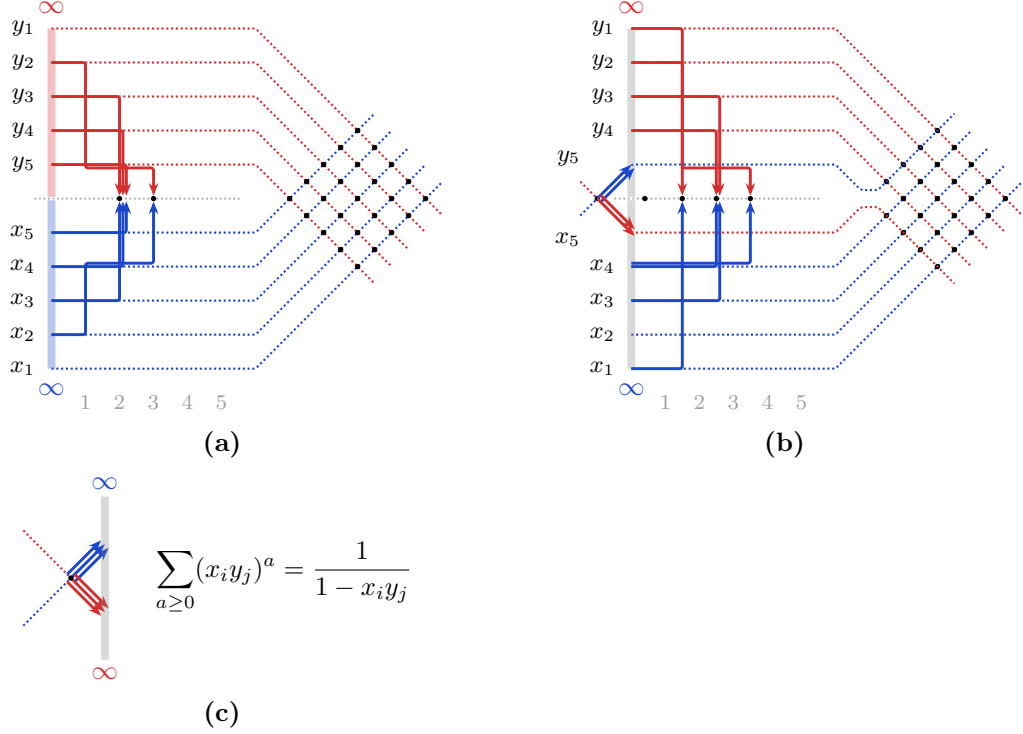

\begin{figure}[htb]
\centering
\begin{tikzpicture}[scale=1.25,>={Stealth[length=1.7mm]},rounded corners=1.0pt,
    baseline={(current bounding box.center)},font=\small]
  \def\sh{2.2}
  \foreach \i in {1,...,3}{
    \pgfmathsetmacro{\xs}{-(\i+3)/2-0.5-\sh}
    \pgfmathsetmacro{\ystart}{(3-\i)/2+0.5}
    \draw[downcol,densely dotted,thick] (\xs,\ystart)--(-\sh,-\i)--(1,-\i);
  }
  \foreach \j in {1,...,3}{
    \pgfmathsetmacro{\xs}{-(\j+3)/2-0.5-\sh}
    \pgfmathsetmacro{\ystart}{(\j-3)/2-0.5}
    \draw[upcol,densely dotted,thick] (\xs,\ystart)--(-\sh,\j)--(1,\j);
  }
  \node[draw, fill=white] at (-3.2,0) {2};
  \node[draw, fill=white] at (-3.7,0.5){0};
  \node[draw, fill=white] at (-4.2,1){1};
  \node[draw, fill=white] at (-3.7,-0.5){2};
  \node[draw, fill=white] at (-4.2,0){4};
  \node[draw, fill=white] at (-4.7,0.5){1};
  \node[draw, fill=white] at (-4.2,-1){3};
  \node[draw, fill=white] at (-4.7,-0.5){0};
  \node[draw, fill=white] at (-5.2,0){2};
  \draw[gray!70,densely dotted,line width=0.7pt] (-0.5,0)--(0.6,0);
  \draw[downpath,line width=3pt,opacity=0.30] (0,-3)--(0,-0.05);
  \draw[uppath,line width=3pt,opacity=0.30] (0,3)--(0,0.05);
  \node[downcol] at (0,-3.62) {$\infty$};
  \node[upcol] at (0,3.62) {$\infty$};
	\node at (-5.95,.75) {\scriptsize{}0};
	\node at (-5.95+.5,.75+.5) {\scriptsize{}0};
	\node at (-5.95+1,.75+1) {\scriptsize{}0};
	\node at (-5.95,-.75) {\scriptsize{}0};
	\node at (-5.95+.5,-.75-.5) {\scriptsize{}0};
	\node at (-5.95+1,-.75-1) {\scriptsize{}0};
	\node at (-5.2+.2,.33) {\scriptsize{}2};
	\node at (-5.2+.2,-.33) {\scriptsize{}2};
	\node at (-4.7+.2,.5+.33) {\scriptsize{}3};
	\node at (-4.7+.2,.5+-.33) {\scriptsize{}1};
	\node at (-4.7+.2,-.5+.33) {\scriptsize{}0};
	\node at (-4.7+.2,-.5+-.33) {\scriptsize{}2};
	\node at (-4.2+.2,-1+.33) {\scriptsize{}3};
	\node at (-4.2+.2,-1+-.33) {\scriptsize{}5};
	\node at (-4.2+.2,1+.33) {\scriptsize{}4};
	\node at (-4.2+.2,1+-.33) {\scriptsize{}1};
	\node at (-4.2+.2,.33) {\scriptsize{}4};
	\node at (-4.2+.2,-.33) {\scriptsize{}5};
	\node at (-3.7+.2,0.5+.33) {\scriptsize{}3};
	\node at (-3.7+.2,0.5+-.33) {\scriptsize{}0};
	\node at (-3.7+.2,-0.5+.33) {\scriptsize{}2};
	\node at (-3.7+.2,-0.5+-.33) {\scriptsize{}4};
	\node at (-3.2+.2,0+.33) {\scriptsize{}4};
	\node at (-3.2+.2,0+-.33) {\scriptsize{}2};
  \foreach \y/\lab in {-3/{y_3},-2/{y_2},-1/{y_1}}
    {\node[anchor=east,font=\footnotesize] at (2.88,\y) {$\lab$};}
  \foreach \y/\lab in {3/{x_3},2/{x_2},1/{x_1}}
    {\node[anchor=east,font=\footnotesize] at (2.88,\y) {$\lab$};}
\end{tikzpicture}
\caption{An example of a configuration in the final step of the
Yang--Baxter moves in the proof of the Cauchy identity (\Cref{prop:Cauchy_identity}).
The numbers along the edges indicate occupations. The framed numbers on the vertices
indicate $\min(i_2,j_2)$, which enter the weights $\mathbb{R}_{x,y}$ \eqref{eq:bbR}.
The edge occupations are uniquely determined by the framed numbers, thanks to the
arrow preservation.
The framed integers form a nonnegative integer matrix (the Robinson--Schensted--Knuth
input matrix of \Cref{prop:train_is_rsk}), and the arrows recorded by the edge
occupations are its Viennot shadow lines \cite{viennot1977forme}, in the generalization
from permutation matrices to arbitrary $\mathbb{Z}_{\ge0}$-matrices \cite{prasad2015representation}.}
\label{fig:crosses_left}
\end{figure}
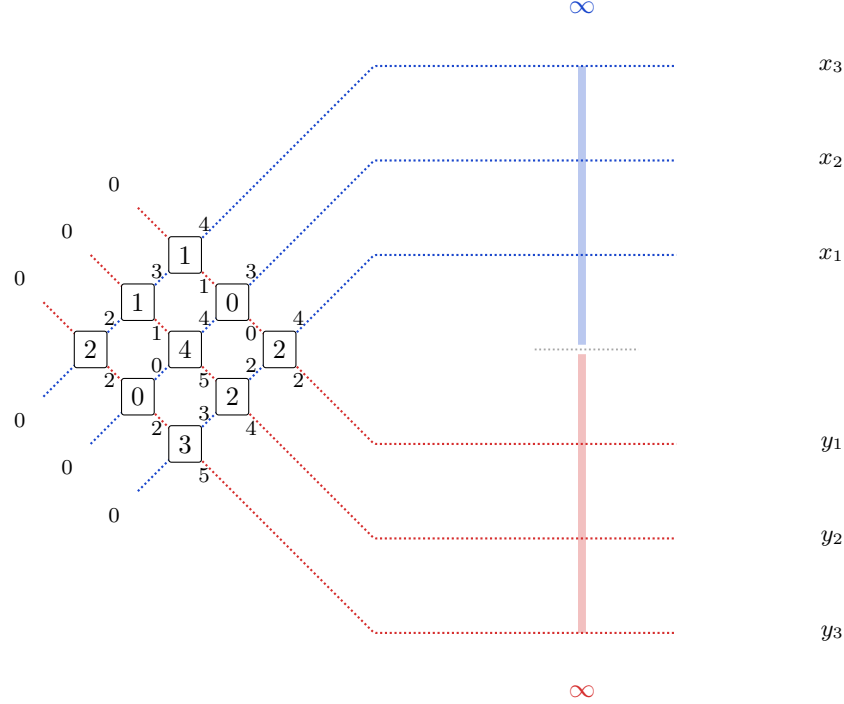

\subsection{Summary: Three kinds of Schur vertex models}

The Schur polynomial has now appeared as the partition function of three vertex
models: the five-vertex model (\Cref{subsec:model_I}),
the higher spin vertex model with capacity one of the horizontal edges (\Cref{rem:higher_spin}),
and the vertical-strip model (\Cref{subsec:model_II}) with arbitrary capacities at the edges,
but with the restriction like $j_2\le i_1$.
Each of the three classes of models is a member of a hierarchy of solvable lattice models:
\begin{enumerate}[label=$\bullet$]
  \item The five-vertex model (\Cref{subsec:model_I}) sits inside the
    \emph{free fermion six-vertex model} (\Cref{rem:ff6v},
		\cite{ABPW2021free}, \cite{Naprienko2023}).
		Furthermore, dropping the free fermion condition,
		we arrive at partition functions of general six-vertex models.
		Notably, this family also includes symmetric Grothendieck polynomials
		\cite{Motegi-Sakai13}.

  \item The higher spin vertex model (\Cref{rem:higher_spin}) was studied in connection
    with symmetric functions \cite{Borodin2014vertex}, and then in connection with
    particle systems \cite{CorwinPetrov2015}, \cite{BorodinPetrov2016inhom}.
    This lifts the Schur polynomials to the Hall--Littlewood polynomials and further to the spin Hall--Littlewood
    symmetric rational functions.

  \item The vertical-strip model (\Cref{subsec:model_II}) lifts the Schur polynomials
    to the $q$-Whittaker polynomials \cite[\S6]{korff2013cylindric} and further to the
    spin $q$-Whittaker polynomials
    \cite{BorodinWheelerSpinq}, \cite{MucciconiPetrov2020},
		\cite{BorodinKorotkikh2021inhom}, \cite{korotkikh2024representation}.
\end{enumerate}

Despite the differences, these deformations are essentially instances of one
general fully fused higher spin six-vertex model (connected to the quantum group $U_q(\widehat{\mathfrak{sl}}_2)$).
We will not describe this general model in full generality, but in \Cref{sec:qdef} below we will
discuss some of the deformations of the higher-spin and vertical-strip models, and connect them together
via the fusion procedure.

\section{Deriving RSK from the Yang--Baxter equation}
\label{sec:rsk}

We now match, term by term, the two sides of
the YBE for the Schur vertex model that we used in
\Cref{subsec:ybe_cauchy}. We will see that in this case the summands
admit a unique weight-preserving matching, and identify it with
the classical Robinson--Schensted--Knuth (RSK) row insertion algorithm.
This material first appeared in \cite[\S5]{MucciconiPetrov2020}.
First, we recall the RSK in the toggle form.

\subsection{RSK via toggles}
\label{subsec:rsk_toggles}

Classical RSK can be built entirely from local moves \cite{Pak2001hook}, 
\cite{hopkins2014rsk}, which we now explain. First, recall Fomin's
\emph{growth-diagram} form of RSK \cite{Fomin1986}, \cite{fomin1995schensted}.
Instead of building tableaux by insertion, one records the algorithm as a labeling
of the vertices of an $N\times N$ grid by partitions:
the partitions at adjacent vertices interlace (that is, differ by a horizontal strip),
and the nonnegative integer entry of the input matrix is placed inside the
corresponding unit cell. The two output tableaux $P$ and $Q$ are then read off the
chains of partitions along the top and the right boundaries of the grid.
See \Cref{fig:growth_diagram} for an illustration of the growth diagram.
A complete worked example --- the insertion algorithm running in parallel
with the corresponding vertex model states --- is presented in
\Cref{fig:rsk_full_example} at the end of this section.

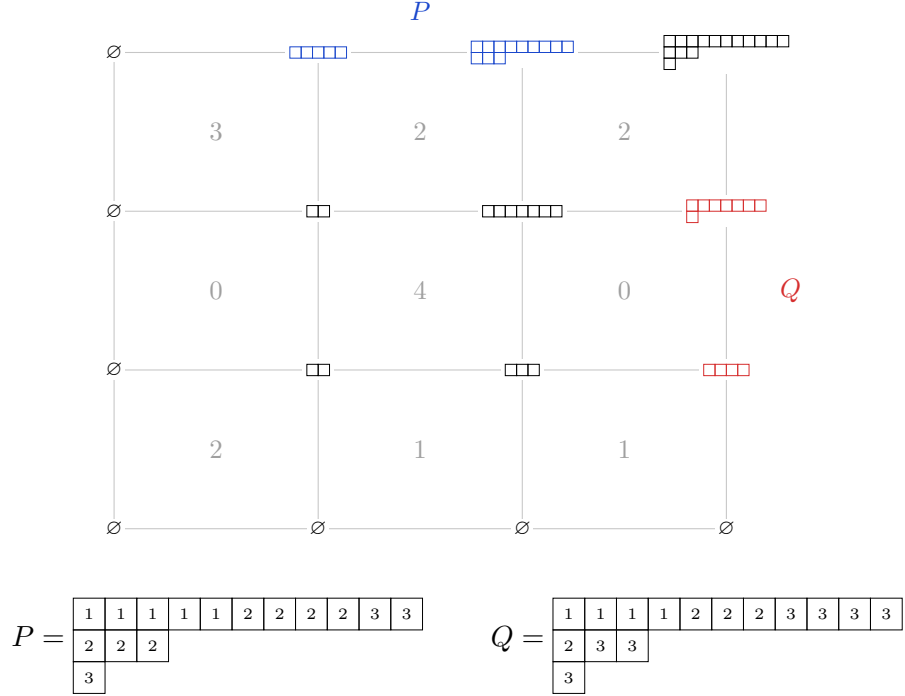
\begin{figure}[ht]
\centering
\begin{tikzpicture}[pp/.style={fill=white,inner sep=1pt,font=\scriptsize},
    yy/.style={fill=white,inner sep=1.5pt}]
  \foreach \a in {0,1,2,3}{\draw[gridln] ({2.7*\a},0)--({2.7*\a},6.3);}
  \foreach \b in {0,1,2,3}{\draw[gridln] (0,{2.1*\b})--(8.1,{2.1*\b});}
  \node[gray!75,font=\small] at (1.35,1.05) {$2$};
  \node[gray!75,font=\small] at (4.05,1.05) {$1$};
  \node[gray!75,font=\small] at (6.75,1.05) {$1$};
  \node[gray!75,font=\small] at (1.35,3.15) {$0$};
  \node[gray!75,font=\small] at (4.05,3.15) {$4$};
  \node[gray!75,font=\small] at (6.75,3.15) {$0$};
  \node[gray!75,font=\small] at (1.35,5.25) {$3$};
  \node[gray!75,font=\small] at (4.05,5.25) {$2$};
  \node[gray!75,font=\small] at (6.75,5.25) {$2$};
  \foreach \p in {(0,0),(2.7,0),(5.4,0),(8.1,0),(0,2.1),(0,4.2),(0,6.3)}
    {\node[pp] at \p {$\varnothing$};}
  \node[yy] at (2.7,2.1) {\yd{0.15}{2}};
  \node[yy] at (5.4,2.1) {\yd{0.15}{3}};
  \node[yy] at (8.1,2.1) {\yd[downcol]{0.15}{4}};
  \node[yy] at (2.7,4.2) {\yd{0.15}{2}};
  \node[yy] at (5.4,4.2) {\yd{0.15}{7}};
  \node[yy] at (8.1,4.2) {\yd[downcol]{0.15}{7,1}};
  \node[yy] at (2.7,6.3) {\yd[upcol]{0.15}{5}};
  \node[yy] at (5.4,6.3) {\yd[upcol]{0.15}{9,3}};
  \node[yy] at (8.1,6.3) {\yd{0.15}{11,3,1}};
  \node[upcol,font=\small] at (4.05,6.85) {$P$};
  \node[downcol,font=\small] at (8.95,3.15) {$Q$};
\end{tikzpicture}

\vspace{0.75cm}

$P=\;$\begin{tikzpicture}[scale=0.42,baseline={(current bounding box.center)},font=\tiny,line width=0.3pt]
  \foreach \c in {1,...,11}{\draw (\c-1,0) rectangle (\c,-1);}
  \foreach \c in {1,2,3}{\draw (\c-1,-1) rectangle (\c,-2);}
  \draw (0,-2) rectangle (1,-3);
  \foreach \c in {1,...,5}{\node at (\c-0.5,-0.5){$1$};}
  \foreach \c in {6,...,9}{\node at (\c-0.5,-0.5){$2$};}
  \foreach \c in {10,11}{\node at (\c-0.5,-0.5){$3$};}
  \foreach \c in {1,2,3}{\node at (\c-0.5,-1.5){$2$};}
  \node at (0.5,-2.5){$3$};
\end{tikzpicture}
\qquad
$Q=\;$\begin{tikzpicture}[scale=0.42,baseline={(current bounding box.center)},font=\tiny,line width=0.3pt]
  \foreach \c in {1,...,11}{\draw (\c-1,0) rectangle (\c,-1);}
  \foreach \c in {1,2,3}{\draw (\c-1,-1) rectangle (\c,-2);}
  \draw (0,-2) rectangle (1,-3);
  \foreach \c in {1,...,4}{\node at (\c-0.5,-0.5){$1$};}
  \foreach \c in {5,6,7}{\node at (\c-0.5,-0.5){$2$};}
  \foreach \c in {8,...,11}{\node at (\c-0.5,-0.5){$3$};}
  \node at (0.5,-1.5){$2$};
  \foreach \c in {2,3}{\node at (\c-0.5,-1.5){$3$};}
  \node at (0.5,-2.5){$3$};
\end{tikzpicture}
\caption{Fomin's growth diagram. Each vertex carries a partition, drawn
as its Young diagram; the left and bottom boundaries are empty ($\varnothing$);
adjacent partitions interlace by a horizontal strip.
Each cell carries a nonnegative integer (the input matrix entry),
and the top-right partition is determined by the other three corners
and this integer via the local (toggle) rule \eqref{eq:toggle}.
The top and right boundaries give the output Young tableaux $P$ and $Q$, respectively.}
\label{fig:growth_diagram}
\end{figure}

The entire diagram is generated by a single \emph{local} (\emph{toggle}) \emph{rule} that 
updates the top-right corner of a cell (denoted by $\nu$)
from the other three corners (denoted by $\lambda,\mu,\varkappa$) and the integer inside the cell
(denoted by $k$):
\begin{equation*}
\begin{tikzpicture}[scale=1.5,baseline={(current bounding box.center)},font=\small,
    pp/.style={fill=white,inner sep=1pt,font=\scriptsize}]
  \draw[gridln] (0,0) rectangle (1,1);
  \node[gray!75,font=\scriptsize] at (0.5,0.5) {$k$};
  \node[pp] at (0,1) {$\lambda$};
  \node[pp] at (1,1) {$\nu$};
  \node[pp] at (0,0) {$\varkappa$};
  \node[pp] at (1,0) {$\mu$};
\end{tikzpicture}
\end{equation*}
The local rule assumes that $\mu,\lambda$ are fixed, and it is a bijection between the 
following two sets:
\begin{equation}
  \label{eq:fomin_sets}
  \{\varkappa : \varkappa\prec\lambda,\ \varkappa\prec\mu\}\times\mathbb{Z}_{\ge0}
  \;\longleftrightarrow\;
  \{\nu : \nu\succ\lambda,\ \nu\succ\mu\},
\end{equation}
which is defined by the ``toggle'' relations
\begin{equation}
  \label{eq:toggle}
  \nu_1 = k + \max(\lambda_1,\mu_1),
  \qquad
  \nu_{c+1} = \max(\lambda_{c+1},\mu_{c+1}) + \min(\lambda_c,\mu_c) - \varkappa_c
  \quad (c\ge1).
\end{equation}
In \Cref{fig:growth_diagram_example} at the end of this section, the roles
of $\varkappa,\lambda,\mu,\nu$ are marked in one cell of the growth diagram
of the running example of \Cref{fig:rsk_full_example}.

\begin{remark}
	Note that other weight-compatible local rules --- bijections between the two
	sets in \eqref{eq:fomin_sets} respecting the grading
	$|\nu|=|\lambda|+|\mu|-|\varkappa|+k$ ---
	would correspond to other versions of the RSK correspondence, for example,
	the column insertion version. We refer to \cite{fomin1995schensted},
	\cite{fomin1995schur} for further discussion.
\end{remark}

\subsection{The deterministic matching in the Schur case}
\label{subsec:schur_yb}

The Schur YBE \eqref{eq:ybe}
for the weights \eqref{eq:upright_cross}, \eqref{eq:downright_cross} and the
cross-vertex weights \eqref{eq:bbR} admits a \emph{deterministic} term-by-term
matching of the summands on its two sides.
First, note that, as in any YBE of the form \eqref{eq:ybe},
for 
fixed boundary conditions
occupations $i_1,i_2,i_3,j_1,j_2,j_3\in\mathbb Z_{\ge0}$, both
sides of \eqref{eq:ybe} vanish unless the global balance
$i_1+j_2+j_3=i_2+i_3+j_1$ holds. Under this balance condition,
path conservation at the three
vertices leaves a single free internal coordinate on each side,
\begin{equation}
	\label{eq:schur_yb_internal_variables}
  (k_1,k_2,k_3)=\bigl(k_1,\;i_2+k_1-i_1,\;i_3+j_1-k_1\bigr),
  \qquad
  (k_1',k_2',k_3')=\bigl(k_1',\;k_1'+j_2-j_1,\;j_3+i_1-k_1'\bigr),
\end{equation}
respectively.

\begin{lemma}
\label{lemma:schur_yb_bijection}
In the Yang--Baxter equation \eqref{eq:ybe} for the Schur vertex model,
with boundary occupations satisfying the balance
$i_1+j_2+j_3=i_2+i_3+j_1$,
the summands 
in the left- and right-hand sides are,
respectively,
\begin{equation}
\label{eq:schur_yb_summands}
  x^{j_2}y^{j_1}(xy)^{\min(k_1,k_2)}
  \qquad\text{and}\qquad
  x^{k_2'}y^{k_1'}(xy)^{\min(j_1,j_2)} .
\end{equation}
These are nonzero exactly on the integer intervals
\begin{equation}\label{eq:schur_yb_k1_bounds}
  \max(0,i_1-i_2)\le k_1\le \min(i_3,\,i_3+j_1-j_2),
  \qquad
  \max(0,j_1-j_2)\le k_1'\le \min(j_3,\,j_3+i_1-i_2),
\end{equation}
which contain the same number of points.

Moreover, the map
\begin{equation}
  \label{eq:schur_yb_bijection}
  k_1'\;=\;j_1-\min(j_1,j_2)+\min(k_1,k_2)
  \;=\;k_1+\max(0,j_1-j_2)-\max(0,i_1-i_2)
\end{equation}
is a weight-preserving bijection between the weights \eqref{eq:schur_yb_summands},
and it is the only one: within each of the two lists
\eqref{eq:schur_yb_summands} the monomials are pairwise distinct, so a
weight-preserving matching of the two lists has no freedom.
\end{lemma}
\begin{proof}
The formulas \eqref{eq:schur_yb_summands} are obtained by substituting 
\eqref{eq:schur_yb_internal_variables} into the vertex weights \eqref{eq:upright_cross}, \eqref{eq:downright_cross},
and \eqref{eq:bbR}.
The occupation bounds
$k_1,k_2\ge0$ and $k_3\ge\max(j_1,j_2)$ in \eqref{eq:upright_cross},
\eqref{eq:downright_cross} (and their primed counterparts) cut out the two
desired 
intervals 
\eqref{eq:schur_yb_k1_bounds}.
Finally, note that since 
$k_1'+\min(j_1,j_2)=j_1+\min(k_1,k_2)$ and $k_2'=k_1'+j_2-j_1$, 
the map
\eqref{eq:schur_yb_bijection} turns the right-hand
summand $x^{k_2'}y^{k_1'}(xy)^{\min(j_1,j_2)}$ into the left-hand one,
$x^{j_2}y^{j_1}(xy)^{\min(k_1,k_2)}$. 
Moreover, it is a bijection between the two intervals \eqref{eq:schur_yb_k1_bounds}.
For the uniqueness, observe that $\min(k_1,k_2)=k_1+\min(0,i_2-i_1)$ and
$k_2'=k_1'+j_2-j_1$, so the exponent of $xy$ in the left-hand summand and the
exponent of $y$ in the right-hand one are strictly increasing in $k_1$ and
$k_1'$, respectively. Hence the monomials within each list are pairwise
distinct, and a weight-preserving bijection between the lists is unique.
This completes the proof.
\end{proof}

\begin{remark}
	\label{rmk:matching_is_deterministic_bijectivization}
	As discussed in the Introduction, the rigidity in
	\Cref{lemma:schur_yb_bijection} is special to the Schur weights.
	\Cref{sec:bijectivization} develops the general coupling of the two sides
	of a YBE by Markov transitions; the matching
	\eqref{eq:schur_yb_bijection} is its deterministic case, the one with
	$\mathsf p^{\mathrm{fwd}},\mathsf p^{\mathrm{bwd}}\in\{0,1\}$ in
	\eqref{eq:bijectivization}.
\end{remark}

\begin{remark}
The bijection \eqref{eq:schur_yb_bijection} takes a cleaner form in terms of the
vertical arrow counts $k_3=i_3+j_1-k_1$, $k_3'=j_3+i_1-k_1'$: 
\begin{equation}
  \label{eq:schur_yb_bijection_k3}
  k_3'\;=\;k_3+\max(i_1,i_2)-\max(j_1,j_2).
\end{equation}
\end{remark}

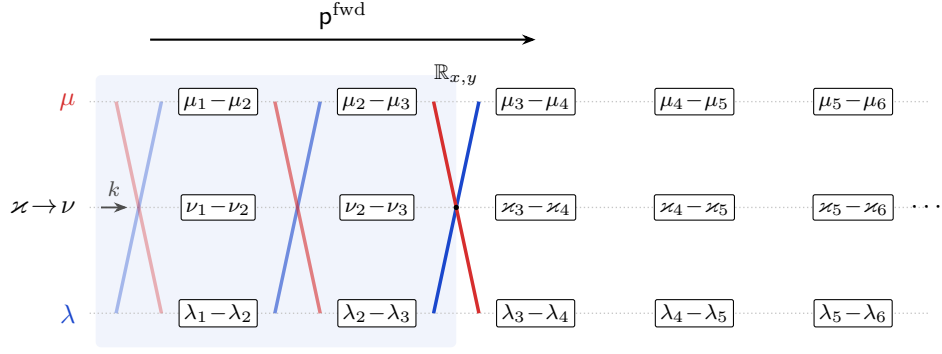
\begin{figure}[ht]
\centering
\begin{tikzpicture}[>={Stealth[length=1.8mm]},font=\small,
    box/.style={draw,rounded corners=0.7pt,fill=white,inner sep=1.8pt,font=\scriptsize},
    rail/.style={gray!45,line width=0.5pt,densely dotted},
    baseline={(current bounding box.center)}]
  \fill[upcol!6,rounded corners=2pt] (-0.62,-0.45) rectangle (4.15,3.15);
  \foreach \y in {0,1.4,2.8}{\draw[rail] (-0.7,\y)--(10.4,\y);}
  \node[downcol,left] at (-0.75,2.8) {$\mu$};
  \node[left]         at (-0.75,1.4) {$\varkappa\!\to\!\nu$};
  \node[upcol,left]   at (-0.75,0)   {$\lambda$};
  \def\ca{1.0}\def\cb{3.1}\def\cc{5.2}\def\cd{7.3}\def\ce{9.4}
  \node[box] at (\ca,2.8){$\mu_{1}\!-\!\mu_{2}$};\node[box] at (\cb,2.8){$\mu_{2}\!-\!\mu_{3}$};
  \node[box] at (\cc,2.8){$\mu_{3}\!-\!\mu_{4}$};\node[box] at (\cd,2.8){$\mu_{4}\!-\!\mu_{5}$};
  \node[box] at (\ce,2.8){$\mu_{5}\!-\!\mu_{6}$};
  \node[box] at (\ca,0){$\lambda_{1}\!-\!\lambda_{2}$};\node[box] at (\cb,0){$\lambda_{2}\!-\!\lambda_{3}$};
  \node[box] at (\cc,0){$\lambda_{3}\!-\!\lambda_{4}$};\node[box] at (\cd,0){$\lambda_{4}\!-\!\lambda_{5}$};
  \node[box] at (\ce,0){$\lambda_{5}\!-\!\lambda_{6}$};
  \node[box] at (\ca,1.4){$\nu_{1}\!-\!\nu_{2}$};\node[box] at (\cb,1.4){$\nu_{2}\!-\!\nu_{3}$};
  \node[box] at (\cc,1.4){$\varkappa_{3}\!-\!\varkappa_{4}$};\node[box] at (\cd,1.4){$\varkappa_{4}\!-\!\varkappa_{5}$};
  \node[box] at (\ce,1.4){$\varkappa_{5}\!-\!\varkappa_{6}$};
  \draw[uppath,line width=1.3pt,opacity=0.35]   (-0.35,0)--(0.25,2.8);
  \draw[downpath,line width=1.3pt,opacity=0.35] (-0.35,2.8)--(0.25,0);
  \draw[uppath,line width=1.3pt,opacity=0.60]   (1.75,0)--(2.35,2.8);
  \draw[downpath,line width=1.3pt,opacity=0.60] (1.75,2.8)--(2.35,0);
  \draw[uppath,line width=1.3pt]   (3.85,0)--(4.45,2.8);
  \draw[downpath,line width=1.3pt] (3.85,2.8)--(4.45,0);
  \node[vtx] at (4.15,1.4){};
  \node[font=\scriptsize] at (4.15,3.16){$\mathbb R_{x,y}$};
  \draw[->,line width=0.8pt] (0.1,3.62)--(5.2,3.62);
  \node[above] at (2.65,3.62){$\mathsf p^{\mathrm{fwd}}$};
  \draw[->,gray!60!black,line width=0.7pt] (-0.55,1.4)--(-0.2,1.4);
  \node[gray!55!black,font=\scriptsize,above] at (-0.38,1.42){$k$};
  \node at (10.4,1.4){$\cdots$};
\end{tikzpicture}
\caption{Dragging the Yang--Baxter cross through two rows 
$\mu/\varkappa$ (top) and $\lambda/\varkappa$ (bottom) of a growth diagram.
Each vertical edge carries a column
multiplicity $\square_c-\square_{c+1}$.
The middle row is rewritten
$\varkappa\mapsto\nu$ (the shaded columns are already updated). 
The only free input is
$k\in\mathbb Z_{\ge0}$ at the first column,
which corresponds to the state of the cross
vertex $\mathbb{R}_{x,y}$ to the left of the
zeroth column (the simplest example is in \Cref{fig:Cauchy_proof}, (c)).
We discuss this free input in detail at the end of the proof of \Cref{prop:train_is_rsk}.
The arrow marks the direction of the drag; its label $\mathsf p^{\mathrm{fwd}}$
is the forward transition probability of \Cref{sec:bijectivization}, which for
the Schur weights takes only the values $0$ and $1$.}
\label{fig:rsk_dragging}
\end{figure}

\begin{proposition}
\label{prop:train_is_rsk}
Dragging the cross \eqref{eq:bbR} rightward across the two rows $\lambda/\varkappa$
and $\mu/\varkappa$ of \Cref{fig:rsk_dragging}, and rewriting the middle rail
$\varkappa\mapsto\nu$ one column at a time by the deterministic move of
\Cref{lemma:schur_yb_bijection}, reproduces the local rule \eqref{eq:toggle}.
Consequently, the Yang--Baxter moves over the $N\times N$ grid 
(\Cref{fig:growth_diagram}) produce the Robinson--Schensted--Knuth 
correspondence.\footnote{This RSK correspondence is the most classical one,
based on row insertions. It would be interesting to see if the other three bijections 
(column insertion, dual row insertion, dual column insertion) can be derived from the Yang--Baxter equation in a similar way.}
\end{proposition}
\begin{proof}
Fix a column $c\ge1$. The three rails give the vertical occupations
\begin{equation}
	\label{eq:rsk_vertical_dictionary}
  i_3=\lambda_c-\lambda_{c+1},\quad
  j_3=\mu_c-\mu_{c+1},\quad
  k_3=\varkappa_c-\varkappa_{c+1},\quad
  k_3'=\nu_c-\nu_{c+1},
\end{equation}
the middle edge reading $\varkappa$ ahead of the cross and $\nu$ behind it. Each
horizontal leg carries the number of paths crossing it, a partial sum of the
strip multiplicities: the two legs entering the cross carry
$i_1=\nu_c-\mu_c$ and $i_2=\nu_c-\lambda_c$, and the two leaving it carry
$j_1=\lambda_{c+1}-\varkappa_{c+1}$ and $j_2=\mu_{c+1}-\varkappa_{c+1}$,
all nonnegative since $\varkappa\prec\lambda$, $\varkappa\prec\mu$,
$\lambda\prec\nu$, and $\mu\prec\nu$.
Hence
\begin{equation}
	\label{eq:rsk_horizontal_maxima}
  \max(i_1,i_2)=\nu_c-\min(\lambda_c,\mu_c),
  \qquad
  \max(j_1,j_2)=\max(\lambda_{c+1},\mu_{c+1})-\varkappa_{c+1}.
\end{equation}
Substituting \eqref{eq:rsk_vertical_dictionary}--\eqref{eq:rsk_horizontal_maxima}
into the $k_3$-identity \eqref{eq:schur_yb_bijection_k3} and cancelling $\nu_c$ and
$\varkappa_{c+1}$ gives
\begin{equation*}
  \nu_{c+1}=\max(\lambda_{c+1},\mu_{c+1})+\min(\lambda_c,\mu_c)-\varkappa_c,
\end{equation*}
which is precisely the toggle \eqref{eq:toggle} for $c\ge1$.

The case $c=0$ is slightly different, because here 
each Yang--Baxter cross carries the free input (an element of the integer input matrix for the RSK). 
This input matrix is encoded by the 
$N\times N$ grid as in \Cref{fig:growth_diagram}. 
Assume that we need to insert the cross 
$\mathbb{R}_{x_a,y_b}$, $1\le a,b\le N$, through the 
zeroth column of the grid which carries infinitely many vertical paths. 
Let this cross carry the parameter $k_{ab}\in\mathbb Z_{\ge0}$.
Assume that the input arrows to $\mathbb{R}_{x_a,y_b}$ are $i_1$ and $i_2$, see
\Cref{fig:rsk_dragging_0}, and $k_{ab}=\min(k_1,k_2)$.
This means that 
\begin{equation*}
	k_1=k_{ab}+\max(0,i_1-i_2),\qquad k_2=k_{ab}+\max(0,i_2-i_1).
\end{equation*}

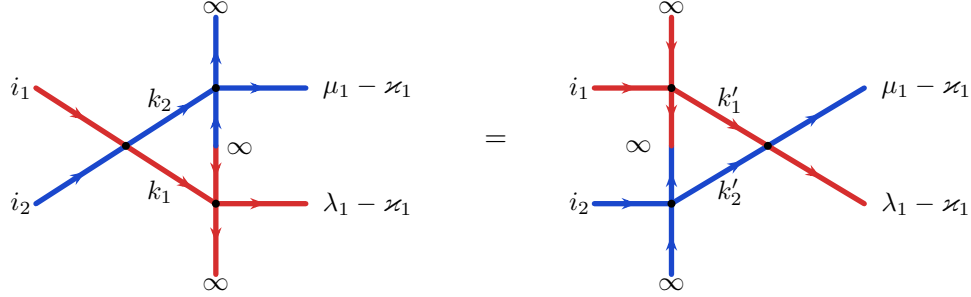
\begin{figure}[htpb]
\centering
\begin{tikzpicture}[scale=0.85,line cap=round,
    baseline={(current bounding box.center)},font=\small,
    every node/.style={inner sep=1.4pt}]
  \def\lw{2pt}
  \draw[uppath,line width=\lw,ybeflow]  (-1.2,-0.9)--(0.2,0);
  \draw[uppath,line width=\lw,ybeflowb] (0.2,0)--(1.6,0.9);
  \draw[uppath,line width=\lw,ybeflow]  (1.6,0.9)--(3.0,0.9);
  \draw[downpath,line width=\lw,ybeflow]  (-1.2,0.9)--(0.2,0);
  \draw[downpath,line width=\lw,ybeflowb] (0.2,0)--(1.6,-0.9);
  \draw[downpath,line width=\lw,ybeflow]  (1.6,-0.9)--(3.0,-0.9);
  \draw[downpath,line width=\lw,ybeflow] (1.6,-0.9)--(1.6,-2.0);
  \draw[downpath,line width=\lw,ybeflow] (1.6,0)--(1.6,-0.9);
  \draw[uppath,line width=\lw,ybeflow]   (1.6,0)--(1.6,0.9);
  \draw[uppath,line width=\lw,ybeflow]   (1.6,0.9)--(1.6,2.0);
  \node[circle,fill=black,inner sep=0pt,minimum size=3pt] at (0.2,0){};
  \node[circle,fill=black,inner sep=0pt,minimum size=3pt] at (1.6,0.9){};
  \node[circle,fill=black,inner sep=0pt,minimum size=3pt] at (1.6,-0.9){};
  \node[left]  at (-1.2,0.9){$i_1$};
  \node[left]  at (-1.2,-0.9){$i_2$};
  \node[above] at (1.6,2.0){$\infty$};
  \node[below] at (1.6,-2.0){$\infty$};
  \node[right] at (3.2,0.9){$\mu_1-\varkappa_1$};
  \node[right] at (3.2,-0.9){$\lambda_1-\varkappa_1$};
  \node at (0.73,0.72){$k_2$};
  \node at (0.73,-0.72){$k_1$};
  \node[right] at (1.70,-0.04){$\infty$};
\end{tikzpicture}
\qquad$=$\qquad
\begin{tikzpicture}[scale=0.85,line cap=round,
    baseline={(current bounding box.center)},font=\small,
    every node/.style={inner sep=1.4pt}]
  \def\lw{2pt}
  \draw[uppath,line width=\lw,ybeflow]  (-1.2,-0.9)--(0,-0.9);
  \draw[uppath,line width=\lw,ybeflowb] (0,-0.9)--(1.5,0);
  \draw[uppath,line width=\lw,ybeflow]  (1.5,0)--(3.0,0.9);
  \draw[downpath,line width=\lw,ybeflow]  (-1.2,0.9)--(0,0.9);
  \draw[downpath,line width=\lw,ybeflowb] (0,0.9)--(1.5,0);
  \draw[downpath,line width=\lw,ybeflow]  (1.5,0)--(3.0,-0.9);
  \draw[uppath,line width=\lw,ybeflow]   (0,-2.0)--(0,-0.9);
  \draw[uppath,line width=\lw,ybeflow]   (0,-0.9)--(0,0);
  \draw[downpath,line width=\lw,ybeflow] (0,0.9)--(0,0);
  \draw[downpath,line width=\lw,ybeflow] (0,2.0)--(0,0.9);
  \node[circle,fill=black,inner sep=0pt,minimum size=3pt] at (0,0.9){};
  \node[circle,fill=black,inner sep=0pt,minimum size=3pt] at (0,-0.9){};
  \node[circle,fill=black,inner sep=0pt,minimum size=3pt] at (1.5,0){};
  \node[left]  at (-1.2,0.9){$i_1$};
  \node[left]  at (-1.2,-0.9){$i_2$};
  \node[above] at (0,2.0){$\infty$};
  \node[below] at (0,-2.0){$\infty$};
  \node[right] at (3.2,0.9){$\mu_1-\varkappa_1$};
  \node[right] at (3.2,-0.9){$\lambda_1-\varkappa_1$};
  \node at (0.91,0.72){$k_1'$};
  \node at (0.91,-0.72){$k_2'$};
  \node[left] at (-0.24,0.0){$\infty$};
\end{tikzpicture}
\caption{Dragging the cross through the zeroth column.}
\label{fig:rsk_dragging_0}
\end{figure}

After dragging the cross, we have $k_1'=\nu_1-\mu_1$
and $k_2'=\nu_1-\lambda_1$. Comparing the weights on the left and on the right
in \Cref{fig:rsk_dragging_0}, we see that 
\begin{equation*}
	\nu_1=k_{ab}+\max(\lambda_1,\mu_1),
\end{equation*}
which does not depend on the incoming occupations $i_1,i_2$.

Thus, the data $k_{ab}=\min(k_1,k_2)$ at each cross $\mathbb{R}_{x_a,y_b}$ 
in \Cref{fig:crosses_left} is exactly the RSK input matrix.
This completes the proof of \Cref{prop:train_is_rsk}.
\end{proof}

\Cref{fig:rsk_biword} fixes the input matrix $M$ of the running example and
its biword. \Cref{fig:rsk_full_example} illustrates the resulting
correspondence step by
step, with the bumping cascades on the RSK side matched to the path hops on
the seam of the vertex model. \Cref{fig:growth_diagram_example} displays the
growth diagram of the same input matrix.

\begin{figure}[ht]
\centering
\includegraphics{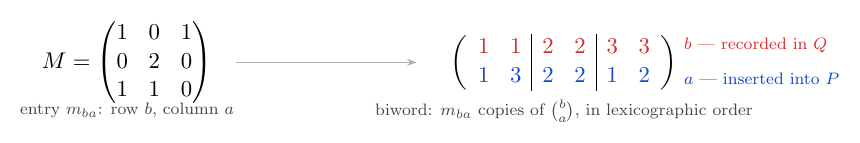}
\caption{The input matrix $M$ of the running example and its biword. The top
row of the biword (the $b$'s, red) is recorded in $Q$, and the bottom row
(the $a$'s, blue) is inserted into $P$.}
\label{fig:rsk_biword}
\end{figure}

\begin{figure}[ht]
\centering
\begin{tikzpicture}[pp/.style={fill=white,inner sep=1pt,font=\scriptsize},
    yy/.style={fill=white,inner sep=1pt},
    corner/.style={font=\tiny,gray!50!black}]
  \fill[gray!12] (1.37,1.37) rectangle (2.38,2.38);
  \foreach \a in {0,1,2,3}{\draw[gridln] ({1.25*\a},0)--({1.25*\a},3.75);}
  \foreach \b in {0,1,2,3}{\draw[gridln] (0,{1.25*\b})--(3.75,{1.25*\b});}
  \node[gray!75,font=\scriptsize] at (0.625,0.625) {$1$};
  \node[gray!75,font=\scriptsize] at (1.875,0.625) {$0$};
  \node[gray!75,font=\scriptsize] at (3.125,0.625) {$1$};
  \node[gray!75,font=\scriptsize] at (0.625,1.875) {$0$};
  \node[gray!75,font=\scriptsize] at (1.875,1.875) {$2$};
  \node[gray!75,font=\scriptsize] at (3.125,1.875) {$0$};
  \node[gray!75,font=\scriptsize] at (0.625,3.125) {$1$};
  \node[gray!75,font=\scriptsize] at (1.875,3.125) {$1$};
  \node[gray!75,font=\scriptsize] at (3.125,3.125) {$0$};
  \node[corner] at (1.48,1.44) {$\varkappa$};
  \node[corner] at (1.48,2.31) {$\lambda$};
  \node[corner] at (2.27,1.44) {$\mu$};
  \node[corner] at (2.27,2.31) {$\nu$};
  \foreach \p in {(0,0),(1.25,0),(2.5,0),(3.75,0),(0,1.25),(0,2.5),(0,3.75)}
    {\node[pp] at \p {$\varnothing$};}
  \node[yy] at (1.25,1.25) {\yd{0.15}{1}};
  \node[yy] at (2.5,1.25) {\yd{0.15}{1}};
  \node[yy] at (3.75,1.25) {\yd[downcol]{0.15}{2}};
  \node[yy] at (1.25,2.5) {\yd{0.15}{1}};
  \node[yy] at (2.5,2.5) {\yd{0.15}{3}};
  \node[yy] at (3.75,2.5) {\yd[downcol]{0.15}{3,1}};
  \node[yy] at (1.25,3.75) {\yd[upcol]{0.15}{2}};
  \node[yy] at (2.5,3.75) {\yd[upcol]{0.15}{4,1}};
  \node[yy] at (3.75,3.75) {\yd{0.15}{4,1,1}};
  \node[upcol,font=\small] at (1.875,4.2) {$P$};
  \node[downcol,font=\small] at (4.45,1.875) {$Q$};
\end{tikzpicture}
\caption{The growth diagram for the $3\times3$ input matrix of
\Cref{fig:rsk_biword}; the top and right boundary chains are the
$P$- and $Q$-chains. In the shaded cell, the corners play the roles
$\varkappa=(1)$, $\lambda=(1)$, $\mu=(1)$ in the local rule
\eqref{eq:toggle} with $k=2$, which yields $\nu=(3)$.}
\label{fig:growth_diagram_example}
\end{figure}
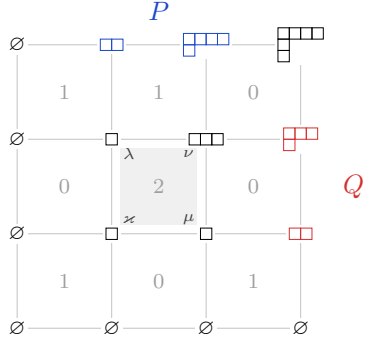

\begin{figure}[p]
\centering
\includegraphics[width=\textwidth,height=0.8\textheight,keepaspectratio]{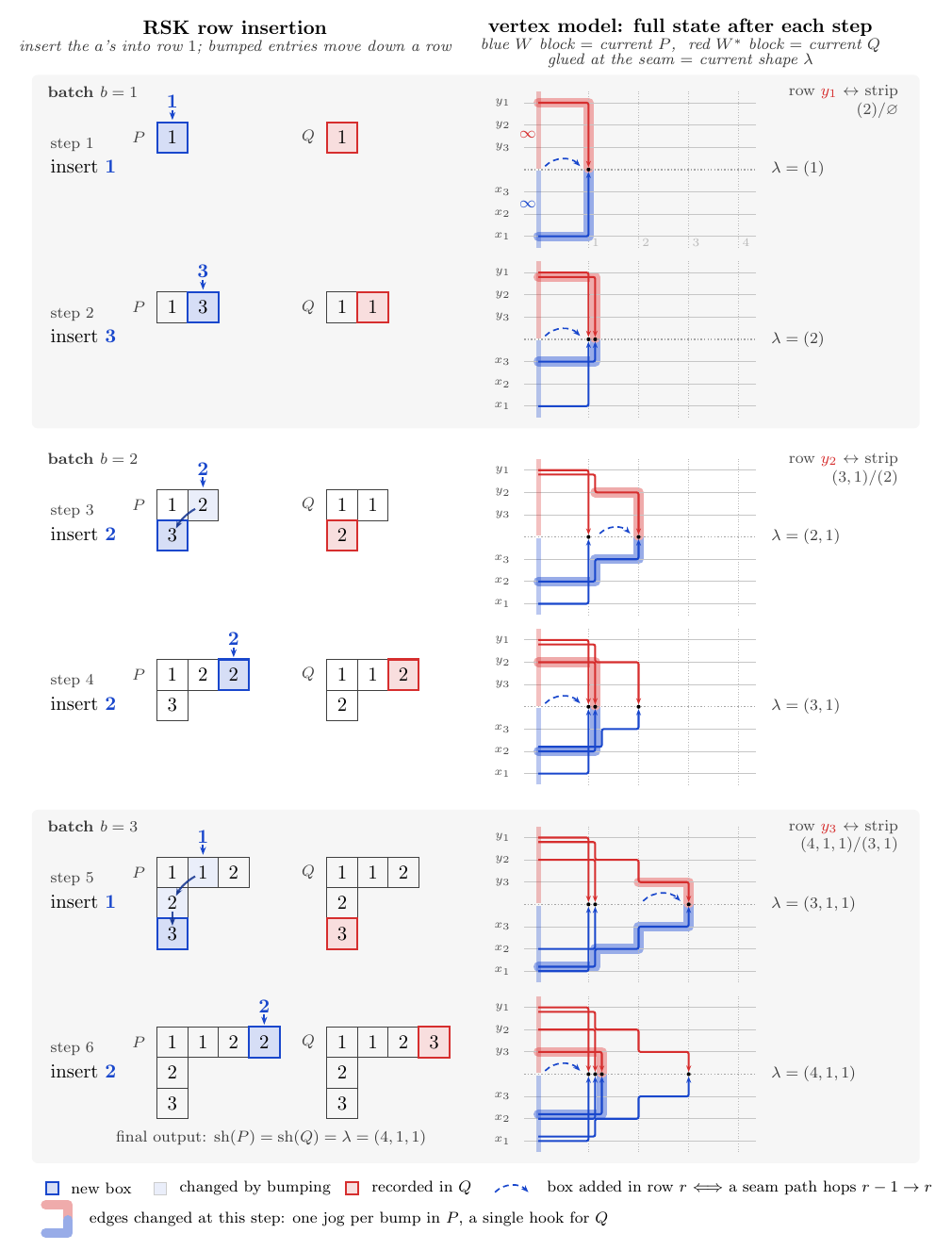}
\caption{RSK insertion and the Schur vertex model running in parallel on the
input matrix $M$ of \Cref{fig:rsk_biword}. After each step, the blue block (rows $x_a$, weights $W$
\eqref{eq:upright_cross}) encodes the current $P$, and the red block (rows
$y_b$, weights $W^{*}$ \eqref{eq:downright_cross}) encodes the current $Q$.
The blocks are glued along the seam carrying the common shape $\lambda$ in the column
coordinates \eqref{eq:vertical_strip_coordinates_dualization}. A new box in
row $r$ of $\lambda$ moves one seam path from column $r-1$ to column $r$
(dashed arcs). Shaded corridors mark the edges that change at each step ---
the trail of the cross $\mathbb{R}_{x_a,y_b}$ being dragged in the proof of
\Cref{prop:train_is_rsk}: the blue corridor enters at row $x_a$ and jogs
once per entry bumped in $P$, while the red corridor is a single hook at
row $y_b$, as $Q$ only gains one box.}
\label{fig:rsk_full_example}
\end{figure}

\section{The combinatorial 3D $R$ and the tetrahedron equation}
\label{sec:3dR}

The matching of \Cref{lemma:schur_yb_bijection} is the
only weight-preserving bijection between the summands of
the two sides of \eqref{eq:ybe}. A reader who has met
\emph{set-theoretic} solutions of the Yang--Baxter
equation will ask whether this matching is one of them.
In this section we answer this question. As a map on the
arrow occupations of a single cross, the Schur matching
is not a binary Yang--Baxter map at all. In three
coordinates read off the monomial weights, however, it
becomes the standard \emph{combinatorial
three-dimensional $R$}, a map of triples whose
consistency equation is the Zamolodchikov tetrahedron
equation --- one dimension above the Yang--Baxter equation
\eqref{eq:ybe}. Moreover, we also explain how ordinary
set-theoretic Yang--Baxter maps reappear from this map of
triples after a periodic closure, see
\Cref{subsubsec:third_periodic} below.

\subsection{Is the Schur matching a set-theoretic Yang--Baxter map?}
\label{subsec:set_theoretic_question}

We refer to
\cite{ReshetikhinVeselov2005PoissonLieYB} and
\cite{KouloukasPapageorgiou2011Entwining} for background
on set-theoretic Yang--Baxter maps which are defined as follows.

\begin{definition}
\label{def:set_theoretic_yb_map}
Let $X$ be a set. A \emph{set-theoretic Yang--Baxter map} is a map
$r\colon X\times X\to X\times X$ satisfying the braid relation
\begin{equation}
  \label{eq:set_theoretic_ybe}
  r_{12}\,r_{23}\,r_{12}=r_{23}\,r_{12}\,r_{23}
  \qquad\text{on }X^3,
\end{equation}
where $r_{ab}$ applies $r$ to the factors $a$ and $b$ and acts as the
identity on the remaining one.
\end{definition}

\begin{remark}
The references above write 
\eqref{eq:set_theoretic_ybe}
in the ``quantum form''
$R_{12}R_{13}R_{23}=R_{23}R_{13}R_{12}$. The braid form
\eqref{eq:set_theoretic_ybe} used here corresponds to it via $r=P\circ R$,
where $P$ is the flip of the two factors.
\end{remark}

\begin{definition}
\label{def:colored_set_theoretic_yb_map}
There exists a colored version of \Cref{def:set_theoretic_yb_map},
in which
one is given sets
$X_l$ and maps $r_{l,m}\colon X_l\times X_m\to X_m\times X_l$ satisfying the
evident analogue of \eqref{eq:set_theoretic_ybe} on
$X_k\times X_l\times X_m$.
\end{definition}

There are two natural
candidates for a binary map
coming out of the Schur matching described in
\Cref{sec:rsk}. Neither of them is a set-theoretic Yang--Baxter map:
\begin{enumerate}[label=$\bullet$]
\item The cross vertex $\mathbb R_{x,y}$ itself, read as a
  map from its two incoming occupations to its two
	outgoing ones is not a function, since it assigns weights to 
	multiple outgoing states for a single incoming state.
\item The deterministic matching
  \eqref{eq:schur_yb_bijection} of
  \Cref{lemma:schur_yb_bijection}, which sends the free
  internal occupation $k_1$ on the left-hand side of the
  YBE to the free internal occupation $k_1'$ on the
  right-hand side, for fixed boundary occupations. It is a
  bijection for each fixed boundary, but not a map on edge
  states.
\end{enumerate}
However,
we do not claim that no other packaging of the Schur
matching as a set-theoretic Yang--Baxter map exists. 
Instead,
in the next \Cref{subsec:hidden_3d_r}
we explain how the Schur matching possesses a
related three-dimensional structure.

\subsection{The hidden three-dimensional $R$}
\label{subsec:hidden_3d_r}

The three coordinates to keep 
for the desired three-dimensional structure
are the ones that carry the monomial exponents in
\eqref{eq:schur_yb_summands}. Set
\begin{equation}
  \label{eq:3d_R_coordinates}
  a=j_1,\qquad h=\min(k_1,k_2),\qquad c=j_2,\qquad m=\min(j_1,j_2)=\min(a,c),
\end{equation}
so that the left-hand summand is $x^{c}y^{a}(xy)^{h}$ and the right-hand one
is $x^{k_2'}y^{k_1'}(xy)^{m}$. In these coordinates the first equality in
\eqref{eq:schur_yb_bijection} reads $k_1'=a-m+h=h+\max(a-c,0)$, and the
conservation relation $k_2'=k_1'+j_2-j_1$ then gives
$k_2'=h+c-m=h+\max(c-a,0)$. Both outputs are thus expressed through
$(a,h,c)$ alone, by the map
\begin{equation}
  \label{eq:3d_R}
  R^{(3)}(a,h,c)
  =\bigl(h+\max(a-c,0),\;\min(a,c),\;h+\max(c-a,0)\bigr),
  \qquad
	R^{(3)}\colon\mathbb Z_{\ge0}^3\to\mathbb Z_{\ge0}^3 .
\end{equation}
This map is the standard combinatorial three-dimensional
$R$ of Kuniba--Okado \cite{KunibaOkado2012Tetrahedron3D}
and Kuniba \cite{Kuniba2016CombinatorialYBmaps}.

In the coordinates \eqref{eq:3d_R_coordinates}, weight
preservation in \Cref{lemma:schur_yb_bijection} reads
\begin{equation}
  \label{eq:3d_R_conservation}
  a+h=k_1'+m,
  \qquad
  h+c=m+k_2' ,
\end{equation}
the first matching the exponents of $y$ in
\eqref{eq:schur_yb_summands} and the second those of $x$.
Thus $R^{(3)}$ preserves each of the finite sets
\begin{equation}
  \label{eq:3d_R_sector}
  C_{p,q}=\{(i,j,k)\in\mathbb Z_{\ge0}^3:\ i+j=p,\ j+k=q\},
  \qquad p,q\in\mathbb Z_{\ge0},
\end{equation}
on which the middle coordinate $j$ ranges over
$0\le j\le\min(p,q)$ and determines the other two. On
$C_{p,q}$ the map $R^{(3)}$ is the reversal
$j\mapsto\min(p,q)-j$ of the middle coordinate, since
$j'=\min(i,k)=\min(p,q)-j$. In particular, $R^{(3)}$ is
an involution.

A map of triples has its own consistency equation, one
step above \eqref{eq:set_theoretic_ybe}.

\begin{definition}
\label{def:set_theoretic_tetrahedron_map}
Let $X$ be a set. A \emph{set-theoretic tetrahedron map} is a map
$R\colon X^3\to X^3$ satisfying the Zamolodchikov tetrahedron equation
\cite{Zamolodchikov1981Tetrahedron}
\begin{equation}
  \label{eq:tetrahedron}
  R_{123}\,R_{145}\,R_{246}\,R_{356}
  =
  R_{356}\,R_{246}\,R_{145}\,R_{123}
  \qquad\text{on }X^6,
\end{equation}
where $R_{abc}$ applies $R$ to the factors $a,b,c$ and acts as the identity on
the other three. The four index triples $123$, $145$, $246$, $356$ are the four
faces of a tetrahedron, each pair of them meeting in exactly one index; see
\Cref{fig:tetrahedron_faces}.\footnote{Note that some
sources use the indexing $124,135,236,456$, obtained by interchanging the
labels $3$ and $4$, and conventions for composing maps may reverse both sides
of \eqref{eq:tetrahedron}.}
\end{definition}

\begin{figure}[ht]
\centering
\begin{tikzpicture}[scale=1.05,line cap=round,
    baseline={(current bounding box.center)},font=\small,
    every node/.style={inner sep=1.4pt},
    edg/.style={line width=1.1pt},
    fc/.style={fill=upcol!8}]
  \coordinate (A) at (0,2.1);
  \coordinate (B) at (-2.0,-1.3);
  \coordinate (C) at (2.0,-1.3);
  \coordinate (D) at (0,-0.2);
  \fill[fc] (A)--(B)--(D)--cycle;
  \fill[fc] (A)--(C)--(D)--cycle;
  \fill[fc] (B)--(C)--(D)--cycle;
  \draw[edg,upcol] (A)--(B) (A)--(C) (B)--(C);
  \draw[edg,downcol] (A)--(D) (B)--(D) (C)--(D);
  \foreach \p in {A,B,C,D}{\node[circle,fill=black,inner sep=0pt,minimum size=3pt] at (\p){};}
  \node[above left]  at ($(A)!0.5!(B)$) {$1$};
  \node[above right] at ($(A)!0.5!(C)$) {$2$};
	\node[below,yshift=-2]       at ($(B)!0.5!(C)$) {$3$};
  \node[right]       at ($(A)!0.5!(D)$) {$4$};
	\node[above,yshift=2]       at ($(B)!0.5!(D)$) {$5$};
	\node[above,yshift=2]       at ($(C)!0.5!(D)$) {$6$};
  \node at (-0.78,0.02) {$R_{145}$};
  \node at (0.78,0.02)  {$R_{246}$};
  \node at (0,-1.00)    {$R_{356}$};
  \node[above] at (0,2.40) {$R_{123}$ (outer face)};
\end{tikzpicture}
\caption{The four factors of the tetrahedron equation \eqref{eq:tetrahedron}.
The six tensor factors of $X^6$ are drawn as the six edges of a tetrahedron,
and each $R_{abc}$ acts on the three edges of one face: $123$ is the outer face,
and $145$, $246$, $356$ are the three faces meeting at the interior vertex.}
\label{fig:tetrahedron_faces}
\end{figure}
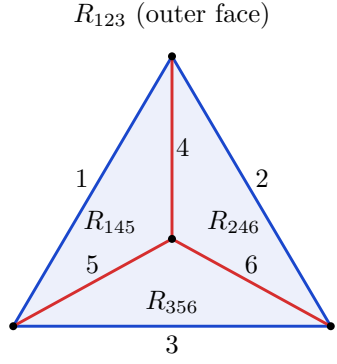

For a fixed boundary, $h$ recovers the original internal
state through
\begin{equation}
  \label{eq:3d_R_reconstruction}
  k_1=h+\max(i_1-i_2,0),\qquad k_2=h+\max(i_2-i_1,0),\qquad
  k_3=i_3+j_1-k_1,\qquad k_3'=j_3+i_1-k_1',
\end{equation}
the first two by \eqref{eq:schur_yb_internal_variables} and the definition of
$h$, and the last two by \eqref{eq:schur_yb_internal_variables} again.
See \Cref{fig:active_triple}.

Thus, we arrive at the following statement, which is the main result of this section.
\begin{proposition}
\label{prop:schur_bijection_is_3d_r}
Let $h=\min(k_1,k_2)$ and $m=\min(j_1,j_2)$.
\begin{enumerate}[label=\textup{(\roman*)}]
\item Under the deterministic Schur matching
  \eqref{eq:schur_yb_bijection}, the three active
  occupations transform as
\begin{equation}
  \label{eq:schur_active_map}
  (j_1,h,j_2)\longmapsto(k_1',m,k_2')=R^{(3)}(j_1,h,j_2).
\end{equation}
\item Consequently, after adjoining the fixed boundary
  data
	$(i_1,i_2,i_3,j_1,j_2,j_3)$
	needed to reconstruct the remaining arrows, the
  Schur matching is the restriction of $R^{(3)}$ to a
  finite boundary-dependent slice of $\mathbb Z_{\ge0}^3$.
\item \textup{(\cite{KunibaOkado2012Tetrahedron3D},
  \cite{Kuniba2016CombinatorialYBmaps})} The map $R^{(3)}$
  is an involutive set-theoretic tetrahedron map: it
  satisfies \eqref{eq:tetrahedron} on $\mathbb Z_{\ge0}^6$.
\item Conversely, every triple
  $(a,h,c)\in\mathbb Z_{\ge0}^3$ arises as the active
  triple of the Schur matching for some boundary. Thus
  $R^{(3)}$ on all of $\mathbb Z_{\ge0}^3$ is assembled
  from the Schur matchings over the various boundaries.
\end{enumerate}
\end{proposition}
\begin{proof}
Part (i) is the computation preceding \eqref{eq:3d_R}.
For (ii), fix a boundary. Then $a=j_1$ and $c=j_2$ are
frozen and only $h$ varies, over the image of the
interval \eqref{eq:schur_yb_k1_bounds} under
$k_1\mapsto\min(k_1,k_2)$. The slice is therefore the
one-dimensional set $\{(j_1,h,j_2)\}$ with $h$ running
over that image. Part (iii) is a theorem of the cited
literature, and we do not reproduce its proof. For (iv),
given $(a,h,c)\in\mathbb Z_{\ge0}^3$, take the boundary
\begin{equation*}
  i_1=i_2=0,\qquad i_3=h+\max(c-a,0),\qquad
  j_1=a,\qquad j_2=c,\qquad j_3=h+\max(a-c,0).
\end{equation*}
The global balance $i_1+j_2+j_3=i_2+i_3+j_1$ holds, since both sides equal
$h+\max(a,c)$. The state $k_1=k_2=h$ lies in the interval
\eqref{eq:schur_yb_k1_bounds}, and \eqref{eq:3d_R_reconstruction} gives
$k_3=\max(a,c)$ and $k_3'=0$.
\end{proof}

\begin{figure}[ht]
\centering
\begin{tikzpicture}[scale=0.85,line cap=round,
    baseline={(current bounding box.center)},font=\small,
    every node/.style={inner sep=1.4pt},
    spec/.style={gray!45,line width=2pt},
    carr/.style={gray!75!black,line width=1.4pt,densely dashed},
    speclab/.style={gray!70!black}]
  \def\lw{2pt}
  \draw[spec] (-1.2,-0.9)--(0.2,0);
  \draw[spec] (-1.2,0.9)--(0.2,0);
  \draw[spec] (1.6,-2.0)--(1.6,-0.9);
  \draw[spec] (1.6,0.9)--(1.6,2.0);
  \draw[carr] (1.6,-0.9)--(1.6,0.9);
  \draw[uppath,line width=\lw,ybeflowb] (0.2,0)--(1.6,0.9);
  \draw[uppath,line width=\lw,ybeflow]  (1.6,0.9)--(3.0,0.9);
  \draw[downpath,line width=\lw,ybeflowb] (0.2,0)--(1.6,-0.9);
  \draw[downpath,line width=\lw,ybeflow]  (1.6,-0.9)--(3.0,-0.9);
  \node[circle,fill=black,inner sep=0pt,minimum size=3pt] at (0.2,0){};
  \node[circle,fill=black,inner sep=0pt,minimum size=3pt] at (1.6,0.9){};
  \node[circle,fill=black,inner sep=0pt,minimum size=3pt] at (1.6,-0.9){};
  \node[left,speclab]  at (-1.2,0.9){$i_1$};
  \node[left,speclab]  at (-1.2,-0.9){$i_2$};
  \node[above,speclab] at (1.6,2.0){$j_3$};
  \node[below,speclab] at (1.6,-2.0){$i_3$};
  \node[right] at (1.70,-0.04){$k_3$};
  \node[right] at (3.0,0.9){$j_2$};
  \node[right] at (3.0,-0.9){$j_1$};
  \node at (0.73,0.72){$k_2$};
  \node at (0.73,-0.72){$k_1$};
  \node at (0.9,-2.75){$(j_1,\,h,\,j_2)$, \ $h=\min(k_1,k_2)$};
\end{tikzpicture}
\qquad$\overset{\textstyle R^{(3)}}{\longmapsto}$\qquad
\begin{tikzpicture}[scale=0.85,line cap=round,
    baseline={(current bounding box.center)},font=\small,
    every node/.style={inner sep=1.4pt},
    spec/.style={gray!45,line width=2pt},
    carr/.style={gray!75!black,line width=1.4pt,densely dashed},
    speclab/.style={gray!70!black}]
  \def\lw{2pt}
  \draw[spec] (-1.2,-0.9)--(0,-0.9);
  \draw[spec] (-1.2,0.9)--(0,0.9);
  \draw[spec] (0,-2.0)--(0,-0.9);
  \draw[spec] (0,0.9)--(0,2.0);
  \draw[carr] (0,-0.9)--(0,0.9);
  \draw[uppath,line width=\lw,ybeflowb] (0,-0.9)--(1.5,0);
  \draw[uppath,line width=\lw,ybeflow]  (1.5,0)--(3.0,0.9);
  \draw[downpath,line width=\lw,ybeflowb] (0,0.9)--(1.5,0);
  \draw[downpath,line width=\lw,ybeflow]  (1.5,0)--(3.0,-0.9);
  \node[circle,fill=black,inner sep=0pt,minimum size=3pt] at (0,0.9){};
  \node[circle,fill=black,inner sep=0pt,minimum size=3pt] at (0,-0.9){};
  \node[circle,fill=black,inner sep=0pt,minimum size=3pt] at (1.5,0){};
  \node[left,speclab]  at (-1.2,0.9){$i_1$};
  \node[left,speclab]  at (-1.2,-0.9){$i_2$};
  \node[above,speclab] at (0,2.0){$j_3$};
  \node[below,speclab] at (0,-2.0){$i_3$};
  \node[left] at (-0.10,0.0){$k_3'$};
  \node[right] at (3.0,0.9){$j_2$};
  \node[right] at (3.0,-0.9){$j_1$};
  \node at (0.91,0.72){$k_1'$};
  \node at (0.91,-0.72){$k_2'$};
  \node at (0.9,-2.75){$(k_1',\,m,\,k_2')$, \ $m=\min(j_1,j_2)$};
\end{tikzpicture}
\caption{The active part \eqref{eq:schur_active_map} of the Yang--Baxter move of
\Cref{fig:ybe}. The matching of \Cref{lemma:schur_yb_bijection} moves the triple
$(j_1,h,j_2)$ to the triple $(k_1',m,k_2')$ by the combinatorial
three-dimensional $R$ of \eqref{eq:3d_R}. The gray edges $i_1,i_2,i_3,j_3$ are
spectators: they are the same on the two sides. The dashed middle segment of the
vertical line is not a spectator --- it changes from $k_3$ to $k_3'$ --- but
both values are recovered
from the active triple and the boundary through
\eqref{eq:3d_R_reconstruction}.}
\label{fig:active_triple}
\end{figure}
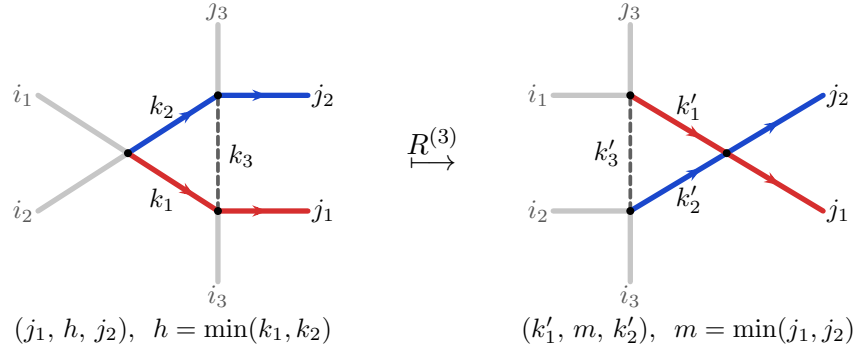

The combinatorial three-dimensional $R$, its conserved
sectors, its involutivity, and the tetrahedron equation
\eqref{eq:tetrahedron} come from the
quantized-function-algebra and crystal literature. The
closest references for the form used here are
Kuniba--Okado \cite{KunibaOkado2012Tetrahedron3D} and
Kuniba \cite{Kuniba2016CombinatorialYBmaps}. The latter
also contains the periodic reductions 
(\Cref{subsubsec:third_periodic} below)
and the local
relations \eqref{eq:periodic_layer} below. For the quantum
three-dimensional $R$ and for dimensional reduction, see
Kuniba--Okado \cite{KunibaOkado2015qOscillator}.

The map $R^{(3)}$ is the $q\to0$ limit of the quantum
three-dimensional $R$
\cite[Remark~2.6]{KunibaOkado2012Tetrahedron3D} and, up
to reversing the two outer coordinates, the transition
map between the two Lusztig (PBW) parametrizations for
the type $A_2$ braid move, see Lusztig
\cite[p.~451]{Lusztig1990CanonicalBases} and
\cite[\S2.5]{KunibaOkado2012Tetrahedron3D}.

Schur polynomials also arise from $q$-oscillator
solutions of the tetrahedron equation in
Iwao--Motegi--Ohkawa
\cite{IwaoMotegiOhkawa2025TetrahedronSchur},
\cite{IwaoMotegiOhkawa2026TetrahedralL}, but that
construction does not involve RSK.
Igonin--Konstan\-tinou-Rizos
\cite{IgoninKonstantinouRizos2023LocalYB} show in general
that solving a local Yang--Baxter correspondence produces
set-theoretic tetrahedron maps. Gleizer--Postnikov
\cite{gleizer2000littlewood} found a Yang--Baxter type
relation for scattering matrices describing tensor
products of $GL(N)$ representations, whose
piecewise-linear transformation of parameters solves the
tetrahedron equation. Genuine set-theoretic
Yang--Baxter maps built from insertion include the
plactic braiding of Lebed \cite{Lebed2020Plactic} and the
combinatorial $R$-matrices on rectangular tableaux of
Shimozono \cite{Shimozono2002AffineA} and Frieden
\cite{Frieden2021GeometricRmatrix}. Both use global
insertion data rather than the fixed-boundary matching of
\Cref{lemma:schur_yb_bijection}.

We have not found in the literature the identification
\eqref{eq:schur_active_map} of the active part of the
Schur Yang--Baxter matching with the combinatorial
three-dimensional $R$, nor the comparison between its
periodic reduction and the open RSK train that we describe in 
\Cref{subsubsec:middle_open} below.

\subsection{From three-dimensional $R$ to RSK, row insertion, and box-ball systems}
\label{subsec:two_geometries}

The local map \eqref{eq:3d_R} acts on a single triple. To
obtain a map on pairs of sequences
$\mathbf u,\mathbf v\in\mathbb Z_{\ge0}^n$ (the
column-difference sequences $(\lambda_t-\lambda_{t+1})_t$
of partitions, as in the dictionary
\eqref{eq:rsk_vertical_dictionary}), one composes
$n$ copies of it along a line of sites, and one must
choose which of the three coordinates is the carrier, that
is, which output of one site is fed as the input of the
next site, in which direction the sites are visited, and
whether the carrier is closed periodically or left open
with fixed boundary data. Since $R^{(3)}$ is invariant
under the exchange of its first and third coordinates, the
carrier is either an outer coordinate or the middle one,
and with the two boundary conditions this leaves four
cases. All of these four combinations
lead to known objects, as we explain in this subsection.
We begin with the case that has already appeared in \Cref{sec:rsk} above.

\subsubsection{Middle coordinate, open: RSK}
\label{subsubsec:middle_open}

In \Cref{sec:rsk}, the carrier is the middle coordinate,
left open: we drag the cross through two rows with fixed
boundary data, as in \Cref{fig:rsk_dragging}. Write
$h_c=\min(\lambda_c,\mu_c)-\varkappa_c$ for $c\ge1$, and
let $h_0=k_{ab}$ be the free input at the zeroth column.
The dictionary \eqref{eq:rsk_vertical_dictionary} turns
the move at column $c$ into
\begin{equation}
  \label{eq:open_train_layer}
  R^{(3)}\bigl(\lambda_{c+1}-\varkappa_{c+1},\;h_c,\;\mu_{c+1}-\varkappa_{c+1}\bigr)
  =\bigl(\nu_{c+1}-\mu_{c+1},\;h_{c+1},\;\nu_{c+1}-\lambda_{c+1}\bigr),
  \qquad c\ge0 .
\end{equation}
For $c\ge1$ this is \Cref{prop:schur_bijection_is_3d_r},
with $h_c$ the middle coordinate $\min(k_1,k_2)$ of
\eqref{eq:3d_R_coordinates}. For $c=0$, where the column
carries infinitely many vertical paths and the proposition
does not apply, it is the computation of $\nu_1$ in the
proof of \Cref{prop:train_is_rsk}. The free input
$h_0=k_{ab}$ is the $(a,b)$-entry of the RSK input matrix,
arrow conservation turns the moves at the later columns
into the toggle \eqref{eq:toggle}, and iterating over the
grid gives the growth diagram, hence RSK.

\subsubsection{Third coordinate, open: row insertion}
\label{subsubsec:third_open}
Taking the third coordinate as the carrier, visiting the
sites in the direction of increasing $t$ and leaving the
carrier open, gives the Schensted row insertion. Indeed, let
$\mathbf u,\mathbf v\in\mathbb Z_{\ge0}^n$
record the letter counts of two one-row tableaux, and
solve
\begin{equation*}
  (a_t,b_t,c_{t+1})=R^{(3)}(u_t,v_t,c_t),
  \qquad t=1,\dots,n,
  \qquad c_1=0.
\end{equation*}
Then $\mathbf a$ and $\mathbf b$ are the letter counts of
the first and second rows of the tableau obtained by row
inserting the letters of $\mathbf v$ into $\mathbf u$, and
$c_t$ is the number of letters $<t$ by which the new first
row exceeds $\mathbf u$. See \Cref{fig:row_insertion_counts}
for an example. Running the carrier in the
direction of decreasing $t$ instead amounts to reversing
the alphabet. 

The map
$(\mathbf u,\mathbf v)\mapsto(\mathbf b,\mathbf a)$ is the
plactic braiding of Lebed \cite{Lebed2020Plactic}, an
idempotent, non-invertible solution of
\eqref{eq:set_theoretic_ybe} on $X=\mathbb Z_{\ge0}^n$. In
a birational setting, Noumi--Yamada \cite{NoumiYamada2004}
obtain both tropical row insertion and a birational
counterpart of the combinatorial $R$-matrix from a system
of discrete Toda type, with open and quasi-periodic
indexing, respectively.

\begin{figure}[ht]
\centering
\begin{tikzpicture}[font=\small,
    bx/.style={draw,line width=0.4pt,minimum size=0.5cm,inner sep=0pt,anchor=south west}]
  \node[anchor=east] at (-0.25,0.25) {$\mathbf u=(2,1,0)$};
  \foreach \i/\l in {0/1,1/1,2/2}{\node[bx] at (0.5*\i,0) {$\l$};}
  \node[anchor=east] at (-0.25,-0.85) {$\mathbf v=(1,0,1)$};
  \foreach \i/\l in {0/1,1/3}{\node[bx] at (0.5*\i,-1.1) {$\l$};}
  \draw[-{Stealth[length=2mm]}] (2.1,-0.3) -- node[above,font=\footnotesize]{row insertion of $\mathbf v$ into $\mathbf u$} (6.1,-0.3);
  \foreach \i/\l in {0/1,1/1,2/1,3/3}{\node[bx] at (6.6+0.5*\i,0) {$\l$};}
  \node[bx] at (6.6,-0.5) {$2$};
  \node[anchor=west] at (8.9,0.25) {$\mathbf a=(3,0,1)$};
  \node[anchor=west] at (8.9,-0.25) {$\mathbf b=(0,1,0)$};
\end{tikzpicture}
\caption{Row insertion in letter counts, with $n=3$. The one-row
tableaux $1\,1\,2$ and $1\,3$ have letter counts $\mathbf u$ and
$\mathbf v$; inserting the letters of $\mathbf v$ into the first row
bumps a $2$ into the second row, and the two rows of the result have
letter counts $\mathbf a$ and $\mathbf b$. The carrier is
$(c_1,c_2,c_3,c_4)=(0,1,0,1)$: the new first row has one more letter
$1$ than $\mathbf u$, the same number of letters $\le2$, and one more
letter~$\le3$.}
\label{fig:row_insertion_counts}
\end{figure}

\subsubsection{Third coordinate, periodic: combinatorial $R$-matrices and the box-ball system}
\label{subsubsec:third_periodic}
Fix $n\ge1$ and, for $l\ge1$, let
$B^{(n)}_l=\{(u_1,\dots,u_n)\in\mathbb Z_{\ge0}^n:u_1+\dots+u_n=l\}$,
the one-row symmetric-tensor crystal. Given
$\mathbf u\in B^{(n)}_l$ and $\mathbf v\in B^{(n)}_{l'}$,
take the third coordinate as the carrier, visit the sites
in the direction of decreasing $t$, and close the carrier
periodically: look for carrier occupations
$c_1,\dots,c_n$ with $c_0=c_n$ such that
\begin{equation}
  \label{eq:periodic_layer}
  (a_t,b_t,c_{t-1})=R^{(3)}(u_t,v_t,c_t),
  \qquad t=1,\dots,n,
\end{equation}
that is, $a_t=v_t+\max(u_t-c_t,0)$, $b_t=\min(u_t,c_t)$ and
$c_{t-1}=v_t+\max(c_t-u_t,0)$. The recursion determines
the carrier from $c_n$, so this is one equation for one
unknown. For $l>l'$ it has a unique solution, and for
$l=l'$ every sufficiently large $c_n$ solves it with the
same output. Setting
$\mathcal R_{l,l'}(\mathbf u,\mathbf v)=(\mathbf b,\mathbf a)$
in these cases and
$\mathcal R_{l,l'}:=(\mathcal R_{l',l})^{-1}$ for $l<l'$
gives bijections
$$\mathcal R_{l,l'}\colon B^{(n)}_l\times B^{(n)}_{l'}\to B^{(n)}_{l'}\times B^{(n)}_l.$$
The tetrahedron equation \eqref{eq:tetrahedron} for the
local map implies the colored set-theoretic Yang--Baxter
relation of \Cref{def:colored_set_theoretic_yb_map} for
the family $\{\mathcal R_{l,l'}\}$. These
$\mathcal R_{l,l'}$ are the classical combinatorial
$R$-matrices of the crystals $B^{(n)}_l$, going back to
Nakayashiki--Yamada \cite{NakayashikiYamada1997Kostka}, and the time evolution of the box-ball system is
their successive application, with a carrier passing
through a row of boxes. The relations
\eqref{eq:periodic_layer}, their solvability, and this
identification are due to
\cite[\S6]{Kuniba2016CombinatorialYBmaps}. 
\begin{remark}
	Box-ball
	systems and RSK are also linked directly 
	in another manner, without passing
	through $R^{(3)}$: the recent work of
	Imamura--Mucciconi--Sasamoto--Scrimshaw
	\cite{ImamuraMucciconiSasamotoScrimshaw2026SkewColumnRSK}
	reads Fomin's local rule for Schensted column insertion as
	a two-lane box-ball system with a carrier, and projects it
	onto the classical box-ball system.
\end{remark}

\subsubsection{Middle coordinate, periodic: the combinatorial $R$-matrix with a dual crystal}
\label{subsubsec:middle_periodic}

Taking the middle coordinate as the carrier, visiting the
sites in the direction of increasing $t$ and closing the
carrier periodically, is simpler, because the middle
output $\min(a,c)$ of \eqref{eq:3d_R}
does not depend on the middle input $h$. For
$\mathbf u\in B^{(n)}_l$ and $\mathbf v\in B^{(n)}_{l'}$,
the relations
\begin{equation*}
  (a_t,c_t,b_t)=R^{(3)}(u_t,c_{t+1},v_t),
  \qquad t=1,\dots,n,
  \qquad c_{n+1}=c_1,
\end{equation*}
force $c_t=\min(u_t,v_t)$ and give $a_t=u_t+\delta_t$,
$b_t=v_t+\delta_t$ with
$\delta_t=\min(u_{t+1},v_{t+1})-\min(u_t,v_t)$. The map
$(\mathbf u,\mathbf v)\mapsto(\mathbf b,\mathbf a)$ is the
combinatorial $R$-matrix $R^{\vee}$ of
Kuniba--Okado--Yamada
\cite{KunibaOkadoYamada2005ReflectingEnd}.
Like $\mathcal R_{l,l'}$, it is a
set-theoretic Yang--Baxter map and not a
three-dimensional one, but it pairs $B^{(n)}_l$ with the
dual crystal, and drives their box-ball system with a
reflecting end. They also record the Yang--Baxter relations satisfied
by $R^{\vee}$ together with the maps $\mathcal R_{l,l'}$ of
\Cref{subsubsec:third_periodic}. Trace reductions
of the tetrahedron equation, including those along the
middle coordinate, are treated in
\cite[Chapter~13]{Kuniba2022QuantumGroups3D}.

\section{Bijectivization: From identities to Markov transitions}
\label{sec:bijectivization}

We outline the idea of \emph{bijectivization} (probabilistic bijection),
the device that converts the algebraic Yang--Baxter identities of the previous sections into
stochastic (Markov) steps. 

When both sides are sums of nonnegative weights,
$\sum_{a}\mathrm{wt}(a)=\sum_{b}\mathrm{wt}(b)$, we \emph{randomize}
the identity by introducing a transformation --- a Markov step ---
which turns
a configuration on the
left into a randomly chosen configuration on the right.
This randomization must be consistent with the weights, but its choice need not be unique.
Carried
across the lattice, these random steps are interpreted as a random growth process.
We call this passage from an identity to a randomized transformation a
\emph{bijectivization} or a \emph{probabilistic bijection}.
It is an application, in a solvable lattice setting, of the general principle of
\emph{coupling} in probability theory; see the books of Lindvall \cite{Lindvall2002} and
Thorisson \cite{Thorisson2000}.

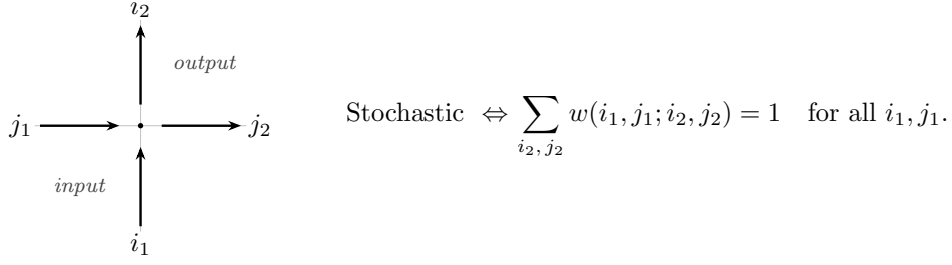
\begin{figure}[ht]
\centering
\begin{tikzpicture}[scale=1.4,>={Stealth[length=2mm]},font=\small,
    every node/.style={inner sep=1.5pt}]
  \draw[gridln] (0,-1.0)--(0,1.0) (-1.0,0)--(1.0,0);
  \draw[-{Stealth[length=1.7mm]},line width=0.9pt] (0,-0.95)--(0,-0.2);
  \draw[-{Stealth[length=1.7mm]},line width=0.9pt] (-0.95,0)--(-0.2,0);
  \draw[-{Stealth[length=1.7mm]},line width=0.9pt] (0,0.2)--(0,0.95);
  \draw[-{Stealth[length=1.7mm]},line width=0.9pt] (0.2,0)--(0.95,0);
  \node[vtx] at (0,0) {};
  \node[below] at (0,-0.98) {$i_1$};
  \node[left]  at (-0.98,0) {$j_1$};
  \node[above] at (0,0.98) {$i_2$};
  \node[right] at (0.98,0) {$j_2$};
  \node[font={\scriptsize\itshape},gray!50!black] at (-0.58,-0.58) {input};
  \node[font={\scriptsize\itshape},gray!50!black] at (0.60,0.58) {output};
  \node[anchor=west] at (1.9,0)
    {Stochastic
		\ $\Leftrightarrow$\ 
		$\displaystyle\sum_{i_2,\,j_2} w(i_1,j_1;i_2,j_2)=1\quad\text{for all }i_1,j_1$.};
\end{tikzpicture}
\caption{A single vertex has an \emph{input} (the incoming occupations
$(i_1,j_1)$ on the bottom and left edges) and an \emph{output} (the outgoing
occupations $(i_2,j_2)$ on the top and right).
The Boltzmann weight of the vertex 
$w(i_1,j_1;i_2,j_2)$ is called \emph{stochastic} when these weights are the
conditional probabilities of the output given the input.}
\label{fig:stochastic_vertex}
\end{figure}

\medskip

Bijectivization of the Yang--Baxter equation was introduced in
\cite{BufetovPetrovYB2017},
extending the earlier works which constructed randomized (Markov)
RSK-like insertion processes for Hall--Littlewood and $q$-Whittaker polynomials
\cite{OConnellPei2012},
\cite{BorodinPetrov2013NN}, \cite{MatveevPetrov2014},
\cite{BufetovPetrov2014}, \cite{BorodinBufetovWheeler2016}.
In integrable probability, Yang--Baxter bijectivization has been 
applied further to connect spin Hall--Littlewood 
and spin $q$-Whittaker polynomials \cite{BufetovMucciconiPetrov2018},
\cite{MucciconiPetrov2020} to various stochastic vertex models; 
to construct Markov maps acting on the fixed-time
distributions of TASEP
\cite{PetrovSaenz2019backTASEP}, \cite{petrov2022rewriting}
or preserving the KPZ pure states of the stochastic
six-vertex model on the full plane
\cite{nicoletti2022irreversible}.
Furthermore, bijectivization has been extended to half-space models
\cite{chen2021stable}
and to the full two-parameter Macdonald setting 
\cite{aigner2020_Macdonald_RSK}, \cite{FriedenSchreierAigner2024qtRSK}.

Strikingly, the
same mechanism appears independently in two-dimensional critical statistical
mechanics in the proof of rotational invariance of the Ising, random-cluster,
and six-vertex models.
Namely, the
star-triangle (Yang--Baxter) relation is realized as a move that ``uses extra randomness'' and ``is not a
deterministic matching of the different configurations''
\cite{DuminilCopinKozlowskiKrachunManolescuOulamara2020}. 
This is the general mechanism which we are about to describe.

\medskip

Consider any identity with positive terms,
\begin{equation}
  \label{eq:identity_sum}
  \sum_{a\in A}\mathrm{wt}(a) \;=\; \sum_{b\in B}\mathrm{wt}(b),
\end{equation}
where $A,B$ are nonempty finite or countable index sets, $\mathrm{wt}>0$, and the common total mass
$Z\coloneqq\sum_{a\in A}\mathrm{wt}(a)=\sum_{b\in B}\mathrm{wt}(b)$ is finite. A
\emph{bijectivization} (or probabilistic bijection) of \eqref{eq:identity_sum}
is a pair of nonnegative matrices $\mathsf{p}^{\mathrm{fwd}}(a\to b)$ and
$\mathsf{p}^{\mathrm{bwd}}(b\to a)$, indexed by
$a\in A$ and $b\in B$, satisfying the
stochasticity
\begin{equation*}
	\sum_{b\in B}\mathsf{p}^{\mathrm{fwd}}(a\to b)=1,\qquad
	\sum_{a\in A}\mathsf{p}^{\mathrm{bwd}}(b\to a)=1,
\end{equation*}
and the 
reversibility (detailed balance) relation
\begin{equation}
  \label{eq:bijectivization}
  \mathsf{p}^{\mathrm{fwd}}(a\to b)\,\mathrm{wt}(a) \;=\; \mathsf{p}^{\mathrm{bwd}}(b\to a)\,\mathrm{wt}(b)
  \qquad\text{for all } a\in A,\ b\in B.
\end{equation}
Summing \eqref{eq:bijectivization} over $b$ recovers each term
$\mathrm{wt}(a)$ of the left side of
\eqref{eq:identity_sum} from $\mathsf{p}^{\mathrm{fwd}}$ being stochastic, and summing over $a$
recovers each term of the right side; thus a bijectivization is exactly a \emph{coupling} 
\cite{Lindvall2002}, \cite{Thorisson2000}
of the two
sides of the identity.
That is, the joint measure on $A\times B$
with weights
$\mathsf{p}^{\mathrm{fwd}}(a\to b)\mathrm{wt}(a)$
has marginals $\mathrm{wt}(a)$ and $\mathrm{wt}(b)$ on $A$ and $B$, respectively;
dividing it by $Z$ turns it into a probability distribution.

Let us highlight two extreme cases of bijectivization:
\begin{enumerate}[label=$\bullet$]
  \item If both $\mathsf{p}^{\mathrm{fwd}}$ and $\mathsf{p}^{\mathrm{bwd}}$ are equal to
	zero or one, then they are supported on the graphs of two mutually inverse maps, and the
	coupling is an honest weight-preserving \emph{bijection} $a\mapsto b$.
	(A zero-one $\mathsf{p}^{\mathrm{fwd}}$ by itself only makes the transport
	deterministic: it may send several $a$'s to one $b$, and then
	$\mathsf{p}^{\mathrm{bwd}}$ is genuinely random.)
	One can say that this bijection provides a combinatorial proof of \eqref{eq:identity_sum}.
  
\item The opposite extreme is the \emph{independent}  (``maximally random'')
	bijectivization
	\begin{equation*}
		\mathsf{p}^{\mathrm{fwd}}(a\to b)=\frac{\mathrm{wt}(b)}{\sum_{b'\in B}\mathrm{wt}(b')},
		\qquad
		\mathsf{p}^{\mathrm{bwd}}(b\to a)=\frac{\mathrm{wt}(a)}{\sum_{a'\in A}\mathrm{wt}(a')},
	\end{equation*}
	which makes the output independent of the input.
	The joint measure in this case is $Z^{-1}\mathrm{wt}(a)\mathrm{wt}(b)$; dividing it by $Z$
	as above turns it into the product of the two normalized marginals
	$\mathrm{wt}(a)/Z$ and $\mathrm{wt}(b)/Z$.
\end{enumerate}
A bijectivization always exists, 
as the second example shows.
It is unique only when one of $A,B$ is a singleton.
When both $A$ and $B$ have exactly two elements, the bijectivization is a one-parameter family.

\medskip

When bijectivization is applied to the YBE \eqref{eq:ybe}, 
the sets $A$ and $B$ are the sets of arrow configurations of the internal
edges (indexed by $k_1,k_2,k_3$ and $k_1',k_2',k_3'$, respectively) which carry a nonzero
weight, and the
weights are the products of the three vertex weights on each side.
Here and below, $A$ and $B$ always mean the supports of the two sums: configurations of
weight zero are discarded, and all stochasticity and uniqueness statements are imposed on
the positive-weight support only.
The full power of bijectivization is realized when it is applied
sequentially on the lattice, by dragging the cross vertex 
through several columns, one at a time. In particular, 
if the cross is dragged from left to right, then locality
of the moves implies that the distribution of $k_1',k_2',k_3'$
in the first (leftmost) column does not depend on the subsequent columns.

For the Schur weights this construction admits the first
of the two extreme cases above: the matching
\eqref{eq:schur_yb_bijection} of
\Cref{lemma:schur_yb_bijection} is a bijectivization of
\eqref{eq:ybe} with
$\mathsf p^{\mathrm{fwd}},\mathsf p^{\mathrm{bwd}}\in\{0,1\}$,
and dragging the cross across the lattice then produces
the deterministic RSK correspondence of
\Cref{prop:train_is_rsk} instead of a random dynamics.
For generic parameter values, the $q$- and $t$-deformed
models of \Cref{sec:qdef} have boundary data admitting no
termwise matching of the two sides, and the couplings
they require are random; see \Cref{sec:def_rsk}.

\section{$q$-deformations: Hall--Littlewood, $q$-Whittaker, and fusion}
\label{sec:qdef}

The $t$- and $q$-deformations treated in this section
continue the probabilistic half of these notes, begun in
\Cref{sec:bijectivization}.

Putting the five-vertex model aside, we now focus on the two other Schur models of
\Cref{sec:schur_models}.
Each of them carries deformations that preserve the YBE and hence
the Cauchy identity (proven exactly in the same way
as in \Cref{subsec:ybe_cauchy}).
We do not explore the full generality of the deformations
(detailed in 
\cite{Borodin2014vertex},
\cite{BorodinPetrov2016inhom} and 
\cite{BorodinWheelerSpinq}, \cite{MucciconiPetrov2020}, \cite{BorodinKorotkikh2021inhom},
\cite{korotkikh2024representation}), 
but focus on the symmetric-function cases of Hall--Littlewood and
$q$-Whittaker polynomials.

\subsection{Vertex model for Hall--Littlewood polynomials}
\label{subsec:hl}

The Hall--Littlewood weights are the $t$-deformation of the higher-spin
six-vertex model of \Cref{rem:higher_spin}: horizontal edges carry $j_1,j_2\in\{0,1\}$,
the vertical edges have unbounded capacity, and 
arrows are conserved,
$i_1+j_1=i_2+j_2$. 
The four nonzero weights, with spectral parameter $x$, are
\begin{equation}
  \label{eq:hl_weights}
  \vcenter{\hbox{\small
  \begin{tabular}{cccc}
  \makebox[2.4cm]{\fvv{%
    \draw[upcol,line width=1pt,->] (-0.1,-0.6)--(-0.1,0.6);
    \draw[upcol,line width=1pt,->] (0,-0.6)--(0,0.6);
    \draw[upcol,line width=1pt,->] (0.1,-0.6)--(0.1,0.6);}} &
  \makebox[2.4cm]{\fvv{%
    \draw[upcol,line width=1pt,->] (-0.1,-0.6)--(-0.1,0.6);
    \draw[upcol,line width=1pt,->] (0,-0.6)--(0,0.6);
    \draw[upcol,line width=1pt,->] (0.1,-0.6)--(0.1,0.6);
    \draw[upcol,line width=1pt,->] (-0.6,0)--(0.6,0);}} &
  \makebox[2.4cm]{\fvv{%
    \draw[upcol,line width=1pt,->] (-0.1,-0.6)--(-0.1,0.6);
    \draw[upcol,line width=1pt,->] (0,-0.6)--(0,0.6);
    \draw[upcol,line width=1pt,->] (0.1,-0.6)--(0.1,0.6);
    \draw[upcol,line width=1pt,->] (-0.6,0)--(-0.2,0)--(-0.2,0.6);}} &
  \makebox[2.4cm]{\fvv{%
    \draw[upcol,line width=1pt,->] (-0.1,-0.6)--(-0.1,0.6);
    \draw[upcol,line width=1pt,->] (0,-0.6)--(0,0.6);
    \draw[upcol,line width=1pt,->] (0.1,-0.6)--(0.1,0)--(0.6,0);}} \\[3pt]
  \makebox[2.4cm]{$w=1$} & \makebox[2.4cm]{$w=x$} & \makebox[2.4cm]{$w=1$} & \makebox[2.4cm]{$w=x\,(1-t^{\,g})$} \\
  \end{tabular}}}
\end{equation}
where $g$ denotes the occupation of the incoming (bottom) vertical edge
($g=3$ in the pictures), and
all other configurations are weighted $0$. At $t=0$, we recover the Schur weights
from \Cref{fig:5v_model_higher_spin}.
Up to the horizontal and vertical gauge factors $x^{\,j_2-j_1}$ and
$(t;t)_{i_1}/(t;t)_{i_2}$, respectively, these weights are the Fock space matrix elements of the
$q$-boson $L$-operator of \cite{BogoliubovIzerginKitanine98}; see
\cite[\S3.2]{korff2013cylindric}.

\begin{proposition}[\cite{tsilevich2006quantum},
\cite{korff2013cylindric},
\cite{foda2009hall}, \cite{Borodin2014vertex}]
\label{prop:hl}
For a partition $\lambda$ with $\ell(\lambda)\le N$, the partition function of the higher-spin model with $N$
rows of spectral parameters $x_1,\dots,x_N$, weights \eqref{eq:hl_weights},
top boundary $\lambda$ (given by the multiset $\left\{ \lambda_1,\ldots,\lambda_N  \right\}$), fully packed left boundary, and empty bottom and right boundaries
(see \Cref{fig:5v_model_higher_spin}) equals the Hall--Littlewood polynomial $P_\lambda(x_1,\dots,x_N;0,t)$.
\end{proposition}

A standard reference on Hall--Littlewood and Macdonald polynomials 
is the book \cite{Macdonald1995}.
Here and below, we use the notation $P_\lambda(\cdots;q,t)$
for the Macdonald polynomials. For $q=0$, they reduce to the Hall--Littlewood polynomials.
For $t=0$, Macdonald polynomials become the $q$-Whittaker polynomials, which we
discuss in \Cref{subsec:qw} below.
Both families become Schur polynomials at $q=t$ (which, in applications
to vertex models, is taken to be $q=t=0$).

\begin{remark}[Spin Hall--Littlewood]
\label{rem:spin_hl}
Adjoining a \emph{spin} parameter $s$ deforms the four weights
\eqref{eq:hl_weights}, in the same order, into
\begin{equation*}
  \frac{1-sxt^{g}}{1-sx},\qquad
  \frac{x-st^{g}}{1-sx},\qquad
  \frac{1-s^{2} t^{g}}{1-sx},\qquad
  \frac{x(1-t^{g})}{1-sx},
\end{equation*}
which return \eqref{eq:hl_weights} at $s=0$. Their partition functions are, up to
normalization, the \emph{spin Hall--Littlewood} rational functions \cite{Borodin2014vertex},
\cite{BorodinPetrov2016inhom}. One may even let $s\rightsquigarrow s_j$ vary column
by column without breaking the YBE.
\end{remark}

\subsection{Vertex model for $q$-Whittaker polynomials}
\label{subsec:qw}

Now
consider a $q$-deformation of the 
up-right weight $W$ \eqref{eq:upright_cross}, whose edges carry arbitrary
multiplicities $i_1,j_1,i_2,j_2\in\mathbb Z_{\ge0}$,
subject to the conservation law $i_1+j_1=i_2+j_2$ and the restriction $j_2\le i_1$.
The $q$-deformed weight is defined as
\begin{equation}
  \label{eq:qw_weights}
  W(i_1,j_1;i_2,j_2)
  =\mathbf 1_{i_1+j_1=i_2+j_2}\,\mathbf 1_{i_1\ge j_2}\,x^{j_2}\,
   \frac{(q;q)_{i_1}}{(q;q)_{j_2}\,(q;q)_{i_1-j_2}},
\end{equation}
where $(a;q)_m=\prod_{k=0}^{m-1}(1-aq^k)$ is the $q$-Pochhammer symbol (with
$(a;q)_0=1$, $(a;q)_{m}=\prod_{k=0}^{-m-1}(1-aq^{m+k})^{-1}$ for $m<0$, and
$(a;q)_{\infty}=\prod_{k\ge0}(1-aq^{k})$).
The down-right weight, deforming \eqref{eq:downright_cross}, is
\begin{equation}
  \label{eq:qw_weights_dual}
  W^{*}(i_1,j_1;i_2,j_2)
  =\mathbf 1_{i_2+j_1=i_1+j_2}\,\mathbf 1_{i_2\ge j_2}\,x^{j_2}\,
   \frac{(q;q)_{i_1}}{(q;q)_{j_2}\,(q;q)_{i_2-j_2}}.
\end{equation}
As $q\to0$, both weights reduce to
\eqref{eq:upright_cross} and \eqref{eq:downright_cross}.
Unlike in the Schur case, \eqref{eq:qw_weights_dual} is not the plain reflection
$i_1\leftrightarrow i_2$ of \eqref{eq:qw_weights}, but carries an extra vertical-edge
factor $(q;q)_{i_1}/(q;q)_{i_2}$. It telescopes along a column, and it is required for the
YBE of \Cref{subsec:qw_rsk}.\footnote{The asymmetry also reflects the shape of the Cauchy identity, which pairs $P_\lambda$
with $Q_\lambda$ rather than with a second copy of $P_\lambda$. The two families are
merely proportional, $Q_\lambda=b_\lambda P_\lambda$ with
$b_\lambda=\prod_{i\ge1}(q;q)_{\lambda_i-\lambda_{i+1}}^{-1}$ depending on~$\lambda$ and~$q$
but not on the variables. The two models are normalized accordingly: the up-right model
\eqref{eq:qw_weights} produces $Q_\lambda$ (\Cref{prop:qw}), while the telescoped factor in
\eqref{eq:qw_weights_dual} makes the down-right model produce $P_\lambda$.}

\begin{remark}
The ratio $(q;q)_{i_1}/(q;q)_{j_2}\,(q;q)_{i_1-j_2}$ is the $q$-binomial
coefficient 
$\begin{bmatrix} i_1 \\ j_2 \end{bmatrix}_q$,
a standard $q$-deformation of the binomial
coefficient $\binom{i_1}{j_2}$. Note that the ``classical'' limit for us 
is $q\to0$, not $q\to1$.
\end{remark}

\begin{proposition}[{cf.\ \cite[\S6.2]{korff2013cylindric}}]
\label{prop:qw}
For a partition $\lambda$, the partition function of the up-right model with $N$
rows of spectral parameters $x_1,\dots,x_N$, weights \eqref{eq:qw_weights},
top boundary $\lambda$, empty bottom boundary, and the infinite column-$0$
reservoir of \Cref{rmk:infinite_boundary_at_column_0} on the left
(see \Cref{fig:upright_model})
equals the $q$-Whittaker polynomial
$Q_\lambda(x_1,\dots,x_N;q,0)=b_\lambda\,P_\lambda(x_1,\dots,x_N;q,0)$, where
$b_\lambda=\prod_{i\ge1}(q;q)_{\lambda_i-\lambda_{i+1}}^{-1}$ does not depend on the
$x_i$'s. The down-right weights \eqref{eq:qw_weights_dual}, with the mirrored boundary
conditions, give $P_\lambda$ itself. At $q=0$ one has $b_\lambda=1$.
\end{proposition}

\begin{remark}
The column-$0$ reservoir vertex carries a $q$-Pochhammer factor: as $M\to\infty$,
$\bigl[\begin{smallmatrix}M\\ j\end{smallmatrix}\bigr]_q\to1/(q;q)_j$, so that
$W(\infty,j_1;\infty,j_2)=x^{j_2}/(q;q)_{j_2}$ (the $q$-analogue of
\Cref{rmk:infinite_boundary_at_column_0}). For $N=1$ and $\lambda=(m)$ this single vertex
already gives $x^{m}/(q;q)_m=Q_{(m)}(x;q,0)$, whereas $P_{(m)}(x;q,0)=x^m$.
For general $N$ the factor $b_\lambda$ is a global
normalization, see \Cref{rem:spin_qw} below.
\end{remark}

\begin{remark}
\label{rem:korff_qw}
The weights \eqref{eq:qw_weights} were introduced in
\cite[\S6.1]{korff2013cylindric} (up to a left--right
reflection of the lattice, and with $q$ denoted by $t$
there), as Fock space matrix elements of a $q$-boson
$L$-operator, together with the identification of the
partition functions with the skew $q$-Whittaker functions
--- in the cylindric (periodic) setting of that paper,
which reduces to \Cref{prop:qw} for the empty bottom
boundary. The
YBE established in
\cite[\S3.2]{korff2013cylindric} exchanges two rows of
the same model via an infinite-dimensional cross
depending on the ratio of the two spectral parameters,
and yields the symmetry of the partition functions in the
$x_i$'s.

The spin form of the weights \eqref{eq:qw_weights} is already present in
\cite[\S6]{Borodin2014vertex}, as the fused weights at their $q$-Hahn point,
together with a Cauchy identity pairing them with the Hall--Littlewood side ---
obtained there by analytic continuation in the fusion parameter $q^J$, without
a YBE for the fused weights themselves. The dual weights
\eqref{eq:qw_weights_dual} and the YBE of Cauchy type, in the
form of an exchange relation between row operators, first
appeared, in spin generality, in \cite[\S4.6, \S5.9]{BorodinWheelerSpinq}
(with the resulting Cauchy identity in \S7.1 there); the
explicit cross \eqref{eq:qw_cross_full} was computed, in its spin form, in
\cite[App.~A]{BufetovMucciconiPetrov2018}.
\end{remark}

\begin{remark}[Spin and diagonal parameters]
\label{rem:spin_qw}
\label{rem:korotkikh}
One can add a spin parameter $s$ further deforming
\eqref{eq:qw_weights} into the spin $q$-Whittaker
weight of \cite{BorodinWheelerSpinq}, \cite{MucciconiPetrov2020}:
\begin{equation*}
  W(i_1,j_1;i_2,j_2)=\mathbf 1_{i_1+j_1=i_2+j_2}\,\mathbf 1_{i_1\ge j_2}\,x^{j_2}\,
   \frac{(-s/x;q)_{j_2}\,(-sx;q)_{i_1-j_2}\,(q;q)_{i_2}}{(q;q)_{j_2}\,(q;q)_{i_1-j_2}\,(s^2;q)_{i_2}}.
\end{equation*}
At $s=0$ this is $(q;q)_{i_2}/(q;q)_{i_1}$ times \eqref{eq:qw_weights}, not
\eqref{eq:qw_weights} itself; the discrepancy is a vertical-edge gauge, which
telescopes along each column to the overall factor $b_\lambda^{-1}$, so the
$s=0$ spin partition function equals $P_\lambda$ rather than the
$Q_\lambda=b_\lambda P_\lambda$ of \Cref{prop:qw}.
The corresponding partition functions, with suitably
normalized reservoir weights, are the spin $q$-Whittaker
polynomials (unlike the spin Hall--Littlewood symmetric
rational functions, these are polynomials in the $x_i$'s).
Going further, 
\cite{BorodinKorotkikh2021inhom},
\cite{korotkikh2024representation} enriched the
spin $q$-Whittaker weights with extra \emph{diagonal} parameters
carried by the vertices; these have a representation-theoretic meaning and 
interpolation properties.
\end{remark}

\subsection{Macdonald duality}
\label{subsec:cauchy_fusion}

Here we illustrate how one can turn Hall--Littlewood
polynomials into the $q$-Whittaker ones, and connect this construction to vertex models.
First, consider the Cauchy identity for Hall--Littlewood polynomials
\cite[Ch.~III, \S4]{Macdonald1995}:\footnote{This Cauchy identity
also admits a Yang--Baxter proof, e.g., see \cite{Borodin2014vertex}.}
\begin{equation}
  \label{eq:hl_cauchy}
  \sum_{\lambda} P_\lambda(x_1,\dots,x_N;0,t)\,Q_\lambda(y_1,\dots,y_M;0,t)
  \;=\;\prod_{i=1}^{N}\prod_{j=1}^{M}\frac{1-t\,x_iy_j}{1-x_iy_j}.
\end{equation}
Here, $Q_\lambda$ is a multiple of $P_\lambda$, where the coefficient 
depends on $t$ and $\lambda$ but not on the $y_j$'s.

We now perform the following steps with \eqref{eq:hl_cauchy}:
\begin{enumerate}[label=$\bullet$]
\item Replace each $y_i$ by a finite geometric sequence
$y_i,ty_i,\dots,t^{J-1}y_i$ of some length $J\in\mathbb Z_{\ge1}$
(denote this finite geometric sequence by $y_i^{(J)}$).
This is a finite version of the \emph{plethystic substitution} appearing in 
symmetric function theory. On the vertex model side, this corresponds to \emph{fusion} of $J$ rows of the Hall--Littlewood vertex model, and dates back to \cite{KulishReshSkl1981yang}.
This turns the Cauchy identity \eqref{eq:hl_cauchy} into
\begin{equation}
\label{eq:hl_cauchy_fused_1}
	\sum_{\lambda} P_\lambda(x_1,\dots,x_N;0,t)\,Q_\lambda(y_1^{(J)},\dots,y_M^{(J)};0,t)
	\;=\;\prod_{i=1}^{N}\prod_{j=1}^{M}\frac{1-t^J\,x_iy_j}{1-x_iy_j}.
\end{equation}
\item Both sides of \eqref{eq:hl_cauchy_fused_1}
depend analytically on the $2M$ parameters $y_i$ and $t^Jy_i$, treated as independent,
in a neighborhood of $y_i=0$. Indeed, each $Q_\lambda(\cdots;0,t)$ is a polynomial in
the power sums, and the substitution $y_i\mapsto y_i^{(J)}$ turns the $k$-th power sum
into $\sum_{i=1}^{M}\bigl(y_i^{k}-(t^Jy_i)^{k}\bigr)/(1-t^{k})$.
The same dependence on the pair $x$, $t^Jx$ appears in the fused weights
\eqref{eq:fused_weight_z} of \Cref{subsec:fusion} below. In particular,
we may perform the substitutions
\begin{equation*}
	y_i\;\rightsquigarrow\;0,\qquad
	t^Jy_i\;\rightsquigarrow\; -z_i,\qquad i=1,\dots,M,
\end{equation*}
where $z_i$ are new indeterminates.
The Cauchy identity \eqref{eq:hl_cauchy_fused_1} becomes
\begin{equation}
\label{eq:hl_cauchy_fused_2}
	\sum_{\lambda} P_\lambda(x_1,\dots,x_N;0,t)F_\lambda(z_1,\ldots,z_M )=
	\prod_{i=1}^{N}\prod_{j=1}^{M}(1+x_iz_j).
\end{equation}

\item It remains to identify the symmetric functions $F_\lambda(z_1,\ldots,z_M )$ in 
\eqref{eq:hl_cauchy_fused_2}. 
Thanks to the Macdonald 
duality $\omega_{q,t}$ which interchanges $q$ and $t$ in the Macdonald polynomials
(and transposes the partitions, turning $Q_\lambda$ into $P_{\lambda'}$),
the Hall--Littlewood and $q$-Whittaker polynomials satisfy
the following dual Cauchy identity:
\begin{equation}
\label{eq:hl_qw_cauchy}
	\sum_{\lambda} P_\lambda(x_1,\dots,x_N;0,t)P_{\lambda'}(z_1,\ldots,z_M;t,0)=
	\prod_{i=1}^{N}\prod_{j=1}^{M}(1+x_iz_j).
\end{equation}
Since, for each $N$, the identity \eqref{eq:hl_qw_cauchy} determines the functions $P_{\lambda'}(\cdots;t,0)$
with $\ell(\lambda)\le N$ (by the linear independence of the polynomials
$P_\lambda(x_1,\ldots,x_N;0,t)$ with $\ell(\lambda)\le N$), and $N$ is
arbitrary, we conclude that
\begin{equation*}
	F_\lambda(z_1,\ldots,z_M )=P_{\lambda'}(z_1,\ldots,z_M;t,0),
\end{equation*}
which are exactly the $q$-Whittaker polynomials (where the role of $q$ is taken by the parameter~$t$).
\end{enumerate}

\subsection{Fusion}
\label{subsec:fusion}

The same steps as in \Cref{subsec:cauchy_fusion} may be carried out
directly at the level of the vertex weights. 
Let us briefly outline this procedure, which is known as \emph{fusion} 
\cite{KulishReshSkl1981yang}, \cite{KR1987Fusion}, \cite{Borodin2014vertex}, \cite{CorwinPetrov2015}, \cite{BorodinPetrov2016inhom}.

\medskip

The fundamental building block is the Hall--Littlewood vertex
\eqref{eq:hl_weights}, of horizontal capacity one, so $j_1,j_2\in\{0,1\}$.
Fix $J\in\mathbb Z_{\ge1}$ and consider $J$ copies of the vertex, stacked vertically.
Assign to the $r$-th row the spectral parameter $t^{\,r-1}x$, so that the
column carries the geometric string of parameters.
This mirrors the substitution $y_i\mapsto y_i^{(J)}$ from \Cref{subsec:cauchy_fusion} above.

The fused weight, with $i_1,i_2\in\mathbb Z_{\ge0}$ and $0\le j_1,j_2\le J$, is
defined as the sum over the $J$ incoming and outgoing
horizontal occupations $\vec h_1,\vec h_2\in\{0,1\}^J$:
\begin{equation}
  \label{eq:fusion_def}
  w^{(J)}_{x}(i_1,j_1;i_2,j_2)
  =\frac{\mathbf 1_{i_1+j_1=i_2+j_2}}{Z_{j_1}(J)}
  \sum_{\substack{|\vec h_1|=j_1\\ |\vec h_2|=j_2}}
  t^{\sum_{r=1}^{J}(r-1)\,h_1^{(r)}}\,
  \prod_{r=1}^{J} w_{t^{\,r-1}x}\!\bigl(\ell_{r-1},h_1^{(r)};\ell_r,h_2^{(r)}\bigr),
\end{equation}
where $w_u$ are the weights \eqref{eq:hl_weights} with spectral parameter $u$,
and the internal vertical occupations are determined by arrow conservation,
$\ell_0=i_1$ and $\ell_r=\ell_{r-1}+h_1^{(r)}-h_2^{(r)}$; the indicator is what
makes $\ell_J=i_2$. Strings that leave $\mathbb Z_{\ge0}$ need not be excluded by
hand: the first step with $\ell_r=-1$ is the fourth vertex of
\eqref{eq:hl_weights} at $g=0$, of weight $x(1-t^{0})=0$.
See \Cref{fig:fusion} for an illustration of the fusion procedure.

\begin{figure}[ht]
\centering
\begin{tikzpicture}[scale=1.0,>={Stealth[length=1.6mm]},
    baseline={(current bounding box.center)},font=\small,
    rap/.style={circle,draw=black,fill=white,line width=0.6pt,inner sep=0.5pt,
      minimum size=16pt,font=\scriptsize},
    occ/.style={font=\scriptsize,inner sep=1.2pt}]
  \def\Lx{0}
  \draw[gridln,->] (\Lx,-0.75)--(\Lx,3.95);
  \foreach \r in {1,2,3}{\draw[gridln,->] (\Lx-1.0,\r)--(\Lx+1.05,\r);}
  \node[rap] at (\Lx,1) {$x$};
  \node[rap] at (\Lx,2) {$tx$};
  \node[rap] at (\Lx,3) {$t^{2}x$};
  \foreach \r in {1,2,3}{%
    \node[occ,left]  at (\Lx-1.03,\r) {$h_1^{(\r)}$};
    \node[occ,right] at (\Lx+1.08,\r) {$h_2^{(\r)}$};}
  \node[occ,right] at (\Lx+0.12,0.2) {$\ell_0=i_1$};
  \node[occ,right] at (\Lx+0.12,1.5) {$\ell_1$};
  \node[occ,right] at (\Lx+0.12,2.5) {$\ell_2$};
  \node[occ,right] at (\Lx+0.12,3.6) {$\ell_3=i_2$};
  \node at (3.75,2.0) {$\xmapsto{\hspace{0.9cm}\text{fuse}\hspace{0.9cm}}$};
  \node[font=\scriptsize] at (3.75,1.15)
    {$\displaystyle\frac{1}{Z_{j_1}(J)}\sum_{\vec h_1,\vec h_2}
      t^{\sum_{r=1}^{J}(r-1)h_1^{(r)}}$};
  \def\Rx{7.4}
  \draw[gridln,->] (\Rx,1.05)--(\Rx,3.1);
  \draw[gridln,line width=1.5pt,->] (\Rx-1.05,2)--(\Rx+1.05,2);
  \node[rap] at (\Rx,2) {$x$};
  \node[occ,below] at (\Rx,1.0)  {$i_1$};
  \node[occ,above] at (\Rx,3.12) {$i_2$};
  \node[occ,left]  at (\Rx-1.08,2) {$j_1$};
  \node[occ,right] at (\Rx+1.08,2) {$j_2$};
  \node[font=\footnotesize] at (\Rx,0.25) {$w^{(J)}_{x}(i_1,j_1;i_2,j_2)$};
\end{tikzpicture}
\caption{The fusion procedure \eqref{eq:fusion_def} for $J=3$. Summing over
the incoming and outgoing strings $\vec h_1,\vec h_2\in\{0,1\}^J$ with totals
$j_1=|\vec h_1|$ and $j_2=|\vec h_2|$, with the $t$-exchangeable weight on the
incoming arrows, collapses the column to a single fused vertex of horizontal
capacity $J$.}
\label{fig:fusion}
\end{figure}
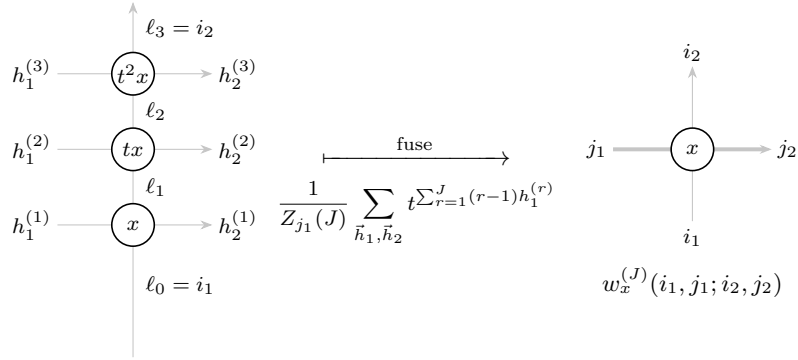

The normalization
in \eqref{eq:fusion_def}
is the $t$-deformed binomial
\[
  Z_{j}(J)\;=\!\!\sum_{\vec h_1:\,|\vec h_1|=j}\!\! t^{\sum_{r=1}^{J}(r-1)\,h_1^{(r)}}
  \;=\;t^{\binom{j}{2}}\,
	\begin{bmatrix}J\\j\end{bmatrix}_t.
\]
The weight\footnote{Throughout this subsection we assume $0<t<1$, so that $Z_j(J)>0$ for all $0\le j\le J$.
At $t=0$ the normalization vanishes for $j\ge2$ and \eqref{eq:fusion_def} reads $0/0$;
the fused weight is then defined by continuity, through the closed form
\eqref{eq:fused_weight_z} below, whose coefficients $c_p(t)$ are regular at $t=0$.}
$t^{\sum_{r}(r-1)h_1^{(r)}}/Z_{j_1}(J)$ placed on the incoming arrows
is the \emph{$t$-exchangeable} normalization (e.g., see \cite{GnedinOlsh2009q}
for a probabilistic perspective).
Fusion is built so as to preserve the YBE
\cite{KulishReshSkl1981yang}, \cite{KR1987Fusion}; the fused weights
\eqref{eq:fusion_def} therefore satisfy it as well (with a correspondingly
fused cross), as do their stochastic counterparts \cite{Borodin2014vertex},
\cite{BorodinPetrov2016inhom}. By
construction \eqref{eq:fusion_def} is a genuine vertex weight of horizontal
capacity $J$ --- a function of the four occupations $i_1,j_1,i_2,j_2$ alone.

\medskip
In fact, the outer sum over the outgoing string in \eqref{eq:fusion_def} is
redundant: a fused column maps $t$-exchangeable inputs to $t$-exchangeable
outputs \cite[Prop.~3.6]{CorwinPetrov2015}, so that for
\emph{every} fixed $\vec h_2$ with $|\vec h_2|=j_2$,
we have the following refined identity:
\begin{equation}
  \label{eq:texch_out}
  \sum_{\vec h_1:\,|\vec h_1|=j_1} t^{\sum_{r=1}^{J}(r-1)\,h_1^{(r)}}\,
  \prod_{r=1}^{J} w_{t^{\,r-1}x}\!\bigl(\ell_{r-1},h_1^{(r)};\ell_r,h_2^{(r)}\bigr)
  \;=\;\frac{Z_{j_1}(J)}{Z_{j_2}(J)}\,
  t^{\sum_{r=1}^{J}(r-1)\,h_2^{(r)}}\,w^{(J)}_{x}(i_1,j_1;i_2,j_2).
\end{equation}

\begin{remark}\label{rem:fusion_J2}
Let us illustrate \eqref{eq:texch_out} for $J=2$.
The general-$J$ case of \eqref{eq:texch_out} follows from the same computation
applied to adjacent rows, since the symmetrization underlying $t$-exchangeability
is generated by swaps of neighboring rows.

Take $i_1=0$, $j_1=1$, so the incoming string is $(1,0)$ with weight $1$ and
$(0,1)$ with weight $t$. For the outgoing string $\vec h_2=(0,1)$ the two incoming terms give
\[
  \underbrace{t\,w_{x}(0,0;0,0)\,w_{tx}(0,1;0,1)}_{=\,t^{2}x}
  +\underbrace{w_{x}(0,1;1,0)\,w_{tx}(1,0;0,1)}_{=\,tx(1-t)}
  \;=\;tx\;=\;x\,t^{\,1},
\]
while $\vec h_2=(1,0)$ gives $x=x\,t^{\,0}$. Since $j_1=j_2=1$ makes
$Z_{j_1}/Z_{j_2}=1$ in \eqref{eq:texch_out}, either outgoing string already
yields $w^{(2)}_{x}(0,1;0,1)=x$, with no outer sum needed.
\end{remark}

Carrying out the fusion sum \eqref{eq:fusion_def} --- a direct if lengthy
computation --- gives the fused weight in closed form:\footnote{This computation may be
reduced to checking certain recurrences for $w^{(J)}_{x}$.
We refer to \cite[Lem.~3.13, Thm.~3.15]{CorwinPetrov2015} for a detailed derivation
via these recurrences
in the presence of the additional spin parameter.}
\begin{multline}
  \label{eq:fused_weight}
  w^{(J)}_{x}(i_1,j_1;i_2,j_2)
  =
	\mathbf{1}_{i_1+j_1=i_2+j_2}\,
	t^{\frac12 i_1(i_1+2j_1-1)+\frac12(j_1-i_2)(j_1-i_2-1)}\,x^{j_2}\,
  \frac{(t;t)_{j_1}}{(t;t)_{j_2}\,(t;t)_{i_2}}
  \\\times\sum_{k=0}^{i_1}\frac{(t^{-i_1};t)_k\,(t^{-i_2};t)_k}{(t;t)_k}\,
  (t^{1+j_1-i_2+k};t)_{i_1-k}\,(t^{1+J-i_1-j_1+k};t)_{i_1-k}\,t^k .
\end{multline}
Since the prefactor multiplying $x^{j_2}$ is
free of both $x$ and $J$, we may write, for $i_1+j_1=i_2+j_2$,
$w^{(J)}_{x}
(i_1,j_1;i_2,j_2)
=x^{j_2}\sum_{p=0}^{i_1}c_p(t)\,(t^{J})^{p}$.
In particular
$t^{J}$ may be set to any complex value.

Expanding the one $t^{J}$-dependent factor by the $t$-binomial theorem and
evaluating the inner sum by the terminating $t$-Chu--Vandermonde identity
\cite[Eq.~(II.6)]{GasperRahman} gives $c_p(t)$ explicitly:
\begin{equation*}
	c_p(t)=
	(-1)^{p}\,t^{\binom{j_2-p}{2}}\,
	\frac{(t;t)_{j_1}}{(t;t)_{j_2}\,(t;t)_{i_2}}\,
	\begin{bmatrix}i_1\\p\end{bmatrix}_t
	\,(t^{\,1+j_2-p};t)_p\,(t^{\,1+j_1};t)_{i_1-p},
  \qquad 0\le p\le i_1 .
\end{equation*}
The factor $(t^{\,1+j_2-p};t)_p$ acquires the vanishing term $1-t^{0}$ as
soon as $p>j_2$, so 
we see that the degree of \eqref{eq:fused_weight} in $t^{J}$ is 
equal to $\min(i_1,j_2)$.

\medskip
Put $z=-t^{J}x$. Then
\begin{equation}
\label{eq:fused_weight_z}
	w^{(J)}_{x}(i_1,j_1;i_2,j_2)
	=x^{j_2}\sum_{p=0}^{\min(i_1,j_2)} c_p(t)\,(-z/x)^{p}
	=\sum_{p=0}^{\min(i_1,j_2)} c_p(t)\,(-1)^{p}\,x^{j_2-p}\,z^{\,p}.
\end{equation}
Send $x\to0$ while keeping $z$ fixed.
If $i_1< j_2$, then \eqref{eq:fused_weight_z} vanishes,
which leads to the indicator $\mathbf 1_{i_1\ge j_2}$ in \eqref{eq:qw_weights}.
If $i_1\ge j_2$, then the limit picks only the term with $p=j_2$, giving
the $q$-Whittaker weight:
\[
  \mathbf 1_{i_1\ge j_2}\;z^{j_2}\,
	\begin{bmatrix}i_1\\j_2\end{bmatrix}_t
  =\mathbf 1_{i_1\ge j_2}\;z^{j_2}\,\frac{(t;t)_{i_1}}{(t;t)_{j_2}\,(t;t)_{i_1-j_2}},
\]
where the parameter $t$ plays the role of $q$ in \eqref{eq:qw_weights}. 

\begin{remark}[Spin parameter]
\label{rem:spin_fusion}
Carrying the spin parameter $s$ of \Cref{rem:spin_hl} through the fusion works
equally well, and produces the \emph{spin} $q$-Whittaker weights of
\Cref{rem:spin_qw} --- but through a different specialization: not $x\to0$ with
$z=-t^{J}x$ fixed, but $x=s$, so that $t^{J}=-z/s$
\cite[App.~A.3]{BufetovMucciconiPetrov2018}, followed by the vertical-edge gauge
$(t;t)_{i_2}(s^2;t)_{i_1}/\bigl[(t;t)_{i_1}(s^2;t)_{i_2}\bigr]$. 
The fused $J$-dependent weights one has to go through
take the following $q$-hypergeometric form 
\cite{Mangazeev2014},
\cite{Borodin2014vertex},
\cite{CorwinPetrov2015}:
\begin{multline*}
  w^{(J)}_{x,s}(i_1,j_1;i_2,j_2)=
	\mathbf{1}_{i_1+j_1=i_2+j_2}\,
  \frac{(-1)^{i_1+j_2}\,t^{\frac12 i_1(i_1+2j_1-1)}\,s^{\,j_2-i_1}\,x^{i_1}\,
  (t;t)_{j_1}\,(s^2;t)_{i_2}\,(x s^{-1};t)_{j_1-i_2}}
  {(t;t)_{j_2}\,(t;t)_{i_2}\,(s^2;t)_{i_1}\,(xs;t)_{i_1+j_1}}
  \\
  \times\;
  {}_{4}\bar\phi_3\!\left(\begin{matrix}
  t^{-i_1};\,t^{-i_2},\,t^{J}sx,\,t\,s x^{-1}\\
  s^2,\,t^{1+j_1-i_2},\,t^{1+J-i_1-j_1}\end{matrix}
  \;\middle|\; t,\,t\right).
\end{multline*}
Here ${}_{r+1}\bar\phi_r$ is the regularized terminating basic hypergeometric
series:
\begin{multline}
\label{eq:reg_bphi}
  {}_{r+1}\bar\phi_r\!\left(\begin{matrix}
  t^{-n};\,a_1,\dots,a_r\\b_1,\dots,b_r\end{matrix}
  \;\middle|\; t,\,z\right)
  =\sum_{k=0}^{n} z^k\,\frac{(t^{-n};t)_k}{(t;t)_k}
  \prod_{i=1}^{r}(a_i;t)_k\,(b_it^{k};t)_{n-k}
  \\
  =\prod_{i=1}^{r}(b_i;t)_n\cdot{}_{r+1}\phi_r\!\left(\begin{matrix}
  t^{-n},\,a_1,\dots,a_r\\b_1,\dots,b_r\end{matrix}
  \;\middle|\; t,\,z\right).
\end{multline}
The displayed spin-fused weight $w^{(J)}_{x,s}$ reduces to the spin Hall--Littlewood weight of
\Cref{rem:spin_hl} at $J=1$, and to \eqref{eq:fused_weight} in the limit
$s\to0$. The latter is a removable singularity and not a substitution: already
at $J=1$ and $(i_1,j_1;i_2,j_2)=(1,0;1,0)$ the factors $s^{\,j_2-i_1}$ and
$(xs^{-1};t)_{j_1-i_2}$ blow up and vanish separately, while the whole
expression equals $(1-stx)/(1-sx)$ and tends to the spinless value $1$. Here $s$ is
the parameter of the vertical (generic Verma) representation
of $U_t(\widehat{\mathfrak{sl}_2})$.
Setting 
$s^2=t^{-I}$
caps the vertical occupations at~$I$.
\end{remark}

\section{Deformed weights and particle systems}
\label{sec:def_rsk}

\subsection{The $q$-Whittaker case and $q$-PushTASEP}
\label{subsec:qw_rsk}

Replace the Schur weights of \Cref{sec:rsk} by their $q$-Whittaker deformation: the
up-right weight $W$ \eqref{eq:qw_weights} and the corresponding down-right dual
$W^{*}$ \eqref{eq:qw_weights_dual}. The YBE
\eqref{eq:ybe} persists, with a $q$-deformed cross $\mathbb R^{(q)}_{x,y}$ in place of
\eqref{eq:bbR}. The cross is the unique (up to scale)
solution of \eqref{eq:ybe}; it obeys the same arrow conservation $i_1+j_2=j_1+i_2$ as
\eqref{eq:bbR}, depends on $x,y$ only through the product $xy$, and degenerates to
\eqref{eq:bbR} as $q\to0$. Its entries include
\begin{equation}
  \begin{gathered}
  \mathbb R^{(q)}_{x,y}(0,0;m,m)=\frac{(xy)^{m}}{(q;q)_{m}},\qquad
  \mathbb R^{(q)}_{x,y}(m,m;0,0)=(q;q)_{m},\\[3pt]
  \mathbb R^{(q)}_{x,y}(1,1;1,1)=xy+q,\qquad
  \mathbb R^{(q)}_{x,y}(1,2;1,2)=\mathbb R^{(q)}_{x,y}(2,1;2,1)=xy(1+q)+q^{2},\\[3pt]
  \mathbb R^{(q)}_{x,y}(2,2;1,1)=(1-q^{2})\bigl(xy+q+q^{2}\bigr),\qquad
  \mathbb R^{(q)}_{x,y}(2,2;2,2)=(xy)^{2}+xy\,q(1+q)^{2}+q^{4},
  \end{gathered}
\end{equation}
and in general it is the regularized terminating ${}_{2}\bar\phi_{1}$ series
\begin{equation}
  \label{eq:qw_cross_full}
  \mathbb R^{(q)}_{x,y}(i_1,j_1;i_2,j_2)
  =\mathbf 1_{i_1+j_2=j_1+i_2}\,
  \frac{(q;q)_{j_1}}{(q;q)_{j_2}\,(q;q)_{i_2}}\,(xy)^{i_2}\,
  {}_{2}\bar\phi_{1}\!\left(\begin{matrix}q^{-i_2};\,q^{-i_1}\\ q^{\,1+j_2-i_2}\end{matrix}
  \;\middle|\;q,\,\frac{q^{\,i_1+j_2+1}}{xy}\right),
\end{equation}
where ${}_{2}\bar\phi_{1}$ is the regularized terminating series
\eqref{eq:reg_bphi}.
Written out, the $q$-deformed cross is the finite sum
\begin{multline}
  \label{eq:qw_cross_sum}
  \mathbb R^{(q)}_{x,y}(i_1,j_1;i_2,j_2)
  =\mathbf 1_{i_1+j_2=j_1+i_2}\,
  \frac{(q;q)_{j_1}}{(q;q)_{j_2}\,(q;q)_{i_2}}
  \\\times\sum_{k=0}^{i_2}
  \frac{(q^{-i_2};q)_{k}\,(q^{-i_1};q)_{k}}{(q;q)_{k}}\,
  (q^{\,1+j_2-i_2+k};q)_{i_2-k}\,q^{(i_1+j_2+1)k}\,(xy)^{i_2-k}.
\end{multline}

\begin{remark}
This $q$-deformation of the benign expression $(xy)^{\min(i_2,j_2)}$ in \eqref{eq:bbR} 
may appear contrived, but it is required to satisfy the YBE.
There also exists a further spin deformation, see
\cite{BufetovMucciconiPetrov2018}, \cite{MucciconiPetrov2020}.
\end{remark}

\begin{remark}
\label{rmk:row-sum}
The sum over the outputs
of the cross
is independent of $(i_1,j_1)$:
\begin{equation}
  \label{eq:qw_cross_rowsum}
  \sum_{i_2,j_2}\mathbb R^{(q)}_{x,y}(i_1,j_1;i_2,j_2)=\frac{1}{(xy;q)_\infty}
  \qquad\text{for every }(i_1,j_1).
\end{equation}
For $(i_1,j_1)=(0,0)$ this is the 
(infinite) $q$-binomial theorem
$$\sum_{m\ge0}\frac{(xy)^m}{(q;q)_m}=\frac{1}{(xy;q)_\infty}.$$
This identity is the 
$q$-Whittaker Cauchy kernel.
For $q=0$, it reduces to the geometric series $\sum_{a\ge0}(xy)^{a}=1/(1-xy)$
which appeared throughout the Schur case.
\end{remark}

Let us apply the bijectivization in the same setting as the one which produced the
RSK correspondence in \Cref{subsec:schur_yb}.

In the bulk (we assume $0<q<1$ and $x,y>0$ here; the summands are then
nonnegative, and zero-weight states are discarded as in
\Cref{sec:bijectivization}) the $q$-deformation turns the deterministic move of
\Cref{lemma:schur_yb_bijection} into a random coupling: $k_1,k_1'$ still range over the
equal-length intervals \eqref{eq:schur_yb_k1_bounds}, but the summands
\eqref{eq:schur_yb_summands} now carry $q$-binomial factors, so, for all but
special boundary occupations,\footnote{For example, at
$(i_1,i_2,i_3;j_1,j_2,j_3)=(0,0,1;0,0,1)$ both lists are $\{1,xy\}$ and the
matching stays deterministic.} only the two
\emph{sums} (and not the individual summands)
match, and so $\mathsf p^{\mathrm{fwd}}(k_1\to k_1')$ must spread the probabilistic mass.

For the instance of \Cref{fig:intro_ybe_instance}, with boundary occupations
$(i_1,i_2,i_3;j_1,j_2,j_3)=(3,2,4;3,4,2)$, each of the six summands
\eqref{eq:schur_yb_summands} deforms into $x^{4}y^{3}(1+q)(1+q^{2})$ times a
polynomial in $xy$ and $q$. Dividing out this common factor, the two sides of
\eqref{eq:ybe} read
\begin{multline}
\label{eq:qw_intro_instance}
\underbrace{(1+q^2+q^4)(1+q+q^2+q^3+q^4)}_{k_1=1}
+\underbrace{(q^2+xy)(1+q+q^2)(1+q+q^2+q^3+q^4)}_{k_1=2}\\
+\underbrace{(xy)^2(1+q+q^2)+q^2\,xy\,(1+q)(1+q+q^2)+q^6}_{k_1=3}\\
=\underbrace{(1+q^2)(1+q+q^2+q^3+q^4)}_{k_1'=0}
+\underbrace{(1+q^2)(1+q+q^2)\bigl(q^2(1+q+q^2)+xy\,(1+q)\bigr)}_{k_1'=1}\\
+\underbrace{(1+q+q^2)\bigl((xy)^2+q^2\,xy\,(1+q+q^2)+q^6\bigr)}_{k_1'=2}\,,
\end{multline}
where the left summands are indexed by $k_1$ as in \eqref{eq:schur_yb_k1_bounds},
with $\varkappa_1=7-k_1$ and $\nu_2=4+k_1'$ in the notation of
\Cref{subsec:intro_worked_example}. At $q=0$, the display collapses to
$1+xy+(xy)^{2}$ on both sides, which is \eqref{eq:intro_two_sides} stripped of
$x^{4}y^{3}$. For $0<q<1$ no term-by-term matching survives: the $k_1=1$ and
$k_1'=0$ summands are the only ones free of $xy$, and they are different
polynomials in $q$. At $q=\frac12$ and $x=y=\frac12$ the six summands become
proportional to $651+434+74=620+455+84$, with no value shared between the two
sides, and the independent coupling ---
with $\mathsf p^{\mathrm{fwd}}(k_1\to k_1')$ proportional to the $k_1'$-th
summand, for every $k_1$ --- has
$\mathsf p^{\mathrm{fwd}}=\frac{620}{1159},\,\frac{455}{1159},\,\frac{84}{1159}$ for
$k_1'=0,1,2$, respectively. (At $xy=1$ the two sides accidentally coincide as
multisets, $651+1085+620=620+1085+651$; an equality of values at a single
point does not produce a matching of the polynomial summands.)

However, at the leftmost column the bijectivization becomes unique after summing
the left-hand side of \eqref{eq:ybe} over the internal occupations. There the
cross pushes its arrows into the system through the reservoir of infinitely many
vertical paths, see
\Cref{rmk:infinite_boundary_at_column_0} and \Cref{fig:crosses_left}. The
vertex weights with infinitely many vertical arrows
do not depend on the incoming horizontal occupation:
$W(\infty,j_1;\infty,j_2)=x^{j_2}/(q;q)_{j_2}$, the $q$-deformation of
\Cref{rmk:infinite_boundary_at_column_0}.
Hence, in the YBE with $i_1,i_2$ fixed, we may sum
over the internal occupations $k_1,k_2$ in the left-hand side of \eqref{eq:ybe}, and by
\eqref{eq:qw_cross_rowsum} that side collapses to a \emph{single} term:
\begin{equation}
  \label{eq:qw_singleton_identity}
  \frac{1}{(xy;q)_\infty}\,
  \frac{x^{j_2}}{(q;q)_{j_2}}\,\frac{y^{j_1}}{(q;q)_{j_1}}
  \;=\;
  \sum_{k_1',k_2'}
  \frac{x^{k_2'}}{(q;q)_{k_2'}}\,\frac{y^{k_1'}}{(q;q)_{k_1'}}\,
  \mathbb R^{(q)}_{x,y}(k_2',k_1';j_2,j_1).
\end{equation}
Note that both sides are independent of the fixed $i_1,i_2$.
A bijectivization of an identity with a singleton on one side is unique
(\Cref{sec:bijectivization}). Under it, 
the pair $(k_1',k_2')$ is drawn
with probability proportional to its weight in the right-hand side of
\eqref{eq:qw_singleton_identity}. 

Let us use the 
symmetry
\begin{equation*}
  \frac{x^{k_2'}\,y^{k_1'}}{(q;q)_{k_2'}\,(q;q)_{k_1'}}\,
  \mathbb R^{(q)}_{x,y}(k_2',k_1';j_2,j_1)
  =\frac{x^{j_2}\,y^{j_1}}{(q;q)_{j_2}\,(q;q)_{j_1}}\,
  \mathbb R^{(q)}_{x,y}(j_2,j_1;k_2',k_1')
\end{equation*}
of the $q$-deformed cross, and denote
these probabilities to pick $(k_1',k_2')$ by 
\begin{equation}
  \label{eq:qw_stochasticization}
  \mathbb L_{x,y}(i_1,j_1;i_2,j_2)
  \coloneqq\frac{\mathbb R^{(q)}_{x,y}(i_1,j_1;i_2,j_2)}
        {\sum_{i_2',j_2'}\mathbb R^{(q)}_{x,y}(i_1,j_1;i_2',j_2')}
  =(xy;q)_\infty\,\mathbb R^{(q)}_{x,y}(i_1,j_1;i_2,j_2).
\end{equation}
In the regime $0\le q<1$, $x,y\ge0$ with $xy<1$ of \Cref{subsec:conventions}, we have
$\mathbb L_{x,y}\ge0$ and $\sum_{i_2,j_2}\mathbb L_{x,y}(i_1,j_1;i_2,j_2)=1$.
Both facts are established in the Appendix of
\cite{BufetovMucciconiPetrov2018} for the more general spin weights: the
constant row sum is \cite[Proposition~A.5]{BufetovMucciconiPetrov2018}, and
nonnegativity is \cite[Proposition~A.8]{BufetovMucciconiPetrov2018}, proved
there for $q\in(0,1)$ and spin parameter $s\in[-\sqrt q,0)$; the $s=0$ weights
\eqref{eq:qw_cross_sum} are covered by the limit $s\to0^-$. Proposition~A.8
also assumes its spectral parameters lie in $[-s,-s^{-1}]$; for fixed $x,y>0$
this holds for all $s<0$ close enough to $0$. The boundary cases $q=0$, $x=0$,
or $y=0$ follow by continuity, or directly from the $q=0$ Schur weights.

\medskip

Let us explain a connection to particle systems.
Applied across the lattice, the moves grow a field $\{\lambda^{(m,n)}\}$
of Young diagrams on the quadrant $\mathbb Z_{\ge0}^2$ --- a $q$-Whittaker
Yang--Baxter field ---
whose \emph{first parts} $\lambda_1^{(m,n)}$ encode the particles of the (discrete-time)
\emph{geometric $q$-PushTASEP} of \cite[\S6.3]{MatveevPetrov2014} (the $s=0$ case
of the $q$-Hahn PushTASEP of \cite[\S3.2.1]{CMP_qHahn_Push}). Concretely: set
$\lambda^{(m,0)}=\lambda^{(0,n)}=\varnothing$ on the axes, fix an admissible
forward kernel for every bulk instance once and for all, and apply these
kernels cell by cell in increasing $m+n$, following
\cite{BufetovMucciconiPetrov2018}. The bulk coupling is not
unique --- any bijectivization of the bulk instances of \eqref{eq:ybe} may be
used; the resulting field depends on this choice, while the marginals used
below do not \cite{BufetovPetrovYB2017}, \cite{BufetovMucciconiPetrov2018}.
Along any down-right cut (a lattice path in $\mathbb Z_{\ge0}^{2}$ made of unit
steps to the right and down) the field has the law of the $q$-Whittaker process
(the $t=0$ Macdonald process) of \cite{BorodinCorwin2011Macdonald} --- whence
the name --- independently of the bulk coupling. The four
diagrams at the corners of a unit cell,
\begin{equation*}
  \varkappa=\lambda^{(m,n)},\quad
  \mu=\lambda^{(m+1,n)},\quad
  \lambda=\lambda^{(m,n+1)},\quad
  \nu=\lambda^{(m+1,n+1)},
\end{equation*}
obey the local rule which randomly picks $\nu$
based on $\mu,\varkappa,\lambda$ and an integer $k$
(the analogue of the RSK input $k_{ab}$ from the proof of \Cref{prop:train_is_rsk}).
The input $k$ is now random: $q$-geometric with parameter $x_{n+1}y_{m+1}$,
that is, $\mathbb P(k)=(x_{n+1}y_{m+1};q)_\infty\,(x_{n+1}y_{m+1})^k/(q;q)_k$,
independently over the cells \cite[\S6.2]{MatveevPetrov2014}. The kernel
$\mathbb L$ enforces the mandatory lower bound
$\nu_1-\lambda_1\ge\max(0,\mu_1-\lambda_1)$, but the excess over it is not in
general the independent input $k$: later particles also receive a random push.
The jump is bare $q$-geometric, as in \eqref{eq:qw_first_particle}, whenever
the preceding particle does not move, which is automatic for the leading
particle. At $q=0$ the inputs reduce to the geometric
matrix entries feeding RSK.
Thanks to the unique bijectivization at the zeroth column, the evolution of the
first parts of the partitions is \emph{marginally Markovian}, governed by the
stochasticized cross $\mathbb L_{x_{n+1},y_{m+1}}(i_1,j_1;i_2,j_2)$ with
\begin{equation}
  \label{eq:qw_L_firstparts}
  (i_1,j_1)=(\lambda_1-\varkappa_1,\ \mu_1-\varkappa_1),\qquad
  (i_2,j_2)=(\nu_1-\mu_1,\ \nu_1-\lambda_1),\qquad
  i_1+j_2=j_1+i_2=\nu_1-\varkappa_1.
\end{equation}

The first coordinate of the field plays the role of time. The update
\eqref{eq:qw_L_firstparts} involves only the first parts, so the array
$\bigl(\lambda_1^{(m,n)}\bigr)_{n\ge1}$ evolves as a Markov chain in $m$, which we
describe as a particle system. Set
\begin{equation}
  \label{eq:qw_particles}
  P_n(m):=\lambda_1^{(m,n)}+n,\qquad n\ge1,\qquad P_0\equiv0.
\end{equation}
By the interlacing, $P_0<P_1(m)<P_2(m)<\cdots$, and the particles move to the right,
pushing one another. One step $m\to m+1$ updates them in order $n=1,2,\ldots$: given
the move $\ell$ of $P_{n-1}$ and the gap $g=P_n(m)-P_{n-1}(m)-1$, the particle
$P_n$ moves right by $L$ with probability
$\mathbb L_{x_n,y_{m+1}}(g,\ell;\,g+L-\ell,\,L)$.
Indeed, the four counts inside $\mathbb{L}_{x_n,y_{m+1}}$ in the latter expression can be matched to
\eqref{eq:qw_L_firstparts} for the cell with the lower left corner
$\varkappa=\lambda^{(m,n-1)}$, with the outgoing pair $(g+L-\ell,L)$ recording the new
gap and the jump. In particular, the leading particle $P_1$ receives ``push''
$\ell=0$ from the frozen
$P_0$ and jumps $q$-geometrically,
\begin{equation}
  \label{eq:qw_first_particle}
  \mathbb L_{x_1,y_{m+1}}(g,0;\,g+L,\,L)
  =(x_1y_{m+1};q)_\infty\,\frac{(x_1y_{m+1})^{L}}{(q;q)_{L}},
\end{equation}
independently of the gap $g$. The dynamics preserves the order:
$\mathbb L_{x_n,y_{m+1}}=0$ unless $L\ge\ell-g$, so $\ell>g$ forces a push.

\medskip

A continuous-time limit of the $q$-PushTASEP 
(set all $y_m=\varepsilon$ and observe the chain after
$\lfloor\tau/\varepsilon\rfloor$ steps at continuous time $\tau$, as
$\varepsilon\to0$)
produces a simpler dynamics introduced in
\cite{BorodinPetrov2013NN} and further studied in \cite{CorwinPetrov2013}.
Under the continuous-time $q$-PushTASEP,
each particle $n$ carries an independent exponential clock with rate
$x_n/(1-q)$; when it rings the
particle steps forward by one and, with
probability $q^{g}$ (with $g$ the gap to the next particle ahead before the
step), instantaneously pushes that neighbor forward by one as well, the push
cascading down the line. At $q=0$
the push probability is $1$ across a zero gap and $0$ otherwise, so a jump shoves the
whole block of adjacent particles ahead by one. That is, for $q=0$
we get back the classical PushTASEP.

\subsection{Hall--Littlewood: a row-indexed RSK and the six-vertex model}
\label{subsec:hl_rsk}

The Hall--Littlewood weights \eqref{eq:hl_weights} --- the $t$-deformation of the
higher-spin model of \Cref{rem:higher_spin} --- carry a randomized bijectivization
parallel to the $q$-Whittaker one of \Cref{subsec:qw_rsk}. 
The \emph{horizontal} edges 
in the Hall--Littlewood case
have
capacity one ($j_1,j_2\in\{0,1\}$), while the vertical edges stay unbounded
($i_1,i_2\in\mathbb Z_{\ge0}$) and record multiplicities of the \emph{row}
lengths of the partitions, rather than of their column lengths as in
\Cref{subsec:qw_rsk} (cf.~\Cref{rmk:conjugate_coordinates}). In this sense the
insertion is indexed by rows, and, as we now explain,
its scalar marginal is the \emph{stochastic six-vertex model} rather than the
$q$-PushTASEP.

\medskip

Recall the vertex weights entering the Hall--Littlewood Cauchy identity
\eqref{eq:hl_cauchy}. The up-right weight is \eqref{eq:hl_weights}, in formulas
\begin{equation}
  \label{eq:hl_W_formula}
  W(i_1,j_1;i_2,j_2)=\mathbf 1_{i_1+j_1=i_2+j_2}\times
  \begin{cases}
    1,& (j_1,j_2)=(0,0),\\
    x, & (j_1,j_2)=(1,1),\\
    1,& (j_1,j_2)=(1,0),\\
    x\,(1-t^{\,i_1}),& (j_1,j_2)=(0,1),
  \end{cases}
\end{equation}
whose $N$-row partition function is $P_\lambda(x_1,\dots,x_N;0,t)$
(\Cref{prop:hl}). The $y$-side of \eqref{eq:hl_cauchy} carries the functions
$Q_\lambda$, whose branching rule differs from that of $P_\lambda$: the factor
$1-t^{m}$ is attached to a path \emph{entering} a column whose multiplicity
becomes $m$, rather than leaving a column of multiplicity $m$
\cite[Ch.~III, \S5]{Macdonald1995}. Accordingly, the
down-right weight (the analogue of the reflection \eqref{eq:downright_cross},
with the conservation law $i_2+j_1=i_1+j_2$) is
\begin{equation}
  \label{eq:hl_Wstar_formula}
  W^{*}(i_1,j_1;i_2,j_2)=\mathbf 1_{i_2+j_1=i_1+j_2}\times
  \begin{cases}
    1,& (j_1,j_2)=(0,0),\\
    y, & (j_1,j_2)=(1,1),\\
    1-t^{\,i_1},& (j_1,j_2)=(1,0),\\
    y,& (j_1,j_2)=(0,1).
  \end{cases}
\end{equation}
Gluing $N$ rows of the up-right model and $M$ rows of the down-right model along a
common boundary $\lambda$, as in
\Cref{fig:cauchy_model}, produces the vertex model for the left-hand side of
the Cauchy identity \eqref{eq:hl_cauchy}.

\medskip

The YBE
\eqref{eq:ybe} with the weights
\eqref{eq:hl_W_formula}--\eqref{eq:hl_Wstar_formula} has the unique (up to
scale) solution
\begin{equation}
  \label{eq:hl_cross}
  \begin{gathered}
  \mathbb R^{(t)}_{x,y}(0,0;0,0)=1,\qquad
  \mathbb R^{(t)}_{x,y}(1,0;1,0)=\mathbb R^{(t)}_{x,y}(0,1;0,1)
  =\frac{1-txy}{1-xy},\qquad
  \mathbb R^{(t)}_{x,y}(1,1;1,1)=t,\\[3pt]
  \mathbb R^{(t)}_{x,y}(0,0;1,1)=\frac{xy\,(1-t)}{1-xy},\qquad
  \mathbb R^{(t)}_{x,y}(1,1;0,0)=\frac{1-t}{1-xy}.
  \end{gathered}
\end{equation}
Like the $q$-Whittaker cross \eqref{eq:qw_cross_full}, it depends on $x,y$
only through the product $xy$.

\begin{remark}
\label{rmk:hl_row_sum}
All row sums of the cross coincide (cf.\ \Cref{rmk:row-sum}):
\begin{equation}
  \label{eq:hl_cross_rowsum}
  \sum_{i_2,j_2\in\{0,1\}}\mathbb R^{(t)}_{x,y}(i_1,j_1;i_2,j_2)
  =\frac{1-txy}{1-xy}
  \qquad\text{for every }(i_1,j_1),
\end{equation}
which is the building block for the Hall--Littlewood Cauchy kernel. 
\end{remark}

Dragging the
$NM$ crosses $\mathbb R^{(t)}_{x_i,y_j}$, one for each of the $NM$ intersections of the
$N$ up-right lines with the $M$ down-right ones, through the lattice exactly as in
\Cref{subsec:ybe_cauchy} yields the Yang--Baxter proof of the Cauchy identity
\eqref{eq:hl_cauchy}.
This Yang--Baxter proof is essentially the same as the one in
\cite{wheeler2015refined},
\cite{Borodin2014vertex},
and can be traced back to the $q$-boson realization of the
Hall--Littlewood functions in \cite{tsilevich2006quantum}.

Stochasticize the cross as in \eqref{eq:qw_stochasticization},
\begin{equation}
  \label{eq:hl_stochasticization}
  \mathbb L^{(t)}_{x,y}(i_1,j_1;i_2,j_2)
  \coloneqq\frac{1-xy}{1-txy}\,\mathbb R^{(t)}_{x,y}(i_1,j_1;i_2,j_2),
\end{equation}
and abbreviate
\begin{equation}
  \label{eq:s6v_weights}
  b_1\coloneqq\frac{1-xy}{1-txy},\qquad
  b_2\coloneqq t\,\frac{1-xy}{1-txy}=t\,b_1,
  \qquad 0\le b_2\le b_1\le1.
\end{equation}
The six transition
probabilities become
\begin{equation}
  \label{eq:hl_L_entries}
  \begin{gathered}
  \mathbb L^{(t)}_{x,y}(1,0;1,0)=\mathbb L^{(t)}_{x,y}(0,1;0,1)=1,
	\\[2pt]
  \mathbb L^{(t)}_{x,y}(0,0;0,0)=b_1,\qquad
  \mathbb L^{(t)}_{x,y}(0,0;1,1)=1-b_1,
	\\[2pt]
  \mathbb L^{(t)}_{x,y}(1,1;1,1)=b_2,\qquad
  \mathbb L^{(t)}_{x,y}(1,1;0,0)=1-b_2.
  \end{gathered}
\end{equation}
\Cref{fig:hl_cross_s6v}, top row.

As in the $q$-Whittaker case, at the leftmost column the bijectivization is
unique. With the reservoir weights $W(\infty,j_1;\infty,j_2)=x^{j_2}$ and
$W^{*}(\infty,j_1;\infty,j_2)=y^{j_2}$
(\Cref{rmk:infinite_boundary_at_column_0}; the factors $1-t^{\infty}=1$
disappear), summing the YBE over the internal occupations
collapses its left-hand side to a single term, and the analogue of
\eqref{eq:qw_singleton_identity} reads
\begin{equation}
  \label{eq:hl_singleton_identity}
  \frac{1-txy}{1-xy}\; x^{j_2}\,y^{j_1}
  =\sum_{k_1',k_2'\in\{0,1\}}
  x^{k_2'}\,y^{k_1'}\;\mathbb R^{(t)}_{x,y}(k_2',k_1';j_2,j_1).
\end{equation}
The same symmetry of the cross as in the $q$-Whittaker case identifies the
resulting unique coupling with $\mathbb L^{(t)}_{x,y}$
\eqref{eq:hl_stochasticization}.

\medskip

Dragging the crosses through the Hall--Littlewood Cauchy lattice --- the
construction of \cite{BufetovPetrovYB2017} --- grows a field
$\{\lambda^{(m,n)}\}$ of partitions on the quadrant, with $\lambda^{(m,n)}$
distributed as the Hall--Littlewood measure
\begin{equation}
  \label{eq:hl_measure}
  \mathrm{Prob}(\lambda)\propto
  P_\lambda(x_1,\dots,x_n;0,t)\,Q_\lambda(y_1,\dots,y_m;0,t),
\end{equation}
the $t$-analogue of the Schur measure. Along any down-right cut the field is
the \emph{Hall--Littlewood process}, the $t$-deformation of the Schur process
\cite{okounkov2003correlation}. In the bulk the coupling is randomized (as in
\Cref{subsec:qw_rsk}), while at the leftmost column it is the unique
\eqref{eq:hl_singleton_identity}. Around a unit cell with corners
$\varkappa=\lambda^{(m,n)}$, $\mu=\lambda^{(m+1,n)}$,
$\lambda=\lambda^{(m,n+1)}$, $\nu=\lambda^{(m+1,n+1)}$, the stochasticized
cross $\mathbb L^{(t)}_{x_{n+1},y_{m+1}}$ governs the \emph{lengths} of the
partitions:
\begin{equation}
  \label{eq:hl_L_lengths}
  (i_1,j_1)=\bigl(\ell(\lambda)-\ell(\varkappa),\ \ell(\mu)-\ell(\varkappa)\bigr),
  \qquad
  (i_2,j_2)=\bigl(\ell(\nu)-\ell(\mu),\ \ell(\nu)-\ell(\lambda)\bigr),
\end{equation}
with $i_1+j_2=j_1+i_2=\ell(\nu)-\ell(\varkappa)$. This is the row-indexed
counterpart of \eqref{eq:qw_L_firstparts}, with the first parts $\lambda_1$
replaced by the lengths $\ell(\lambda)=\lambda_1'$
(cf.~\Cref{rmk:conjugate_coordinates}). In particular, the array
of lengths --- equivalently, of the numbers of zero parts
$m_0(\lambda^{(m,n)})=n-\ell(\lambda^{(m,n)})$ --- evolves 
in a
marginally
Markovian way.

\medskip

To recognize this marginal, recall the stochastic
six-vertex model
\cite{GwaSpohn1992}, \cite{BCG6V}.
The general six-vertex model on $\mathbb Z^{2}$ has six admissible
arrow configurations around a vertex, weighted
$a_1,a_2,b_1,b_2,c_1,c_2\ge0$ (see \Cref{fig:hl_cross_s6v}, bottom row, and
\cite{baxter2007exactly}). The weights are called \emph{stochastic}
\cite{GwaSpohn1992} if
\begin{equation}
  \label{eq:s6v_condition}
  a_1=a_2=1,\qquad b_1+c_1=b_2+c_2=1,\qquad b_1,b_2\in[0,1].
\end{equation}
Then the weight of each vertex is the conditional probability of its outgoing
arrows (through the top and the right) given the incoming ones (from below and
from the left), cf.\ \Cref{fig:stochastic_vertex}, and the model in the
quadrant $\mathbb Z_{\ge1}^{2}$ is sampled by a Markov chain, vertex by vertex
along the successive diagonals $\{m+n=\mathrm{const}\}$. Impose the
\emph{infinite domain wall} boundary conditions: a path enters from the left
at every row, and no paths enter through the bottom. The \emph{height
function} is
\begin{equation}
  \label{eq:s6v_height}
  H(m,n)\coloneqq
  \#\bigl\{\text{up-right paths passing strictly to the right of }(m,n)\bigr\},
  \qquad 0\le H(m,n)\le n,
\end{equation}
see \Cref{fig:s6v_height} for an illustration.

\begin{figure}[ht]
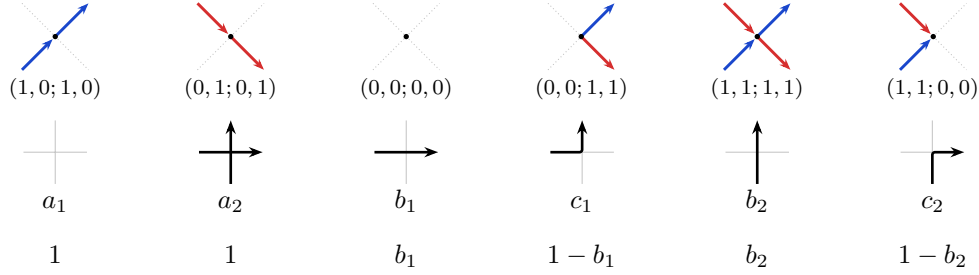

\centering
{\small
\begin{tabular}{cccccc}
\hcross{1}{0}{1}{0} & \hcross{0}{1}{0}{1} & \hcross{0}{0}{0}{0} &
\hcross{0}{0}{1}{1} & \hcross{1}{1}{1}{1} & \hcross{1}{1}{0}{0} \\[2pt]
\makebox[1.9cm]{\scriptsize$(1,0;1,0)$} &
\makebox[1.9cm]{\scriptsize$(0,1;0,1)$} &
\makebox[1.9cm]{\scriptsize$(0,0;0,0)$} &
\makebox[1.9cm]{\scriptsize$(0,0;1,1)$} &
\makebox[1.9cm]{\scriptsize$(1,1;1,1)$} &
\makebox[1.9cm]{\scriptsize$(1,1;0,0)$} \\[6pt]
\fvv{} &
\fvv{\draw[line width=1pt,->] (0,-0.6)--(0,0.6);
     \draw[line width=1pt,->] (-0.6,0)--(0.6,0);} &
\fvv{\draw[line width=1pt,->] (-0.6,0)--(0.6,0);} &
\fvv{\draw[line width=1pt,->] (-0.6,0)--(0,0)--(0,0.6);} &
\fvv{\draw[line width=1pt,->] (0,-0.6)--(0,0.6);} &
\fvv{\draw[line width=1pt,->] (0,-0.6)--(0,0)--(0.6,0);} \\[2pt]
$a_1$ & $a_2$ & $b_1$ & $c_1$ & $b_2$ & $c_2$ \\[9pt]
$1$ & $1$ & $b_1$ & $1-b_1$ & $b_2$ & $1-b_2$ \\
\end{tabular}}
\caption{Top row: the six states $(i_1,j_1;i_2,j_2)$ of the stochasticized
cross $\mathbb L^{(t)}_{x,y}$ \eqref{eq:hl_L_entries}
(\textcolor{upcol}{blue} = the $P$-line carrying $x$,
\textcolor{downcol}{red} = the $Q$-line carrying $y$; dotted = empty).
Bottom row: the corresponding vertices of the stochastic six-vertex model,
in the standard order $a_1,a_2,b_1,c_1,b_2,c_2$. The cross state
$(i_1,j_1;i_2,j_2)$ becomes the vertex whose (bottom, left; top, right)
occupations are $(j_1,1-i_1;j_2,1-i_2)$.}
\label{fig:hl_cross_s6v}
\end{figure}

\begin{figure}[ht]
\centering
\begin{tikzpicture}[scale=0.62,>={Stealth[length=1.6mm]},rounded corners=1.0pt,font=\small]
  \foreach \c in {1,...,6}{\draw[gridln] (\c,0.4)--(\c,4.6);}
  \foreach \r in {1,...,4}{\draw[gridln] (0.4,\r)--(6.6,\r);}
  \draw[uppath,->] (0.4,1)--(3,1)--(3,2)--(5,2)--(5,3)--(6.6,3);
  \draw[uppath,->] (0.4,2)--(2,2)--(2,3)--(4,3)--(4,4)--(6.6,4);
  \draw[uppath,->] (0.4,3)--(1,3)--(1,4.6);
  \draw[uppath,->] (0.4,4)--(2,4)--(2,4.6);
  \foreach \x/\y/\h in {1/1/1, 2/2/1, 2/3/2, 5/1/0, 1/4/3, 6/3/1}{
    \node[vtx] at (\x,\y) {};
    \node[fill=white,inner sep=1pt,font=\scriptsize,anchor=north west]
      at ({\x+0.06},{\y-0.09}) {$\h$};
  }
\end{tikzpicture}
\caption{A configuration of the stochastic six-vertex model in the quadrant
with the infinite domain wall boundary conditions, and several values of the
height function \eqref{eq:s6v_height}. Paths leaving the picture on the right
continue and eventually turn up.}
\label{fig:s6v_height}
\end{figure}
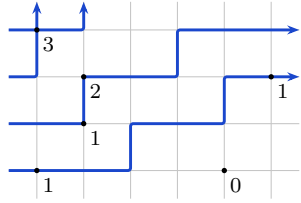

The next statement is proven in the same way as its $q$-Whittaker counterpart
from \Cref{subsec:qw_rsk}:
\begin{proposition}[{\cite{borodin2016stochastic_MM}, \cite{BorodinBufetovWheeler2016}, \cite{BufetovMatveev2017}, \cite{BufetovPetrovYB2017}}]
\label{prop:hl_s6v}
Under the Yang--Baxter field $\{\lambda^{(m,n)}\}$
constructed from the bijectivization of the Hall--Littlewood YBE,
the array
$m_0(\lambda^{(m,n)})$ evolves as the height function of the
stochastic six-vertex model in the quadrant with the infinite domain wall
boundary conditions, with $b_1,b_2$ given by
\eqref{eq:s6v_weights} under $xy\rightsquigarrow x_ny_m$ at the coordinate $(m,n)$.
\end{proposition}

With homogeneous parameters $x_ny_m\equiv xy$, in the Poisson-type limit around
the diagonal ($xy\to1$, so that $b_1,b_2\to0$ with the ratio $b_2/b_1=t$ fixed
by \eqref{eq:s6v_weights}) the six-vertex model degenerates to the asymmetric
simple exclusion process (ASEP) with left and right jump rates $t$ and $1$
\cite{BCG6V}. Parameter convergence alone is not enough here: one observes
the particle configuration along the diagonal and accelerates the diagonal coordinate
by the order $b_1^{-1}$, so that jumps, which occur with probabilities of order $b_1$
and $b_2$, happen at rates $1$ and $t$ per unit of limiting time. See \cite{BCG6V} for
the precise scaling, and \cite{Aggarwal2017convergence} for the proof and the mode of
convergence.
This is the row-indexed counterpart of the $q$-PushTASEP
produced from columns in \Cref{subsec:qw_rsk}. At $t=0$ the Hall--Littlewood
weights \eqref{eq:hl_weights} become the Schur weights
(\Cref{rem:higher_spin}), and the height field reduces to the corresponding
Schur-measure observable. 
The particle system along the diagonal becomes the continuous-time TASEP.

\section{Further directions}
\label{sec:outlook}

Reading the Yang--Baxter equation as a \emph{probabilistic normalization} identity between two
random configurations, and not only an algebraic one, is fruitful: dragging a cross through a
lattice gave the Cauchy
identity (\Cref{subsec:ybe_cauchy}), RSK and its deformations (\Cref{sec:rsk,sec:def_rsk}),
and their particle-system marginals.
Here we collect further directions in which this reading
has been used.

\subsection{Stationary measures via the forgetting property}
\label{sec:colored}

For the stochastic vertex models of these notes, sending the number of arrows
carried by one line to
infinity turns that line into an inexhaustible reservoir that ``forgets'' how many arrows entered, so its
output is independent of its input --- a property of the specific weights, not
of stochasticity alone (we saw it first in \Cref{rmk:infinite_boundary_at_column_0}).
Then, the 
YBE
allows one to exchange this infinite line with another, finite stochastic line (encoding the evolution
of a particle system). Since the distribution before and after applying the stochastic line is the
same, such ``reservoir measures'' are \emph{stationary} for particle systems.
This approach was used to construct stationary measures
for the colored ASEP, the colored $q$-Boson process, and the
colored $q$-PushTASEP, on the ring and on the line (the latter through a
stationary vertex model in the quadrant) by
Aggarwal--Nicoletti--Petrov \cite{AggarwalNicolettiPetrov2023colored}.
For a single color the quadrant construction is due to Aggarwal \cite{Amol2016Stationary}.

\medskip

The reservoir weights are a lattice form of the multiline queues of
Ferrari--Martin \cite{Ferrari_2007} and Martin \cite{martin2020stationary},
and the underlying single-color Burke property appeared earlier for
directed polymers in O'Connell--Yor \cite{OConnellYor2001} and
Sepp\"al\"ainen \cite{Seppalainen2012}.
Stationary measures for multiclass particle systems also admit an algebraic description via
the Matrix Product Ansatz, which goes back to Derrida--Evans--Hakim--Pasquier
\cite{Derrida1993solution} and Derrida--Janowsky--Lebowitz--Speer \cite{Derrida1993shock},
and was extended to several species by Mallick--Mallick--Rajewsky
\cite{MallickMallickRajewsky1999} and Prolhac--Evans--Mallick \cite{Prolhac_2009};
see the survey of Blythe--Evans
\cite{BlytheEvansSolverGuide2007}. The queueing constructions of Angel \cite{angel2006stationary}
and Ferrari--Martin \cite{FerrariMartin2005}, \cite{ferrari2009multiclass} give a combinatorial
counterpart, and letting the number of classes grow leads to the TASEP speed process of
Amir--Angel--Valk\'o \cite{amir2011tasep}.
The stationary measures tie to
nonsymmetric Macdonald polynomials, as in Cantini--de Gier--Wheeler \cite{cantini2015matrix},
Corteel--Mandelshtam--Williams \cite{corteel2018multiline}, and
Borodin--Wheeler \cite{borodin2019nonsymmetric}, \cite{borodin_wheeler2018coloured};
and, for the colored $q$-Boson
process, to modified Macdonald polynomials, as in Ayyer--Mandelshtam--Martin
\cite{Ayyer_2023}, \cite{ayyer2022modified}.

\subsection{Symmetry in the spectral parameters}
\label{sec:parameter_symmetry}

The partition functions of \Cref{sec:schur_models,sec:qdef} build symmetric polynomials one row
at a time, feeding in the spectral parameters in a fixed order; that they come out symmetric is a
corollary of the YBE, since an $R$-cross (of a different form than the ones we
use here) dragged through two adjacent rows interchanges their parameters while fixing the
partition function. Bijectivizing this single swap turns the symmetry into a local Markov move on
semistandard Young tableaux of a fixed shape: the move preserves the partition function but
carries the tableau distribution for one order of the parameters to the one for the swapped
order. It is a probabilistic version of the Bender--Knuth involution \cite{bender1972enumeration}.

\medskip

Iterating this single swap up an infinite interlacing array drags the bottom parameter off to
infinity, and
homogeneity turns the residual rescaling into a shift $t\mapsto qt$ of the process time
(here $q$ is the ratio of the geometric family of TASEP speeds and $t$ is the
running time --- not the Macdonald parameters of \Cref{sec:qdef}). In the limit $q=e^{-\varepsilon}\to1$ with
$\lfloor\tau/\varepsilon\rfloor$ iterations of this drag, this becomes a
time-homogeneous \emph{backward} process:
Petrov--Saenz \cite{PetrovSaenz2019backTASEP} build a backward Hammersley-type process with
$\mu_t\,\mathbf L_\tau=\mu_{e^{-\tau}t}$, mapping the fixed-time law $\mu_t$ of TASEP
with the step initial condition to the law at an earlier time $e^{-\tau}t$.

\medskip

The $q$-Hahn TASEP goes back to Povolotsky \cite{Povolotsky2013} and
Corwin \cite{Corwin2014qmunu}; in its multiparameter version, due to
Borodin--Petrov \cite[\S6.6]{BorodinPetrov2016inhom}, each particle
$x_n$ (a position, not a spectral parameter) carries its own jump parameter
$\nu_n$ (coming from the vertex model spin parameter).
The parameter symmetry described next is from Petrov \cite{petrov2019qhahn}.
For $0<\nu_{n+1}<\nu_n<1$ the swap $\nu_n\leftrightarrow\nu_{n+1}$ is a $q$-deformed beta-binomial
redistribution of $x_n$, reproducing the process with the two parameters swapped at every fixed
time. The swap exists because, from the step initial configuration $x_n(0)=-n$,
transposing $\nu_n$ and $\nu_{n+1}$ leaves the joint law of all the other
particles unchanged.
Both statements need that initial configuration; for other deterministic
data the symmetry typically fails already after a single update.
The same elementary
swap degenerates to $q$-TASEP and to the directed beta polymer, where it becomes a
$\mathrm{Beta}$-splitting of two adjacent partition functions.
A random matrix shadow of this probabilistic symmetry was explored in
Petrov--Tikhonov \cite{PetrovTikhonov2019}.

\medskip

The symmetry lifts from fixed-time marginals to whole trajectories.
Indeed,
Petrov--Saenz \cite{petrov2022rewriting}
constructed
``rewriting
history''
Markov operators on the space of trajectories of a particle system such as TASEP.
These operators act by
resampling a particle's entire past trajectory while carrying the joint law onto
that of the parameter-swapped process. Dragging to infinity gives the shift intertwiner,
whose continuous limit is a Lax
equation $t\,\dot{\mathsf T}=[\mathsf B,\mathsf T]$
for the TASEP semigroup $\mathsf T=\mathsf T(t)$, with $\mathsf B$ the
generator of the backward process. This last step is formal:
\cite{petrov2022rewriting} exhibits no common invariant domain on which
$\mathsf B\mathsf T$ and $\mathsf T\mathsf B$ are both defined.
The Lax equation has the same form for
$q$-TASEP and TASEP, and differs only in the form of the generator $\mathsf B$
of the backward process.

\subsection{Shift-invariance and hidden symmetries}
\label{sec:shift_invariance}

The transposition of spectral parameters also surfaces in identities between \emph{distant} observables,
though the phenomenon is less the swap itself than a hidden symmetry that the swap merely records.
Borodin--Gorin--Wheeler \cite{borodin2019shift} prove (under certain admissibility conditions)
that the joint law of colored height
functions of the inhomogeneous colored stochastic six-vertex model is invariant under moving
one observation point up by one lattice step and raising its color cutoff by one, at the cost of
transposing two row spectral parameters.
The identity is nonlocal, as the swapped rows may be
arbitrarily far apart. The proof is analytic, by Lagrange interpolation in the spectral
parameters at the nodes where a row parameter equals a column parameter and the
$R$-matrix degenerates to a permutation.
Fusion, analytic continuation in the spin, and scaling limits
push the invariance down to the Beta, Gamma, and O'Connell--Yor polymers, the KPZ equation,
and the Airy sheet. The degenerate specialization of the vertex model identity is
the color--position symmetry of
Borodin--Wheeler \cite{borodin_wheeler2018coloured} and
Borodin--Bufetov \cite{BorodinBufetov2021ColorPosition}.

\medskip

Independently, Dauvergne \cite{dauvergne2020hidden} settled a conjecture of
\cite{borodin2019shift}
for last passage percolation and directed polymers, by a route close to the present work:
the RSK and geometric RSK correspondences together with a decoupling property.
The cleanest formulation, due to Galashin \cite{galashin2020symmetries}, is a single \emph{flip}
invariance of the partition function (a $180^\circ$ rotation of the left-to-right boundary
connections, with the bottom-to-top ones fixed and the row spectral parameters
reversed), of which the shift is a composition of two flips; the proof
is algebraic, through the Yang--Baxter basis of the Hecke algebra. See also
Bufetov \cite{bufetov2020interacting} and Zhang \cite{Zhang2023}.
The
two sides of a flip carry \emph{equal numbers} of configurations, yet the weights match only in
aggregate. It would be interesting to find useful couplings of the two sides, 
as we have done for the Yang--Baxter equation and Cauchy identities in this work.

\subsection{Markov duality}
\label{sec:duality}

Markov duality extends the integrable structure to the \emph{dynamics} of observables
of particle systems coming from stochastic vertex models. We use the
quantum-group vocabulary below --- co-product, universal $R$-matrix --- only as
pointers to the literature, without definitions.
The
symmetry at work is now the co-product rather than the spectral parameter: the stochastic Markov generator
of a particle system (for discrete-time chains, the transition operator)
commutes with the $U_q$-action; combined with a reversible measure or an
explicit conjugation, this symmetry produces a \emph{duality function}
$D(\eta,\xi)$ intertwining the generator
with the transpose of a dual generator,
\begin{equation*}
  \mathbb E_\eta\bigl[D(\eta_t,\xi)\bigr]=\mathbb E_\xi\bigl[D(\eta,\xi_t)\bigr].
\end{equation*}
Here $\xi_t$ is the dual Markov process, with finitely many particles in the
examples below.
Much of the hierarchy of integrable particle systems (ASEP, after
Sch\"utz \cite{schutz1997dualityASEP} and
Borodin--Corwin--Sasamoto \cite{BorodinCorwinSasamoto2012}, with the symmetric case
going back to Spitzer \cite{Spitzer1970} (see also
Liggett \cite[Ch.~8, Thm.~1.1]{Liggett1985});
the stochastic six-vertex model, after Corwin--Petrov \cite{CorwinPetrov2015}
and Lin \cite{YierLin2019_6v_duality}; $q$-TASEP \cite{BorodinCorwinSasamoto2012}; the higher-spin vertex
models, after Corwin--Petrov \cite{CorwinPetrov2015}; and their dynamic deformations)
admits dualities traceable to the same quantum group structure.
Namely, $D$
can be built as a
matrix element of the universal $R$-matrix: a row of vertex weights for
an auxiliary line adjoined to the lattice, carried out for the dynamic higher-spin models
via the Drinfeld twister of $U_q(\mathfrak{sl}_2)$ in Kuan--Zhou \cite{KuanZhou2023}.

\medskip

Two regimes should be distinguished. \emph{Triangular} dualities build $D$ from raising operators on
the reversible measure, as in \cite{schutz1997dualityASEP},
Giardin\`a--Kurchan--Redig--Vafayi \cite{giardina2009duality},
Belitsky--Sch\"utz \cite{belitsky2016self}, and Kuan \cite{Kuan2018}; for ASEP and $q$-TASEP
they collapse a $k$-point $q$-moment of the height function to a closed evolution
of $k$ dual particles, solved by nested contour integrals (the engine behind the
exact one-point laws). They return moments rather than the law, so
recovering a distribution still needs a moment problem, ill-posed in the polymer limits.
\emph{Orthogonal-polynomial} dualities instead take $D$ orthogonal with
respect to the reversible measure, as in Groenevelt \cite{Groenevelt2019},
Carinci--Franceschini--Giardin\`a--Groenevelt--Redig
\cite{CarinciFranceschiniGiardinaGroeneveltRedig2019},
Carinci--Franceschini--Groenevelt \cite{CarinciFranceschiniGroenevelt2021},
Zhou \cite{Zhou2021}, and Franceschini--Kuan--Zhou \cite{Franceschini2024};
the orthogonality gives direct probabilistic access --- Boltzmann--Gibbs principles, after
Ayala--Carinci--Redig
\cite{AyalaCarinciRedig2018}, \cite{AyalaCarinciRedig2021}, and correlation identities
\cite{CarinciFranceschiniGroenevelt2021} --- without
routing through moments.

\bibliographystyle{alpha}
\bibliography{bib}

\medskip

\textsc{l. petrov, university of virginia, charlottesville, va, usa}

e-mail: \texttt{lenia.petrov@gmail.com}

\end{document}